\documentclass{amsart}
\usepackage{amssymb}
\usepackage{amsfonts}
\usepackage{geometry}
\usepackage{color}

\newtheorem{theorem}{Theorem}
\theoremstyle{plain}
\newtheorem{acknowledgement}[theorem]{Acknowledgement}

\newtheorem{corollary}[theorem]{Corollary}

\newtheorem{definition}[theorem]{Definition}

\newtheorem{lemma}[theorem]{Lemma}
\newtheorem{notation}[theorem]{Notation}

\newtheorem{proposition}[theorem]{Proposition}
\newtheorem{remark}[theorem]{Remark}

\numberwithin{equation}{section}
\input{tcilatex}
\begin{document}
\title[Two weight $L^{p}$ inequalities]{Two weight $L^{p}$ inequalities for $%
\lambda $-fractional vector Riesz transforms and doubling measures}
\author[E. T. Sawyer]{Eric T. Sawyer$^\dagger$}
\address{Eric T. Sawyer, Department of Mathematics and Statistics\\
McMaster University\\
1280 Main Street West\\
Hamilton, Ontario L8S 4K1 Canada}
\thanks{$\dagger $ Research supported in part by a grant from the National
Science and Engineering Research Council of Canada.}
\email{sawyer@mcmaster.ca}
\author[B. D. Wick]{Brett D. Wick$^\ddagger$}
\address{Brett D. Wick, Department of Mathematics \& Statistics, Washington
University -- St. Louis, One Brookings Drive, St. Louis, MO USA 63130-4899.}
\email{wick@math.wustl.edu}
\thanks{$\ddagger $ B. D. Wick's research is supported in part by National
Science Foundation Grants DMS \# 1800057, \# 2054863, and \# 2000510 and
Australian Research Council -- DP 220100285.}
\date{\today }

\begin{abstract}
If $\mathbf{R}^{\lambda }$ denotes the $\lambda $-fractional vector Riesz
transform on $\mathbb{R}^{n}$, $1<p<\infty $, and $\left( \sigma ,\omega
\right) $ is a pair of doubling measures, then the two weight $L^{p}$ norm
inequality,%
\begin{equation*}
\int_{\mathbb{R}^{n}}\left\vert \mathbf{R}^{\lambda }\left( f\sigma \right)
\right\vert ^{p}d\omega \leq \mathfrak{N}_{T^{\lambda },p}^{p}\int_{\mathbb{R%
}^{n}}\left\vert f\right\vert ^{p}d\sigma ,\ \ \ \ \ f\in L^{p}\left( \sigma
\right)
\end{equation*}%
holds \emph{if and only if} the following quadratic triple testing
conditions of Hyt\"{o}nen and Vuorinen hold,%
\begin{eqnarray*}
\int_{\mathbb{R}^{n}}\left( \sum_{j=1}^{\infty }\left( a_{j}\mathbf{1}%
_{3I_{j}}^{\lambda }\mathbf{R}^{\lambda }\left( \mathbf{1}_{I_{j}}\sigma
\right) \right) ^{2}\right) ^{\frac{p}{2}}d\omega &\leq &\left( \mathfrak{T}%
_{\mathbf{R}^{\lambda },p}^{\ell ^{2},\limfunc{triple}}\right) ^{p}\int_{%
\mathbb{R}^{n}}\left( \sum_{j=1}^{\infty }\left( a_{j}\mathbf{1}%
_{I_{j}}\right) ^{2}\right) ^{\frac{p}{2}}d\sigma , \\
\int_{\mathbb{R}^{n}}\left( \sum_{j=1}^{\infty }\left( a_{j}\mathbf{1}%
_{3I_{j}}^{\lambda }\mathbf{R}^{\lambda }\left( \mathbf{1}_{I_{j}}\omega
\right) \right) ^{2}\right) ^{\frac{p^{\prime }}{2}}d\sigma &\leq &\left( 
\mathfrak{T}_{\mathbf{R}^{\lambda ,\ast },p^{\prime }}^{\ell ^{2},\limfunc{%
triple}}\right) ^{p^{\prime }}\int_{\mathbb{R}^{n}}\left( \sum_{j=1}^{\infty
}\left( a_{j}\mathbf{1}_{I_{j}}\right) ^{2}\right) ^{\frac{p^{\prime }}{2}%
}d\omega ,
\end{eqnarray*}%
where the inequalities are taken over all sequences $\left\{ I_{j}\right\}
_{j=1}^{\infty }$ and $\left\{ a_{j}\right\} _{j=1}^{\infty }$ of cubes and
real numbers respectively. We also show that these quadratic triple testing
conditions can be relaxed to local quadratic testing conditions, quadratic
offset Muckenhoupt conditions, and a quadratic weak boundedness property.
\end{abstract}

\maketitle
\tableofcontents

\section{Introduction}

The Nazarov-Treil-Volberg $T1$ conjecture on the boundedness of the Hilbert
transform from one weighted space $L^{2}\left( \sigma \right) $ to another $%
L^{2}\left( \omega \right) $, was settled affirmatively in the two part
paper \cite{LaSaShUr3},\cite{Lac} when the measures have no common point
masses, and this restriction was removed by Hyt\"{o}nen in \cite{Hyt}. Since
then there have been a number of generalizations of boundedness of Calder%
\'{o}n-Zygmund operators from one weighted $L^{2}$ space to another,
including

\begin{itemize}
\item to higher dimensional Euclidean spaces (see e.g. \cite{SaShUr7}, \cite%
{LaWi} and \cite{LaSaShUrWi}),

\item to spaces of homogeneous type (see e.g. \cite{DuLiSaVeWiYa}), and

\item and to partial results for the case when both measures are doubling
(see \cite{AlSaUr}).
\end{itemize}

It had been known from work of Neugebauer \cite{Neu} and Coifman and
Fefferman \cite{CoFe} some time ago that in the case of $A_{\infty }$
weights, the two weight norm inequality for a Calder\'{o}n-Zygmund operator
was implied by the classical two weight $A_{p}$ condition; see \cite{AlSaUr2}
for the elementary proof when $p=2$, and \cite{HyLa} for a sharp estimate on
the characteristics. In addition there have been some generalizations to
Sobolev spaces in place of\thinspace $L^{2}$ spaces in the setting of a
single weight (see e.g. \cite{DilWiWi} and \cite{KaLiPeWa}).

The purpose of this paper is to prove a \emph{two weight} $T1$ theorem for $%
\lambda $-fractional vector \emph{Riesz} transforms on weighted\emph{\ }$%
L^{p}\left( \mathbb{R}^{n}\right) $ spaces with $1<p<\infty $, in the
special case when the measures are\ both doubling. In view of the $L^{2}$
result in \cite{LaSaShUr3},\cite{Lac} one might suspect that the Hilbert
transform $H$ is bounded from $L^{p}\left( \sigma \right) $ to $L^{p}\left(
\omega \right) $ with general locally finite positive Borel measures $\sigma 
$ and $\omega $ if and only if the local testing conditions for $H$,%
\begin{equation*}
\int_{I}\left\vert H\mathbf{1}_{I}\sigma \right\vert ^{p}d\omega \lesssim
\left\vert I\right\vert _{\sigma }\text{ and }\int_{I}\left\vert H\mathbf{1}%
_{I}\omega \right\vert ^{p}d\sigma \lesssim \left\vert I\right\vert _{\omega
},
\end{equation*}%
both hold, along with the tailed Muckenhoupt $\mathcal{A}_{p}$ conditions,%
\begin{equation*}
\left( \int_{I}\frac{\left\vert I\right\vert }{\left[ \left\vert
I\right\vert +\limfunc{dist}\left( x,I\right) \right] ^{p}}d\omega \right) ^{%
\frac{1}{p}}\left( \frac{\left\vert I\right\vert _{\sigma }}{\left\vert
I\right\vert }\right) ^{\frac{1}{p^{\prime }}}\lesssim 1\text{ and }\left( 
\frac{\left\vert I\right\vert _{\omega }}{\left\vert I\right\vert }\right) ^{%
\frac{1}{p}}\left( \int_{I}\frac{\left\vert I\right\vert }{\left[ \left\vert
I\right\vert +\limfunc{dist}\left( x,I\right) \right] ^{p^{\prime }}}d\sigma
\right) ^{\frac{1}{p^{\prime }}}\lesssim 1.
\end{equation*}%
In fact this conjecture was already made in \cite[see\ Conjecture 1.8]%
{LaSaUr1}, where the case of maximal singular integrals was treated when one
of the measures was doubling, but with more complicated testing conditions.
However, this conjecture fails for the Hilbert transform \cite{AlLuSaUr},
and even for pairs of doubling measures and Riesz transforms (including the
Hilbert transform) \cite{AlLuSaUr2}.

Another stronger conjecture, but difficult nonetheless, has been put forward
by Hyt\"{o}nen and Vuorinen \cite[pages 16-18]{HyVu}, see also \cite{Vuo}
and \cite{Vuo2}. Namely, that $H$ is bounded from $L^{p}\left( \sigma
\right) $ to $L^{p}\left( \omega \right) $ if and only if certain \emph{%
quadratic} interval testing conditions for $H$ hold, along with
corresponding \emph{quadratic} Muckenhoupt conditions and a \emph{quadratic}
weak boundedness property. Here `quadratic' refers to $\ell ^{2}$-valued
extensions of the familiar scalar conditions. More generally, these
quadratic conditions can be formulated for fractional singular integrals $%
T^{\lambda }$ in higher dimensions in a straightforward way.

We emphasize that our doubling assumptions are in part offset by the fact
that we characterize boundedness for \textbf{all} vector fractional Riesz
transform operators, and in part due to the fact that we have obtained a two
weight $T1$ theorem for$\ p\neq 2$ (for the first time). If one considers a
matrix of Calder\'{o}n-Zygmund operators and weight pairs such as,

\begin{center}
\frame{$%
\begin{array}{cccccc}
& T=Hilbert & T=Cauchy & T=Beurling & T=Riesz & T=General \\ 
\sigma ,\omega \in A_{p} & \ast & \ast & \ast & \ast & \ast \\ 
\sigma ,\omega \in A_{\infty } & \ast & \ast & \ast & \ast & \ast \\ 
\sigma ,\omega \in doubling & \ast & \ast & \ast & \limfunc{proved}\text{
here}\ \text{for }1<p<\infty & ? \\ 
\sigma ,\omega \in Borel & \limfunc{known}\ \text{only for }p=2 & ? & ? & ?
& ?%
\end{array}%
$},
\end{center}

two features stand out,

\begin{enumerate}
\item for general (locally finite positive) Borel measures, a two weight $T1$
characterization for $1<p<\infty $ has been found in this matrix \textbf{%
only }for the Hilbert transform when $p=2$,

\item for fractional Riesz transform operators, a two weight $T1$
characterization for $1<p<\infty $ has been found in this matrix \textbf{only%
} for pairs of doubling measures.
\end{enumerate}

The starred entries in the matrix correspond to $T1$ characterizations that
hold by virtue of the $\limfunc{known}$ results, and the question mark
entries remain unknown for any $1<p<\infty $ at this time. Of course there
are other geometric restrictions on the measures that give rise to a $T1$
theorem, and these can be found in the references at the end of this paper.
On the other hand, it appears quite challenging to find a natural class of
measures $\mathcal{M}$, more general than doubling measures, for which a $T1$
theorem can be obtained for all $1<p<\infty $, all fractional Riesz
transforms, and all measure pairs in $\mathcal{M}\times \mathcal{M}$.

\begin{acknowledgement}
We thank the referee for a very close reading of the manuscript and many
helpful comments, and in particular for pointing to a serious error in our
treatment of the stopping form, whose fix resulted in the restriction to
fractional Riesz transforms instead of the larger class of smooth Calder\'{o}%
n-Zygmund operators.
\end{acknowledgement}

\subsection{Quadratic conditions of Hyt\"{o}nen and Vuorinen}

For a $\lambda $-fractional singular integral operator $T^{\lambda }$ on $%
\mathbb{R}^{n}$, and locally finite positive Borel measures $\sigma $ and $%
\omega $, let $T_{\sigma }^{\lambda }f=T^{\lambda }\left( fd\sigma \right) $
and $T_{\omega }^{\lambda ,\ast }g=T^{\lambda ,\ast }\left( gd\omega \right) 
$ (see below for definitions). The \emph{quadratic} cube testing conditions
of Hyt\"{o}nen and Vuorinen are%
\begin{eqnarray}
\left\Vert \left( \sum_{i=1}^{\infty }\left\vert a_{i}\mathbf{1}%
_{I_{i}}T_{\sigma }^{\lambda }\mathbf{1}_{I_{i}}\right\vert ^{2}\right) ^{%
\frac{1}{2}}\right\Vert _{L^{p}\left( \omega \right) } &\leq &\mathfrak{T}%
_{T^{\lambda },p}^{\ell ^{2},\func{loc}}\left( \sigma ,\omega \right)
\left\Vert \left( \sum_{i=1}^{\infty }\left\vert a_{i}\mathbf{1}%
_{I_{i}}\right\vert ^{2}\right) ^{\frac{1}{2}}\right\Vert _{L^{p}\left(
\sigma \right) },  \label{quad HV} \\
\left\Vert \left( \sum_{i=1}^{\infty }\left\vert a_{i}\mathbf{1}%
_{I_{i}}T_{\omega }^{\lambda ,\ast }\mathbf{1}_{I_{i}}\right\vert
^{2}\right) ^{\frac{1}{2}}\right\Vert _{L^{p^{\prime }}\left( \sigma \right)
} &\leq &\mathfrak{T}_{T^{\lambda ,\ast },p^{\prime }}^{\ell ^{2},\func{loc}%
}\left( \omega ,\sigma \right) \left\Vert \left( \sum_{i=1}^{\infty
}\left\vert a_{i}\mathbf{1}_{I_{i}}\right\vert ^{2}\right) ^{\frac{1}{2}%
}\right\Vert _{L^{p^{\prime }}\left( \omega \right) },  \notag
\end{eqnarray}%
taken over all sequences $\left\{ I_{i}\right\} _{i=1}^{\infty }$ and $%
\left\{ a_{i}\right\} _{i=1}^{\infty }$ of cubes and numbers respectively.
The corresponding quadratic \emph{global} cube testing constants $\mathfrak{T%
}_{T^{\lambda },p}^{\ell ^{2},\func{global}}\left( \sigma ,\omega \right) $
and $\mathfrak{T}_{T^{\lambda ,\ast },p^{\prime }}^{\ell ^{2},\func{global}%
}\left( \omega ,\sigma \right) $ are defined as in (\ref{quad HV}), but 
\emph{without} the indicator $\mathbf{1}_{I_{i}}$ outside the operator,
namely with $\mathbf{1}_{I_{i}}T_{\sigma }^{\lambda }\mathbf{1}_{I_{i}}$
replaced by $T_{\sigma }^{\lambda }\mathbf{1}_{I_{i}}$. The \emph{quadratic}
Muckenhoupt conditions of Hyt\"{o}nen and Vuorinen are%
\begin{eqnarray}
\left\Vert \left( \sum_{i=1}^{\infty }\left\vert \int_{\mathbb{R}%
^{n}\setminus I_{i}}\frac{f_{i}\left( y\right) }{\left\vert
y-c_{i}\right\vert ^{n-\lambda }}d\sigma \left( y\right) \right\vert ^{2}%
\mathbf{1}_{I_{i}}\right) ^{\frac{1}{2}}\right\Vert _{L^{p}\left( \omega
\right) } &\leq &\mathcal{A}_{p}^{\lambda ,\ell ^{2}}\left( \sigma ,\omega
\right) \left\Vert \left( \sum_{i=1}^{\infty }\left\vert f_{i}\right\vert
^{2}\right) ^{\frac{1}{2}}\right\Vert _{L^{p}\left( \sigma \right) },
\label{quad Muck} \\
\left\Vert \left( \sum_{i=1}^{\infty }\left\vert \int_{\mathbb{R}%
^{n}\setminus I_{i}}\frac{f_{i}\left( y\right) }{\left\vert
y-c_{i}\right\vert ^{n-\lambda }}d\omega \left( y\right) \right\vert ^{2}%
\mathbf{1}_{I_{i}}\right) ^{\frac{1}{2}}\right\Vert _{L^{p^{\prime }}\left(
\sigma \right) } &\leq &\mathcal{A}_{p^{\prime }}^{\lambda ,\ell ^{2}}\left(
\omega ,\sigma \right) \left\Vert \left( \sum_{i=1}^{\infty }\left\vert
f_{i}\right\vert ^{2}\right) ^{\frac{1}{2}}\right\Vert _{L^{p^{\prime
}}\left( \omega \right) },  \notag
\end{eqnarray}%
taken over all sequences $\left\{ I_{i}\right\} _{i=1}^{\infty }$ and $%
\left\{ f_{i}\right\} _{i=1}^{\infty }$ of cubes and functions respectively.
Note that $\mathcal{A}_{p}^{\lambda ,\ell ^{2}}\left( \sigma ,\omega \right) 
$ is homogeneous of degree $1$ in the measure pair $\left( \sigma ,\omega
\right) $, as opposed to the usual formulation with degree $2$. Finally, the 
\emph{quadratic} weak boundedness property of Hyt\"{o}nen and Vuorinen (not
so named in \cite{HyVu}) is%
\begin{eqnarray}
&&\sum_{i=1}^{\infty }\left\vert \int_{\mathbb{R}^{n}}a_{i}T_{\sigma
}^{\lambda }\mathbf{1}_{I_{i}}\left( x\right) b_{i}\mathbf{1}_{J\left(
I_{i}\right) }\left( x\right) d\omega \left( x\right) \right\vert
\label{WBP HV} \\
&\leq &\mathcal{WBP}_{T^{\lambda },p}^{\ell ^{2}}\left( \sigma ,\omega
\right) \left\Vert \left( \sum_{i=1}^{\infty }\left\vert a_{i}\mathbf{1}%
_{I_{i}}\right\vert ^{2}\right) ^{\frac{1}{2}}\right\Vert _{L^{p}\left(
\sigma \right) }\left\Vert \left( \sum_{i=1}^{\infty }\left\vert b_{i}%
\mathbf{1}_{I_{i}}\right\vert ^{2}\right) ^{\frac{1}{2}}\right\Vert
_{L^{p^{\prime }}\left( \omega \right) },  \notag
\end{eqnarray}%
taken over all sequences $\left\{ I_{i}\right\} _{i=1}^{\infty }$, $\left\{
J\left( I_{i}\right) \right\} _{i=1}^{\infty }$, $\left\{ a_{i}\right\}
_{i=1}^{\infty }$ and $\left\{ b_{i}\right\} _{i=1}^{\infty }$ of cubes and
numbers respectively where $J\left( I_{i}\right) $ denotes any cube \emph{%
adjacent} to $I_{i}$ with the same side length.

If the Calder\'{o}n-Zygmund operator $T^{\lambda }$ is bounded from $%
L^{p}\left( \sigma \right) $ to $L^{p}\left( \omega \right) $, then the
Hilbert space valued extension $\left( T^{\lambda }\right) ^{\ell ^{2}}$ is
bounded from $L^{p}\left( \sigma ;\ell ^{2}\right) $ to $L^{p}\left( \omega
;\ell ^{2}\right) $, and it is now not hard to see that 
\begin{eqnarray*}
&&\mathfrak{T}_{T^{\lambda },p}^{\ell ^{2},\func{global}}\left( \sigma
,\omega \right) +\mathfrak{T}_{T^{\lambda },p^{\prime }}^{\ell ^{2},\func{%
global}}\left( \omega ,\sigma \right) +\mathcal{A}_{p}^{\lambda ,\ell
^{2}}\left( \sigma ,\omega \right) +\mathcal{A}_{p^{\prime }}^{\lambda ,\ell
^{2}}\left( \omega ,\sigma \right) \\
&&\ \ \ \ \ \ \ \ \ \ \ \ \ \ \ \ \ \ \ +\mathcal{WBP}_{T^{\lambda
},p}^{\ell ^{2}}\left( \sigma ,\omega \right) \lesssim \mathfrak{N}%
_{T^{\lambda },p}\left( \sigma ,\omega \right) ,
\end{eqnarray*}%
where $\mathfrak{N}_{T^{\lambda },p}\left( \sigma ,\omega \right) $ denotes
the operator norm of $H$ from $L^{p}\left( \sigma \right) $ to $L^{p}\left(
\omega \right) $, more generally see (\ref{two weight'}) below.

For now the conjecture of Hyt\"{o}nen and Vuorinen for the Hilbert transform
also remains open, but we settle here in the affirmative the boundedness
question for the Hilbert transform $H$, and more generally for $\lambda $%
-fractional vector Riesz transforms $\mathbf{R}^{\lambda }$ on $\mathbb{R}%
^{n}$, in the case that the measures $\sigma $ and $\omega $ are both \emph{%
doubling}. Moreover, we use certain `logically weaker' quadratic conditions,
which we now describe in the general setting of $\lambda $-fractional Calder%
\'{o}n-Zygmund operators $T^{\lambda }$.

\subsection{Weaker quadratic conditions for doubling measures}

First, we will use local scalar testing conditions,%
\begin{eqnarray}
\left\Vert \mathbf{1}_{I}T_{\sigma }^{\lambda }\mathbf{1}_{I}\right\Vert
_{L^{p}\left( \omega \right) } &\leq &\mathfrak{T}_{T^{\lambda },p}\left(
\sigma ,\omega \right) \left\vert I\right\vert _{\sigma }^{\frac{1}{p}},
\label{test} \\
\left\Vert \mathbf{1}_{I}T_{\omega }^{\lambda ,\ast }\mathbf{1}%
_{I}\right\Vert _{L^{p^{\prime }}\left( \sigma \right) } &\leq &\mathfrak{T}%
_{T^{\lambda ,\ast },p^{\prime }}\left( \omega ,\sigma \right) \left\vert
I\right\vert _{\omega }^{\frac{1}{p^{\prime }}},  \notag
\end{eqnarray}%
which do not involve any vector-valued extensions.

Second, we will typically use $\ell ^{2}$ in a superscript instead of $%
\limfunc{quad}$ to indicate a `quadratic' constant, and we will use
quadratic \emph{offset} Muckenhoupt conditions given by%
\begin{eqnarray}
\left\Vert \left( \sum_{i=1}^{\infty }\left\vert a_{i}\frac{%
\min_{I_{i}^{\ast }}\left\vert I_{i}^{\ast }\right\vert _{\sigma }}{%
\left\vert I_{i}\right\vert ^{1-\frac{\lambda }{n}}}\right\vert ^{2}\mathbf{1%
}_{I_{i}}\right) ^{\frac{1}{2}}\right\Vert _{L^{p}\left( \omega \right) }
&\leq &A_{p}^{\lambda ,\ell ^{2},\limfunc{offset}}\left( \sigma ,\omega
\right) \left\Vert \left( \sum_{i=1}^{\infty }\left\vert a_{i}\right\vert
^{2}\mathbf{1}_{I_{i}}\right) ^{\frac{1}{2}}\right\Vert _{L^{p}\left( \sigma
\right) },  \label{quad A2 tailless} \\
\left\Vert \left( \sum_{i=1}^{\infty }\left\vert a_{i}\frac{%
\min_{I_{i}^{\ast }}\left\vert I_{i}^{\ast }\right\vert _{\omega }}{%
\left\vert I_{i}\right\vert ^{1-\frac{\lambda }{n}}}\right\vert ^{2}\mathbf{1%
}_{I_{i}}\right) ^{\frac{1}{2}}\right\Vert _{L^{p^{\prime }}\left( \sigma
\right) } &\leq &A_{p^{\prime }}^{\lambda ,\ell ^{2},\limfunc{offset}}\left(
\omega ,\sigma \right) \left\Vert \left( \sum_{i=1}^{\infty }\left\vert
a_{i}\right\vert ^{2}\mathbf{1}_{I_{i}}\right) ^{\frac{1}{2}}\right\Vert
_{L^{p^{\prime }}\left( \omega \right) },  \notag
\end{eqnarray}%
where for each $i$, the minimums are taken over the finitely many dyadic
cubes $I_{i}^{\ast }$ such that $\ell \left( I_{i}^{\ast }\right) =\ell
\left( I_{i}\right) $ and $\limfunc{dist}\left( I_{i}^{\ast },I_{i}\right)
\leq C_{0}\ell \left( I_{i}\right) $ for some positive constant $C_{0}$%
\footnote{%
In applications one takes $C_{0}$ sufficiently large depending on the Stein
elliptic constant for the operator $T^{\lambda }$. But if $\sigma $ is
doubling the condition doesn't depend on $C_{0}$.}. Of course, when the
measures are doubling, we may take $I_{i}^{\ast }=I_{i}$ so that (\ref{quad
A2 tailless}) is equivalent to the following condition of Vuorinen \cite%
{Vuo2} that was introduced in the context of dyadic shifts, 
\begin{eqnarray}
\left\Vert \left( \sum_{i=1}^{\infty }\left\vert a_{i}\frac{\left\vert
I_{i}\right\vert _{\sigma }}{\left\vert I_{i}\right\vert ^{1-\frac{\lambda }{%
n}}}\right\vert ^{2}\mathbf{1}_{I_{i}}\right) ^{\frac{1}{2}}\right\Vert
_{L^{p}\left( \omega \right) } &\lesssim &A_{p}^{\lambda ,\ell ^{2},\limfunc{%
offset}}\left( \sigma ,\omega \right) \left\Vert \left( \sum_{i=1}^{\infty
}\left\vert a_{i}\right\vert ^{2}\mathbf{1}_{I_{i}}\right) ^{\frac{1}{2}%
}\right\Vert _{L^{p}\left( \sigma \right) },  \label{quad A2 tailless'} \\
\left\Vert \left( \sum_{i=1}^{\infty }\left\vert a_{i}\frac{\left\vert
I_{i}\right\vert _{\omega }}{\left\vert I_{i}\right\vert ^{1-\frac{\lambda }{%
n}}}\right\vert ^{2}\mathbf{1}_{I_{i}}\right) ^{\frac{1}{2}}\right\Vert
_{L^{p^{\prime }}\left( \sigma \right) } &\lesssim &A_{p^{\prime }}^{\lambda
,\ell ^{2},\limfunc{offset}}\left( \omega ,\sigma \right) \left\Vert \left(
\sum_{i=1}^{\infty }\left\vert a_{i}\right\vert ^{2}\mathbf{1}%
_{I_{i}}\right) ^{\frac{1}{2}}\right\Vert _{L^{p^{\prime }}\left( \omega
\right) }.  \notag
\end{eqnarray}%
We prove below that the offset constants $A_{p}^{\lambda ,\ell ^{2},\limfunc{%
offset}}\left( \sigma ,\omega \right) $ in (\ref{quad A2 tailless}) are
necessary for the norm inequality $\left\Vert T_{\sigma }^{\lambda
}f\right\Vert _{L^{p}\left( \omega \right) }\leq \mathfrak{N}_{T^{\lambda
}}\left( \sigma ,\omega \right) \left\Vert f\right\Vert _{L^{p}\left( \sigma
\right) }$ when $\sigma $ and $\omega $ are doubling. Here we simply note
that using the Fefferman-Stein vector-valued inequality for the maximal
function $M_{\sigma }$ on a space of homogeneous type $\left( \mathbb{R}%
^{n},\left\vert \cdot \right\vert ,\sigma \right) $ \cite{GrLiYa}, we see
that $A_{p}^{\lambda ,\ell ^{2},\limfunc{offset}}\left( \sigma ,\omega
\right) $ is smaller than $\mathcal{A}_{p}^{\lambda ,\ell ^{2},\limfunc{quad}%
}\left( \sigma ,\omega \right) $ for doubling measures because%
\begin{equation*}
\frac{\left\vert I_{i}^{\ast }\right\vert _{\sigma }}{\left\vert
I_{i}\right\vert ^{1-\frac{\lambda }{n}}}\lesssim \int_{\mathbb{R}%
^{n}\setminus I_{i}}\frac{M_{\sigma }\mathbf{1}_{I_{i}^{\ast }}\left(
y\right) }{\left\vert y-c_{i}\right\vert ^{n-\lambda }}d\sigma \left(
y\right) ,\ \ \ \ \ \text{when }I_{i}^{\ast }\cap I_{i}=\emptyset .
\end{equation*}%
Such use of the Fefferman-Stein vector-valued inequality occurs frequently
in the sequel. Note again that $A_{p}^{\lambda ,\ell ^{2},\limfunc{offset}%
}\left( \sigma ,\omega \right) $ is homogeneous of degree $1$ in the measure
pair $\left( \sigma ,\omega \right) $.

Third, we use a variant of the weak boundedness property (\ref{WBP HV}) of
Hyt\"{o}nen and Vuorinen given by%
\begin{eqnarray}
&&\sum_{i=1}^{\infty }\sum_{I_{i}^{\ast }\in \func{Adj}\left( I_{i}\right)
}\left\vert \int_{\mathbb{R}^{n}}a_{i}T_{\sigma }^{\lambda }\mathbf{1}%
_{I_{i}}\left( x\right) b_{i}^{\ast }\mathbf{1}_{I_{i}^{\ast }}\left(
x\right) d\omega \left( x\right) \right\vert  \label{WBP} \\
&\leq &\mathcal{WBP}_{T^{\lambda },p}^{\ell ^{2}}\left( \sigma ,\omega
\right) \left\Vert \left( \sum_{i=1}^{\infty }\left\vert a_{i}\mathbf{1}%
_{I_{i}}\right\vert ^{2}\right) ^{\frac{1}{2}}\right\Vert _{L^{p}\left(
\sigma \right) }\left\Vert \left( \sum_{i=1}^{\infty }\sum_{I_{i}^{\ast }\in 
\func{Adj}\left( I_{i}\right) }\left\vert b_{i}^{\ast }\mathbf{1}%
_{I_{i}^{\ast }}\right\vert ^{2}\right) ^{\frac{1}{2}}\right\Vert
_{L^{p^{\prime }}\left( \omega \right) },  \notag
\end{eqnarray}%
where for $I\in \mathcal{D}$, its \emph{adjacent} cubes are defined by 
\begin{equation}
\func{Adj}\left( I\right) \equiv \left\{ I^{\ast }\in \mathcal{D}:\overline{%
I^{\ast }}\cap \overline{I}\neq \emptyset \text{ and }\ell \left( I^{\ast
}\right) =\ell \left( I\right) \right\} ,  \label{def adj}
\end{equation}%
and in particular include $I$ itself.

Finally, we also define the stronger quadratic \emph{triple} testing
constants by%
\begin{eqnarray}
\left\Vert \left( \sum_{i=1}^{\infty }\left( a_{i}\mathbf{1}%
_{3I_{i}}T_{\sigma }^{\lambda }\mathbf{1}_{I_{i}}\right) ^{2}\right) ^{\frac{%
1}{2}}\right\Vert _{L^{p}\left( \omega \right) } &\leq &\mathfrak{T}%
_{T^{\lambda },p}^{\ell ^{2},\limfunc{triple}}\left( \sigma ,\omega \right)
\left\Vert \left( \sum_{i=1}^{\infty }\left( a_{i}\mathbf{1}_{I_{i}}\right)
^{2}\right) ^{\frac{1}{2}}\right\Vert _{L^{p}\left( \sigma \right) },
\label{quad triple test} \\
\left\Vert \left( \sum_{i=1}^{\infty }\left( a_{i}\mathbf{1}%
_{3I_{i}}T_{\omega }^{\lambda ,\ast }\mathbf{1}_{I_{i}}\right) ^{2}\right) ^{%
\frac{1}{2}}\right\Vert _{L^{p^{\prime }}\left( \sigma \right) } &\leq &%
\mathfrak{T}_{T^{\lambda ,\ast },p^{\prime }}^{\ell ^{2},\limfunc{triple}%
}\left( \omega ,\sigma \right) \left\Vert \left( \sum_{i=1}^{\infty }\left(
a_{i}\mathbf{1}_{I_{i}}\right) ^{2}\right) ^{\frac{1}{2}}\right\Vert
_{L^{p^{\prime }}\left( \omega \right) }.  \notag
\end{eqnarray}

\subsection{Statement of the main theorem}

Our main theorem is restrict to $\lambda $-fractional vector Riesz
transforms $\mathbf{R}^{\lambda }$, but we will continue with general $%
\lambda $-fractional vector Calder\'{o}n-Zygmund operators $T^{\lambda }$ in
describing the setup.

Denote by $\Omega _{\limfunc{dyad}}$ the collection of all dyadic grids in $%
\mathbb{R}^{n}$, and let $\mathcal{Q}^{n}$ denote the collection of all
cubes in $\mathbb{R}^{n}$ having sides parallel to the coordinate axes. A
positive locally finite Borel measure $\mu $ on $\mathbb{R}^{n}$ is said to
be doubling if there is a constant $C_{\limfunc{doub}}$, called the doubling
constant, such that%
\begin{equation*}
\left\vert 2Q\right\vert _{\mu }\leq C_{\limfunc{doub}}\left\vert
Q\right\vert _{\mu }\ ,\ \ \ \ \ \text{for all cubes }Q\in \mathcal{Q}^{n}.
\end{equation*}

For $0\leq \lambda <n$ we define a smooth $\lambda $-fractional Calder\'{o}%
n-Zygmund kernel $K^{\lambda }(x,y)$ to be a function $K^{\lambda }:\mathbb{R%
}^{n}\times \mathbb{R}^{n}\rightarrow \mathbb{R}$ satisfying the following
fractional size and smoothness conditions%
\begin{equation}
\left\vert \nabla _{x}^{j}K^{\lambda }\left( x,y\right) \right\vert
+\left\vert \nabla _{y}^{j}K^{\lambda }\left( x,y\right) \right\vert \leq
C_{\lambda ,j}\left\vert x-y\right\vert ^{\lambda -j-n},\ \ \ \ \ 0\leq
j<\infty ,  \label{sizeandsmoothness'}
\end{equation}%
and we denote by $T^{\lambda }$ the associated $\lambda $-fractional
singular integral on $\mathbb{R}^{n}$. Following \cite[(39) on page 210]{Ste}
as in \cite{AlSaUr}, we say that a $\lambda $-fractional Calder\'{o}%
n-Zygmund kernel $K^{\lambda }$ is \emph{elliptic in the sense of Stein} if
there is a unit vector $\mathbf{u}_{0}\in \mathbb{R}^{n}$ and a constant $%
c>0 $ such that%
\begin{equation*}
\left\vert K^{\lambda }\left( x,x+t\mathbf{u}_{0}\right) \right\vert \geq
c\left\vert t\right\vert ^{\lambda -n},\ \ \ \ \ \text{for all }t\in \mathbb{%
R}.
\end{equation*}

\subsubsection{Defining the norm inequality\label{Subsubsection norm}}

As in \cite[see page 314]{SaShUr9}, we introduce a family $\left\{ \eta
_{\delta ,R}^{\lambda }\right\} _{0<\delta <R<\infty }$ of smooth
nonnegative functions on $\left[ 0,\infty \right) $ so that the truncated
kernels $K_{\delta ,R}^{\lambda }\left( x,y\right) =\eta _{\delta
,R}^{\lambda }\left( \left\vert x-y\right\vert \right) K^{\lambda }\left(
x,y\right) $ are bounded with compact support for fixed $x$ or $y$, and
uniformly satisfy (\ref{sizeandsmoothness'}). Then the truncated operators 
\begin{equation*}
T_{\sigma ,\delta ,R}^{\lambda }f\left( x\right) \equiv \int_{\mathbb{R}%
^{n}}K_{\delta ,R}^{\lambda }\left( x,y\right) f\left( y\right) d\sigma
\left( y\right) ,\ \ \ \ \ x\in \mathbb{R}^{n},
\end{equation*}%
are pointwise well-defined when $f$ is bounded with compact support, and we
will refer to the pair $\left( K^{\lambda },\left\{ \eta _{\delta
,R}^{\lambda }\right\} _{0<\delta <R<\infty }\right) $ as a $\lambda $%
-fractional singular integral operator, which we typically denote by $%
T^{\lambda }$, suppressing the dependence on the truncations. For $%
1<p<\infty $, we say that a $\lambda $-fractional singular integral operator 
$T^{\lambda }=\left( K^{\lambda },\left\{ \eta _{\delta ,R}^{\lambda
}\right\} _{0<\delta <R<\infty }\right) $ satisfies the norm inequality%
\begin{equation}
\left\Vert T_{\sigma }^{\lambda }f\right\Vert _{L^{p}\left( \omega \right)
}\leq \mathfrak{N}_{T^{\lambda }}\left( \sigma ,\omega \right) \left\Vert
f\right\Vert _{L^{p}\left( \sigma \right) },\ \ \ \ \ f\in L^{p}\left(
\sigma \right) ,  \label{two weight'}
\end{equation}%
where $\mathfrak{N}_{T^{\lambda }}\left( \sigma ,\omega \right) $ denotes
the best constant in (\ref{two weight'}), provided%
\begin{equation*}
\left\Vert T_{\sigma ,\delta ,R}^{\lambda }f\right\Vert _{L^{p}\left( \omega
\right) }\leq \mathfrak{N}_{T^{\lambda }}\left( \sigma ,\omega \right)
\left\Vert f\right\Vert _{L^{p}\left( \sigma \right) },\ \ \ \ \ f\in
L^{p}\left( \sigma \right) ,0<\delta <R<\infty .
\end{equation*}%
In the presence of the classical Muckenhoupt condition $A_{p}^{\alpha }$, it
can be easily shown that the norm inequality is independent of the choice of
truncations used - see e.g. \cite{LaSaShUr3} where rough operators are
treated in the case $p=2$, but the proofs can be modified. We can now state
our main theorem. Note that the second two parts of the theorem apply to
vector Riesz transforms only.

\begin{theorem}
\label{main}Suppose that $1<p<\infty $, that $\sigma $ and $\omega $ are
locally finite positive Borel measures on $\mathbb{R}^{n}$, and that $%
T^{\lambda }$ is a smooth $\lambda $-fractional singular integral operator
on $\mathbb{R}^{n}$. Denote by $\mathfrak{N}_{T^{\lambda },p}\left( \sigma
,\omega \right) $ the smallest constant $C$ in the two weight norm inequality%
\begin{equation}
\left\Vert T_{\sigma }^{\lambda }f\right\Vert _{L^{p}\left( \omega \right)
}\leq C\left\Vert f\right\Vert _{L^{p}\left( \sigma \right) }.
\label{2 wt norm}
\end{equation}

\begin{enumerate}
\item Then 
\begin{equation*}
\mathfrak{T}_{T^{\lambda },p}\left( \sigma ,\omega \right) +\mathfrak{T}%
_{T^{\lambda }{}^{,\ast },p^{\prime }}\left( \omega ,\sigma \right) +%
\mathcal{WBP}_{T^{\lambda \lambda },p}^{\ell ^{2}}\left( \sigma ,\omega
\right) \leq \mathfrak{T}_{T^{\lambda },p}^{\ell ^{2},\limfunc{triple}%
}\left( \sigma ,\omega \right) +\mathfrak{T}_{T^{\lambda }{}^{,\ast
},p^{\prime }}^{\ell ^{2},\limfunc{triple}}\left( \omega ,\sigma \right)
\leq \mathfrak{N}_{T^{\lambda },p}\left( \sigma ,\omega \right) ,
\end{equation*}%
and when $T^{\lambda }$ is Stein elliptic, we also have 
\begin{equation*}
A_{p}^{\lambda ,\ell ^{2},\limfunc{offset}}\left( \sigma ,\omega \right)
+A_{p^{\prime }}^{\lambda ,\ell ^{2},\limfunc{offset}}\left( \omega ,\sigma
\right) \lesssim \mathfrak{T}_{T^{\lambda },p}^{\ell ^{2},\limfunc{triple}%
}\left( \sigma ,\omega \right) +\mathfrak{T}_{T^{\lambda ,\ast },p^{\prime
}}^{\ell ^{2},\limfunc{triple}}\left( \omega ,\sigma \right) .
\end{equation*}

\item Suppose in addition that $\sigma $ and $\omega $ are doubling measures
on $\mathbb{R}^{n}$, and that $T^{\lambda }$ is replaced by the $\lambda $%
-fractional vector Riesz transform $\mathbf{R}^{\lambda }$ on $\mathbb{R}%
^{n} $. Then the two weight norm inequality (\ref{2 wt norm}) holds provided
the quadratic weak boundedness property (\ref{WBP}) holds, and the quadratic
local testing conditions (\ref{quad HV}) hold, and the quadratic offset
Muckenhoupt conditions (\ref{quad A2 tailless}) hold; and moreover in this
case we have 
\begin{eqnarray}
\mathfrak{N}_{\mathbf{R}^{\lambda },p}\left( \sigma ,\omega \right)
&\lesssim &\mathfrak{T}_{\mathbf{R}^{\lambda },p}^{\ell ^{2},\func{loc}%
}\left( \sigma ,\omega \right) +\mathfrak{T}_{\mathbf{R}^{\lambda ,\ast
},p^{\prime }}^{\ell ^{2},\func{loc}}\left( \omega ,\sigma \right) +\mathcal{%
WBP}_{\mathbf{R}^{\lambda },p}^{\ell ^{2}}\left( \sigma ,\omega \right)
\label{main inequ} \\
&&+A_{p}^{\lambda ,\ell ^{2},\limfunc{offset}}\left( \sigma ,\omega \right)
+A_{p^{\prime }}^{\lambda ,\ell ^{2},\limfunc{offset}}\left( \omega ,\sigma
\right) .  \notag
\end{eqnarray}

\item Suppose in addition that $\sigma $ and $\omega $ are doubling measures
on $\mathbb{R}^{n}$, and that $T^{\lambda }$ is replaced by the $\lambda $%
-fractional vector Riesz transform $\mathbf{R}^{\lambda }$ on $\mathbb{R}%
^{n} $. Then the two weight norm inequality (\ref{2 wt norm}) holds \emph{if
and only if} the quadratic \emph{triple} testing conditions (\ref{quad
triple test}) hold, and moreover, 
\begin{equation*}
\mathfrak{N}_{\mathbf{R}^{\lambda },p}\left( \sigma ,\omega \right) \approx 
\mathfrak{T}_{\mathbf{R}^{\lambda },p}^{\ell ^{2},\limfunc{triple}}\left(
\sigma ,\omega \right) +\mathfrak{T}_{\mathbf{R}^{\lambda ,\ast },p^{\prime
}}^{\ell ^{2},\limfunc{triple}}\left( \omega ,\sigma \right) .
\end{equation*}
\end{enumerate}
\end{theorem}

The constants on the right hand side of (\ref{main inequ}) represent the
most `elementary' constants we were able to find that characterize the norm
of the two weight inequality for Riesz transforms and doubling measures when 
$p\neq 2$.

\begin{remark}
In the case of equal measures $\sigma =\omega $, the quadratic $%
A_{p}^{\lambda ,\ell ^{2}}$ and $A_{p}^{\lambda ,\ell ^{2},\func{offest}}$
conditions trivally reduce to the scalar $A_{p}^{\lambda }$ and $%
A_{p}^{\lambda ,\func{offest}}$ conditions respectively. We show in the
appendix that $A_{p}^{\lambda ,\ell ^{2},\func{offest}}\left( \sigma ,\omega
\right) +A_{p^{\prime }}^{\lambda ,\ell ^{2},\func{offest}}\left( \omega
,\sigma \right) $ is \emph{not} controlled by $A_{p}^{\lambda }\left( \sigma
,\omega \right) $ in general, but the case of doubling measures remains
open. We also note that our weak boundedness property (\ref{WBP}) excludes
the case $I_{i}^{\ast }=I_{i}$. Finally, we note that our proof shows that
we can extend the theorem to include all smooth Stein elliptic Calder\'{o}%
n-Zygmund operators if we assume the classical pivotal condition.
\end{remark}

Part (3) is an easy corollary of parts (1) and (2). Indeed, it is trivial
that $\mathfrak{T}_{T^{\lambda },p}^{\ell ^{2},\limfunc{triple}}\left(
\sigma ,\omega \right) +\mathfrak{T}_{T^{\lambda }{}^{,\ast },p^{\prime
}}^{\ell ^{2},\limfunc{triple}}\left( \omega ,\sigma \right) \lesssim 
\mathfrak{N}_{T^{\lambda },p}\left( \sigma ,\omega \right) $, and a simple
exercise to see that for general measures,%
\begin{eqnarray*}
&&\mathfrak{T}_{T^{\lambda },p}\left( \sigma ,\omega \right) +\mathfrak{T}%
_{T^{\lambda ,\ast },p^{\prime }}\left( \omega ,\sigma \right) +\mathcal{WBP}%
_{T^{\lambda },p}^{\ell ^{2}}\left( \sigma ,\omega \right) \\
&&+A_{p}^{\lambda ,\ell ^{2},\limfunc{offset}}\left( \sigma ,\omega \right)
+A_{p^{\prime }}^{\lambda ,\ell ^{2},\limfunc{offset}}\left( \omega ,\sigma
\right) \lesssim \mathfrak{T}_{T^{\lambda },p}^{\ell ^{2},\limfunc{triple}%
}\left( \sigma ,\omega \right) +\mathfrak{T}_{T^{\lambda ,\ast },p^{\prime
}}^{\ell ^{2},\limfunc{triple}}\left( \omega ,\sigma \right) .
\end{eqnarray*}

\begin{notation}
In the interest of reducing notational clutter we will sometimes omit
specifying the measure pair and simply write $\mathfrak{T}_{T^{\lambda },p}$
and $A_{p}^{\lambda ,\ell ^{2},\limfunc{offset}}$ in place of $\mathfrak{T}%
_{T^{\lambda },p}\left( \sigma ,\omega \right) $ and $A_{p}^{\lambda ,\ell
^{2},\limfunc{offset}}\left( \sigma ,\omega \right) $ etc. especially when
in line.
\end{notation}

\section{Organization of the proof}

We follow the overall outline of an argument for the case $p=2$ given in 
\cite{AlSaUr}, but only for Haar wavelets which simplifies matters a bit,
but also with a number of adaptations to the use of square functions. The
proof of Theorem \ref{main} is achieved by proving the bilinear form bound,%
\begin{equation*}
\frac{\left\vert \left\langle \mathbf{R}_{\sigma }^{\lambda
}f,g\right\rangle _{\omega }\right\vert }{\left\Vert f\right\Vert
_{L^{p}\left( \sigma \right) }\left\Vert g\right\Vert _{L^{p^{\prime
}}\left( \omega \right) }}\lesssim \mathfrak{T}_{\mathbf{R}^{\lambda },p}+%
\mathfrak{T}_{\mathbf{R}^{\lambda ,\ast },p^{\prime }}+\mathcal{WBP}_{%
\mathbf{R}^{\lambda },p}^{\ell ^{2}}+A_{p}^{\lambda ,\ell ^{2},\limfunc{%
offset}}+A_{p^{\prime }}^{\lambda ,\ell ^{2},\limfunc{offset}},
\end{equation*}%
for $\func{good}$ functions $f$ and $g$ in the sense of Nazarov, Treil and
Volberg, see \cite{NTV} for the treatment we use here\footnote{%
See also \cite[Subsection 3.1]{SaShUr10} for a treatment using finite
collections of grids, in which case the conditional probability arguments
are elementary.}. Following the weighted Haar expansions as given by
Nazarov, Treil and Volberg in \cite{NTV4}, we write $f$ and $g$ in weighted
Alpert wavelet expansions,%
\begin{equation}
\left\langle \mathbf{R}_{\sigma }^{\lambda }f,g\right\rangle _{\omega
}=\left\langle \mathbf{R}_{\sigma }^{\lambda }\left( \sum_{I\in \mathcal{D}%
}\bigtriangleup _{I}^{\sigma }f\right) ,\left( \sum_{J\in \mathcal{D}%
}\bigtriangleup _{J}^{\omega }g\right) \right\rangle _{\omega }=\sum_{I\in 
\mathcal{D}\ \text{and }J\in \mathcal{D}}\left\langle \mathbf{R}_{\sigma
}^{\lambda }\left( \bigtriangleup _{I}^{\sigma }f\right) ,\left(
\bigtriangleup _{J}^{\omega }g\right) \right\rangle _{\omega }\ .
\label{expand}
\end{equation}%
The sum is further decomposed, as depicted in the brief schematic diagram
below, by first \emph{Cube Size Splitting}, then using the \emph{Shifted
Corona Decomposition}, according to the \emph{Canonical Splitting}. All of
these `descriptive' expressions will be defined as the proof proceeds.

Here is the brief schematic diagram as in \cite{AlSaUr}, summarizing the
shifted corona decompositions as used in \cite{AlSaUr} and \cite{SaShUr7}
for Alpert and Haar wavelet expansions of $f$ and $g$, and where $T^{\lambda
}$ is a smooth $\lambda $-fractional Calder\'{o}n-Zygmund operator in $%
\mathbb{R}^{n}$. The parameter $\rho $ is defined below.%
\begin{equation*}
\fbox{$%
\begin{array}{ccccccccc}
\left\langle T_{\sigma }^{\lambda }f,g\right\rangle _{\omega } &  &  &  &  & 
&  &  &  \\ 
\downarrow &  &  &  &  &  &  &  &  \\ 
\mathsf{B}_{\Subset _{\rho }}\left( f,g\right) & + & \mathsf{B}_{_{\rho
}\Supset }\left( f,g\right) & + & \mathsf{B}_{\cap }\left( f,g\right) & + & 
\mathsf{B}_{\diagup }\left( f,g\right) & + & \mathsf{B}_{\func{Adj},\rho
}\left( f,g\right) \\ 
\downarrow &  & \left[ \limfunc{duality}\right] &  & \left[ A_{p}^{\lambda
,\ell ^{2},\limfunc{offset}}\right] &  & \left[ A_{p}^{\lambda ,\ell ^{2},%
\limfunc{offset}}\right] &  & \left[ \mathcal{WBP}_{T^{\lambda },p}^{\ell
^{2}}\right] \\ 
\mathsf{T}_{\limfunc{diagonal}}\left( f,g\right) & + & \mathsf{T}_{\limfunc{%
far}\limfunc{below}}\left( f,g\right) & + & \mathsf{T}_{\limfunc{far}%
\limfunc{above}}\left( f,g\right) & + & \mathsf{T}_{\limfunc{disjoint}%
}\left( f,g\right) &  &  \\ 
\downarrow &  & \downarrow &  & \left[ =0\right] &  & \left[ =0\right] &  & 
\\ 
\mathsf{B}_{\Subset _{\mathbf{\rho }}}^{F}\left( f,g\right) &  & \mathsf{T}_{%
\limfunc{far}\limfunc{below}}^{1}\left( f,g\right) & + & \mathsf{T}_{%
\limfunc{far}\limfunc{below}}^{2}\left( f,g\right) &  &  &  &  \\ 
\downarrow &  & \left[ A_{p}^{\lambda ,\ell ^{2},\limfunc{offset}}\right] & 
& \left[ A_{p}^{\lambda ,\ell ^{2},\limfunc{offset}}\right] &  &  &  &  \\ 
\mathsf{B}_{\func{stop}}^{F}\left( f,g\right) & + & \mathsf{B}_{\func{%
paraproduct}}^{F}\left( f,g\right) & + & \mathsf{B}_{\func{neighbour}%
}^{F}\left( f,g\right) & + & \mathsf{B}_{\limfunc{commutator}}^{F}\left(
f,g\right) &  &  \\ 
\left[ A_{p}^{\lambda ,\ell ^{2},\limfunc{offset}}\right] &  & \left[ 
\mathfrak{T}_{T^{\lambda },p}\right] &  & \left[ A_{p}^{\lambda ,\ell ^{2},%
\limfunc{offset}}\right] &  & \left[ A_{p}^{\lambda ,\ell ^{2},\limfunc{%
offset}}\right] &  & 
\end{array}%
$}
\end{equation*}%
The condition that is used to control the indicated form is given in square
brackets directly underneath. Note that all forms are controlled solely by
the quadratic offset Muckenhoupt condition, save for the adjacent form which
uses only the weak boundedness property, and the paraproduct form which uses
only the scalar testing condition.

There are however notable exceptions in our treatment here as compared to
that in \cite{AlSaUr}. For example, we use only the classical Calder\'{o}%
n-Zygmund stopping time to bound the forms, and we control the stopping form
by Muckehoupt, Riesz testing and doubling conditions. We will bound the
remaining forms using only the fact that for a doubling measure $\mu $, the
Poisson averages reduce to ordinary averages in the presence of vector Riesz
testing and the Muckenhoupt condition. Indeed, the Poisson kernel of order $%
\lambda $ is given by 
\begin{equation}
\mathrm{P}^{\lambda }\left( Q,\mu \right) \equiv \int_{\mathbb{R}^{n}}\frac{%
\ell \left( Q\right) }{\left( \ell \left( Q\right) +\left\vert
y-c_{Q}\right\vert \right) ^{n+1-\lambda }}d\mu \left( y\right) ,
\label{def kappa Poisson}
\end{equation}%
and a doubling measure $\mu $ has a `doubling exponent' $\theta >0$ and a
positive constant $c$ that satisfy the condition,%
\begin{equation*}
\left\vert 2^{-j}Q\right\vert _{\mu }\geq c2^{-j\theta }\left\vert
Q\right\vert _{\mu }\ ,\ \ \ \ \ \text{for all }j\in \mathbb{N}.
\end{equation*}%
Thus if $\mu $ has doubling exponent $\theta $ and $\kappa >\theta +\lambda
-n$, we have%
\begin{eqnarray}
&&\mathrm{P}^{\lambda }\left( Q,\mu \right) =\int_{\mathbb{R}^{n}}\frac{\ell
\left( Q\right) }{\left( \ell \left( Q\right) +\left\vert x-c_{Q}\right\vert
\right) ^{n+1-\lambda }}d\mu \left( x\right)  \label{kappa large} \\
&=&\ell \left( Q\right) ^{\lambda -n}\left\{ \int_{Q}+\sum_{j=1}^{\infty
}\int_{2^{j}Q\setminus 2^{j-1}Q}\right\} \frac{1}{\left( 1+\frac{\left\vert
x-c_{Q}\right\vert }{\ell \left( Q\right) }\right) ^{n+1-\lambda }}d\mu
\left( x\right)  \notag \\
&\approx &\left\vert Q\right\vert ^{\frac{\lambda }{n}-1}\sum_{j=0}^{\infty
}2^{-j\left( n+1-\lambda \right) }\left\vert 2^{j}Q\right\vert _{\mu
}\approx \left\vert Q\right\vert ^{\frac{\lambda }{n}-1}\sum_{j=0}^{\infty
}2^{-j\left( n+1-\lambda \right) }\frac{1}{c2^{-j\theta }}\left\vert
Q\right\vert _{\mu }\approx C_{n,\kappa ,\lambda ,\theta }\left\vert
Q\right\vert ^{\frac{\lambda }{n}-1}\left\vert Q\right\vert _{\mu }\ . 
\notag
\end{eqnarray}

We now turn to defining the decompositions of the bilinear form $%
\left\langle T_{\sigma }^{\lambda }f,g\right\rangle _{\omega }$ used in the
schematic diagram above. For this we first need some preliminaries. We
introduce parameters $r,\varepsilon ,\rho ,\tau $ as in \cite{AlSaUr} and 
\cite{SaShUr7}. We will choose $\varepsilon >0$ sufficiently small later in
the argument, and then $r$ must be chosen sufficiently large depending on $%
\varepsilon $ in order to reduce matters to $\left( r,\varepsilon \right) -%
\func{good}$ functions by the Nazarov, Treil and Volberg argument - see
either \cite{NTV4} or \cite{SaShUr7} for details.

\begin{definition}
\label{def parameters}The parameters $\tau $ and $\rho $ are fixed to
satisfy 
\begin{equation*}
\tau >r\text{ and }\rho >r+\tau ,
\end{equation*}%
where $r$ is the goodness parameter already fixed.
\end{definition}

Let $\mu $ be a positive locally finite Borel measure on $\mathbb{R}^{n}$
that is doubling, let $\mathcal{D}$ be a dyadic grid on $\mathbb{R}^{n}$,
and let $\left\{ \bigtriangleup _{Q}^{\mu }\right\} _{Q\in \mathcal{D}}$ be
the set of weighted Haar projections on $L^{2}\left( \mu \right) $ and $%
\left\{ \mathbb{E}_{Q}^{\mu }\right\} _{Q\in \mathcal{D}}$ the associated
set of projections (see \cite{RaSaWi} for definitions). Recall also the
following bound for the `average' projections $\mathbb{E}_{I}^{\mu }f=\left(
E_{I}^{\mu }f\right) \mathbf{1}_{I}$: 
\begin{equation}
\left\Vert \mathbb{E}_{I}^{\mu }f\right\Vert _{L_{I}^{\infty }\left( \mu
\right) }\lesssim E_{I}^{\mu }\left\vert f\right\vert \leq \sqrt{\frac{1}{%
\left\vert I\right\vert _{\mu }}\int_{I}\left\vert f\right\vert ^{2}d\mu },\
\ \ \ \ \text{for all }f\in L_{\limfunc{loc}}^{2}\left( \mu \right) .
\label{analogue}
\end{equation}%
In terms of the Haar coefficient vectors 
\begin{equation*}
\widehat{f}\left( I\right) \equiv \left\{ \left\langle f,h_{I}^{\mu
,a}\right\rangle \right\} _{a\in \Gamma _{I,n}}
\end{equation*}%
for an orthonormal basis $\left\{ h_{I}^{\mu ,a}\right\} _{a\in \Gamma
_{I,n}}$ of $L_{I}^{2}\left( \mu \right) $ where $\Gamma _{I,n}$ is a
convenient finite index set of size $d_{Q}$, we thus have%
\begin{equation}
\left\vert \widehat{f}\left( I\right) \right\vert =\left\Vert \bigtriangleup
_{I}^{\mu }f\right\Vert _{L^{2}\left( \mu \right) }\leq \left\Vert
\bigtriangleup _{I}^{\mu }f\right\Vert _{L^{\infty }\left( \mu \right) }%
\sqrt{\left\vert I\right\vert _{\mu }}\lesssim \left\Vert \bigtriangleup
_{I}^{\mu }f\right\Vert _{L^{2}\left( \mu \right) }=\left\vert \widehat{f}%
\left( I\right) \right\vert .  \label{analogue'}
\end{equation}

\begin{notation}
We will write the equal quantities $\left\vert \widehat{f}\left( I\right)
\right\vert $ and $\left\Vert \bigtriangleup _{I}^{\mu }f\right\Vert
_{L^{2}\left( \mu \right) }$ interchangeably throughout the paper, depending
on context.
\end{notation}

\subsection{The cube size and corona decompositions}

Now we can define the cube size decomposition in the second row of the
diagram as given in \cite{AlSaUr}. For a sufficiently large positive integer 
$\rho \in \mathbb{N}$, we let 
\begin{equation}
\func{Adj}_{\rho }\left( I\right) \equiv \left\{ J\in \mathcal{D}:2^{-\rho
}\leq \frac{\ell \left( J\right) }{\ell \left( I\right) }\leq 2^{\rho }\text{
and }\overline{J}\cap \overline{I}\neq \emptyset \right\} ,\ \ \ \ \ I\in 
\mathcal{D},  \label{def Adj}
\end{equation}%
be the finite collection of dyadic cubes of side length between $2^{-\rho
}\ell \left( I\right) $ and $2^{\rho }\ell \left( I\right) $, and whose
closures have nonempty intersection. We write $J\Subset _{\rho ,\varepsilon
}I$ to mean that $J\subset I$, $\ell \left( J\right) \leq 2^{-\rho }\ell
\left( I\right) $ and $\limfunc{dist}\left( J,\partial I\right) >2\sqrt{n}%
\ell \left( J\right) ^{\varepsilon }\ell \left( I\right) ^{1-\varepsilon }$.
Then we write%
\begin{eqnarray*}
\left\langle T_{\sigma }^{\lambda }f,g\right\rangle _{\omega }
&=&\dsum\limits_{I,J\in \mathcal{D}}\left\langle T_{\sigma }^{\lambda
}\bigtriangleup _{I}^{\sigma }f,\bigtriangleup _{J}^{\omega }g\right\rangle
_{\omega } \\
&=&\dsum\limits_{\substack{ I,J\in \mathcal{D}  \\ J\Subset _{\rho
,\varepsilon }I}}\left\langle T_{\sigma }^{\lambda }\bigtriangleup
_{I}^{\sigma }f,\bigtriangleup _{J}^{\omega }g\right\rangle _{\omega
}+\dsum\limits_{\substack{ I,J\in \mathcal{D}  \\ J_{\rho ,\varepsilon
}\Supset I}}\left\langle T_{\sigma }^{\lambda }\bigtriangleup _{I}^{\sigma
}f,\bigtriangleup _{J}^{\omega }g\right\rangle _{\omega } \\
&&+\dsum\limits_{I,J\in \mathcal{D}:\ J\cap I=\emptyset ,\frac{\ell \left(
J\right) }{\ell \left( I\right) }<2^{-\rho }\text{ or }\frac{\ell \left(
J\right) }{\ell \left( I\right) }>2^{\rho }}\left\langle T_{\sigma
}^{\lambda }\bigtriangleup _{I}^{\sigma }f,\bigtriangleup _{J}^{\omega
}g\right\rangle _{\omega } \\
&&+\dsum\limits_{\substack{ I,J\in \mathcal{D}  \\ 2^{-\rho }\leq \frac{\ell
\left( J\right) }{\ell \left( I\right) }\leq 2^{\rho }\text{ and }\overline{J%
}\cap \overline{I}=\emptyset }}\left\langle T_{\sigma }^{\lambda
}\bigtriangleup _{I}^{\sigma }f,\bigtriangleup _{J}^{\omega }g\right\rangle
_{\omega }+\dsum\limits_{I\in \mathcal{D},\ J\in \func{Adj}_{\rho }\left(
I\right) }\left\langle T_{\sigma }^{\lambda }\bigtriangleup _{I}^{\sigma
}f,\bigtriangleup _{J}^{\omega }g\right\rangle _{\omega } \\
&\equiv &\mathsf{B}_{\Subset _{\rho ,\varepsilon }}\left( f,g\right) +%
\mathsf{B}_{_{\rho ,\varepsilon }\Supset }\left( f,g\right) +\mathsf{B}%
_{\cap }\left( f,g\right) +\mathsf{B}_{\diagup }\left( f,g\right) +\mathsf{B}%
_{\func{Adj},\rho }\left( f,g\right) .
\end{eqnarray*}%
The disjoint and comparable forms $\mathsf{B}_{\cap }\left( f,g\right) $ and 
$\mathsf{B}_{\diagup }\left( f,g\right) $ are controlled using only the
quadratic offset Muckenhoupt condition, while the adjacent form $\mathsf{B}_{%
\func{Adj},\rho }\left( f,g\right) $ is controlled by the Alpert weak
boundedness property. The above form $\mathsf{B}_{_{\rho ,\varepsilon
}\Supset }\left( f,g\right) $ is handled exactly as is the below form $%
\mathsf{B}_{\Subset _{\rho ,\varepsilon }}\left( f,g\right) $ but
interchanging the measures $\sigma $ and $\omega $, and the exponents $p$
and $p^{\prime }$, as well as using the duals of the scalar testing and
quadratic Muckenhoupt testing conditions. So it remains only to treat the
below form $\mathsf{B}_{\Subset _{\rho ,\varepsilon }}\left( f,g\right) $,
to which we now turn.

In order to describe the ensuing decompositions of $\mathsf{B}_{\Subset
_{\rho ,\varepsilon }}\left( f,g\right) $, we first need to introduce the
corona and shifted corona decompositions of $f$ and $g$ respectively. We
construct the \emph{Calder\'{o}n-Zygmund} corona decomposition for a
function $f$ in $L^{p}\left( \mu \right) $ (where $\mu =\sigma $ here, and
where $\mu =\omega $ when treating $\mathsf{B}_{_{\rho ,\varepsilon }\Supset
}\left( f,g\right) $) and that is supported in a dyadic cube $F_{1}^{0}$.
Fix $\Gamma >1$ and define $\mathcal{G}_{0}=\left\{ F_{1}^{0}\right\} $ to
consist of the single cube $F_{1}^{0}$, and define the first generation $%
\mathcal{G}_{1}=\left\{ F_{k}^{1}\right\} _{k}$ of \emph{CZ stopping children%
} of $F_{1}^{0}$ to be the \emph{maximal} dyadic subcubes $I$ of $F_{0}$
satisfying%
\begin{equation*}
E_{I}^{\mu }\left\vert f\right\vert \geq \Gamma E_{F_{1}^{0}}^{\mu
}\left\vert f\right\vert .
\end{equation*}%
Then define the second generation $\mathcal{G}_{2}=\left\{ F_{k}^{2}\right\}
_{k}$ of CZ\emph{\ }stopping children of $F_{1}^{0}$ to be the \emph{maximal}
dyadic subcubes $I$ of some $F_{k}^{1}\in \mathcal{G}_{1}$ satisfying%
\begin{equation*}
E_{I}^{\mu }\left\vert f\right\vert \geq \Gamma E_{F_{k}^{1}}^{\mu
}\left\vert f\right\vert .
\end{equation*}%
Continue by recursion to define $\mathcal{G}_{n}$ for all $n\geq 0$, and
then set 
\begin{equation*}
\mathcal{F\equiv }\dbigcup\limits_{n=0}^{\infty }\mathcal{G}_{n}=\left\{
F_{k}^{n}:n\geq 0,k\geq 1\right\}
\end{equation*}%
to be the set of all CZ stopping intervals in $F_{1}^{0}$ obtained in this
way.

The $\mu $-Carleson condition for $\mathcal{F}$ follows as usual from the
first step,%
\begin{equation*}
\sum_{F^{\prime }\in \mathfrak{C}_{\mathcal{F}}\left( F\right) }\left\vert
F^{\prime }\right\vert _{\mu }\leq \frac{1}{\Gamma }\sum_{F^{\prime }\in 
\mathfrak{C}_{\mathcal{F}}\left( F\right) }\frac{1}{E_{F}^{\mu }\left\vert
f\right\vert }\int_{F^{\prime }}\left\vert f\right\vert d\mu \leq \frac{1}{%
\Gamma }\left\vert F\right\vert _{\mu }.
\end{equation*}%
Moreover, if we define 
\begin{equation}
\alpha _{\mathcal{F}}\left( F\right) \equiv \sup_{F^{\prime }\in \mathcal{F}%
:\ F\subset F^{\prime }}E_{F^{\prime }}^{\mu }\left\vert f\right\vert ,
\label{def alpha}
\end{equation}%
then in each corona 
\begin{equation*}
\mathcal{C}_{F}\equiv \left\{ I\in \mathcal{D}:I\subset F\text{ and }%
I\not\subset F^{\prime }\text{ for any }F^{\prime }\in \mathcal{F}\text{
with }F^{\prime }\varsubsetneqq F\right\} ,
\end{equation*}%
we have, from the definition of the stopping times, the following average
control%
\begin{equation}
E_{I}^{\mu }\left\vert f\right\vert <\Gamma \alpha _{\mathcal{F}}\left(
F\right) ,\ \ \ \ \ I\in \mathcal{C}_{F}\text{ and }F\in \mathcal{F}.
\label{average control}
\end{equation}

Finally, as in \cite{NTV4}, \cite{LaSaShUr3} and \cite{SaShUr7}, we obtain
the Carleson condition and quasiorthogonality inequality,%
\begin{equation}
\sum_{F^{\prime }\preceq F}\left\vert F^{\prime }\right\vert _{\mu }\leq
C_{0}\left\vert F\right\vert _{\mu }\text{ for all }F\in \mathcal{F};\text{\
and }\sum_{F\in \mathcal{F}}\alpha _{\mathcal{F}}\left( F\right)
^{2}\left\vert F\right\vert _{\mu }\mathbf{\leq }C_{0}^{2}\left\Vert
f\right\Vert _{L^{2}\left( \mu \right) }^{2},  \label{Car and quasi}
\end{equation}%
where $\preceq $ denotes the tree relation $F^{\prime }\subset F$ for $%
F^{\prime },F\in \mathcal{F}$. Moreover, there is the following useful
consequence of (\ref{Car and quasi}) that says the sequence $\left\{ \alpha
_{\mathcal{F}}\left( F\right) \mathbf{1}_{F}\right\} _{F\in \mathcal{F}}$
has an additional \emph{quasiorthogonal} property relative to $f$ with a
constant $C_{0}^{\prime }$ depending only on $C_{0}$:%
\begin{equation}
\left\Vert \sum_{F\in \mathcal{F}}\alpha _{\mathcal{F}}\left( F\right) 
\mathbf{1}_{F}\right\Vert _{L^{2}\left( \mu \right) }^{2}\leq C_{0}^{\prime
}\left\Vert f\right\Vert _{L^{2}\left( \mu \right) }^{2}.  \label{q orth}
\end{equation}%
Indeed, this is an easy consequence of a geometric decay in levels of the
tree $\mathcal{F}$, that follows in turn from the Carleson condition in the
first inequality of (\ref{Car and quasi}).

This geometric decay asserts that there are positive constants $C_{1}$ and $%
\varepsilon $, depending on $C_{0}$, such that if $\mathfrak{C}_{\mathcal{F}%
}^{\left( n\right) }\left( F\right) $ denotes the set of $n^{th}$ generation
children of $F$ in $\mathcal{F}$,%
\begin{equation}
\sum_{F^{\prime }\in \mathfrak{C}_{\mathcal{F}}^{\left( n\right) }\left(
F\right) }\left\vert F^{\prime }\right\vert _{\mu }\leq \left(
C_{1}2^{-\varepsilon n}\right) ^{2}\left\vert F\right\vert _{\mu },\ \ \ \ \ 
\text{for all }n\geq 0\text{ and }F\in \mathcal{F}.  \label{geom decay}
\end{equation}%
To see this, let $\beta _{k}\left( F\right) \equiv \sum_{F^{\prime }\in 
\mathfrak{C}_{\mathcal{F}}^{\left( k\right) }\left( F\right) }\left\vert
F^{\prime }\right\vert _{\mu }$ and note that $\beta _{k+1}\left( F\right)
\leq \beta _{k}\left( F\right) $ implies that for any integer $N\geq C$, we
have%
\begin{equation*}
\left( N+1\right) \beta _{N}\left( F\right) \leq \sum_{k=0}^{N}\beta
_{k}\left( F\right) \leq C\left\vert F\right\vert _{\mu },
\end{equation*}%
and hence 
\begin{equation*}
\beta _{N}\left( F\right) \leq \frac{C}{N+1}\left\vert F\right\vert _{\mu }<%
\frac{1}{2}\left\vert F\right\vert _{\mu }\ ,\ \ \ \ \ \text{for }F\in 
\mathcal{F}\text{ and }N=\left[ 2C\right] .
\end{equation*}%
It follows that 
\begin{equation*}
\beta _{\ell N}\left( F\right) \leq \frac{1}{2}\beta _{\left( \ell -1\right)
N}\left( F\right) \leq ...\leq \frac{1}{2^{\ell }}\beta _{0}\left( F\right) =%
\frac{1}{2^{\ell }}\left\vert F\right\vert _{\mu },\ \ \ \ \ \ell =0,1,2,...
\end{equation*}%
and so given $n\in \mathbb{N}$, choose $\ell $ such that $\ell N\leq
n<\left( \ell +1\right) N$, and note that 
\begin{equation*}
\sum_{F^{\prime }\in \mathfrak{C}_{\mathcal{F}}^{\left( n\right) }\left(
F\right) }\left\vert F^{\prime }\right\vert _{\mu }=\beta _{n}\left(
F\right) \leq \beta _{\ell N}\left( F\right) \approx C_{1}2^{-\varepsilon
n}\left\vert F\right\vert _{\mu }\ ,
\end{equation*}%
which proves the geometric decay (\ref{geom decay}).

Now let $\sigma $ and $\omega $ be doubling measures and define the two
corona projections%
\begin{equation*}
\mathsf{P}_{\mathcal{C}_{F}}^{\sigma }\equiv \sum_{I\in \mathcal{C}%
_{F}}\bigtriangleup _{I}^{\sigma }\text{ and }\mathsf{P}_{\mathcal{C}%
_{F}^{\tau -\limfunc{shift}}}^{\omega }\equiv \sum_{J\in \mathcal{C}%
_{F}^{\tau -\limfunc{shift}}}\bigtriangleup _{J}^{\omega }\ ,
\end{equation*}%
where%
\begin{align}
\mathcal{C}_{F}^{\tau -\limfunc{shift}}& \equiv \left[ \mathcal{C}%
_{F}\setminus \mathcal{N}_{\mathcal{D}}^{\tau }\left( F\right) \right] \cup
\dbigcup\limits_{F^{\prime }\in \mathfrak{C}_{\mathcal{F}}\left( F\right) }%
\left[ \mathcal{N}_{\mathcal{D}}^{\tau }\left( F^{\prime }\right) \setminus 
\mathcal{N}_{\mathcal{D}}^{\tau }\left( F\right) \right] ;  \label{def shift}
\\
\text{where }\mathcal{N}_{\mathcal{D}}^{\tau }\left( F\right) & \equiv
\left\{ J\in \mathcal{D}:J\subset F\text{ and }\ell \left( J\right)
>2^{-\tau }\ell \left( F\right) \right\} ,  \notag
\end{align}%
and note that $f=\sum_{F\in \mathcal{F}}\mathsf{P}_{\mathcal{C}_{F}}^{\sigma
}f$. Thus the corona $\mathcal{C}_{F}^{\tau -\limfunc{shift}}$ has the top $%
\tau $ levels from $\mathcal{C}_{F}$ removed, and includes the first $\tau $
levels from each of its $\mathcal{F}$-children, except if they have already
been removed.

\subsection{The canonical splitting}

We can now continue with the definitions of decompositions in the schematic
diagram above. To bound the below form $\mathsf{B}_{\Subset _{\rho
,\varepsilon }}\left( f,g\right) $, we proceed with the \emph{Canonical
Splitting} of 
\begin{equation*}
\mathsf{B}_{\Subset _{\mathbf{\rho },\varepsilon }}\left( f,g\right)
=\dsum\limits_{\substack{ I,J\in \mathcal{D}  \\ J\Subset _{\rho
,\varepsilon }I}}\left\langle T_{\sigma }^{\lambda }\left( \bigtriangleup
_{I}^{\sigma }f\right) ,\left( \bigtriangleup _{J}^{\omega }g\right)
\right\rangle _{\omega }
\end{equation*}%
as in \cite{SaShUr7} and \cite{AlSaUr},%
\begin{align*}
\mathsf{B}_{\Subset _{\mathbf{\rho },\varepsilon }}\left( f,g\right) &
=\sum_{F\in \mathcal{F}}\left\langle T_{\sigma }^{\lambda }\mathsf{P}_{%
\mathcal{C}_{F}}^{\sigma }f,\mathsf{P}_{\mathcal{C}_{F}^{\tau -\limfunc{shift%
}}}^{\omega }g\right\rangle _{\omega }^{\Subset _{\rho }}+\sum_{\substack{ %
F,G\in \mathcal{F}  \\ G\subsetneqq F}}\left\langle T_{\sigma }^{\lambda }%
\mathsf{P}_{\mathcal{C}_{F}}^{\sigma }f,\mathsf{P}_{\mathcal{C}_{G}^{\tau -%
\limfunc{shift}}}^{\omega }g\right\rangle _{\omega }^{\Subset _{\rho }} \\
& +\sum_{\substack{ F,G\in \mathcal{F}  \\ G\supsetneqq F}}\left\langle
T_{\sigma }^{\lambda }\mathsf{P}_{\mathcal{C}_{F}}^{\sigma }f,\mathsf{P}_{%
\mathcal{C}_{G}^{\tau -\limfunc{shift}}}^{\omega }g\right\rangle _{\omega
}^{\Subset _{\rho }}+\sum_{\substack{ F,G\in \mathcal{F}  \\ F\cap
G=\emptyset }}\left\langle T_{\sigma }^{\lambda }\mathsf{P}_{\mathcal{C}%
_{F}}^{\sigma }f,\mathsf{P}_{\mathcal{C}_{G}^{\tau -\limfunc{shift}%
}}^{\omega }g\right\rangle _{\omega }^{\Subset _{\rho }} \\
& \equiv \mathsf{T}_{\limfunc{diagonal}}\left( f,g\right) +\mathsf{T}_{%
\limfunc{far}\limfunc{below}}\left( f,g\right) +\mathsf{T}_{\limfunc{far}%
\limfunc{above}}\left( f,g\right) +\mathsf{T}_{\limfunc{disjoint}}\left(
f,g\right) ,
\end{align*}%
where for $F\in \mathcal{F}$ we use the shorthand%
\begin{equation*}
\left\langle T_{\sigma }^{\lambda }\left( \mathsf{P}_{\mathcal{C}%
_{F}}^{\sigma }f\right) ,\mathsf{P}_{\mathcal{C}_{F}^{\tau -\limfunc{shift}%
}}^{\omega }g\right\rangle _{\omega }^{\Subset _{\rho }}\equiv \dsum\limits 
_{\substack{ I\in \mathcal{C}_{F},\ J\in \mathcal{C}_{F}^{\tau -\limfunc{%
shift}}  \\ J\Subset _{\rho ,\varepsilon }I}}\left\langle T_{\sigma
}^{\lambda }\left( \bigtriangleup _{I}^{\sigma }f\right) ,\left(
\bigtriangleup _{J}^{\omega }g\right) \right\rangle _{\omega }.
\end{equation*}%
The final two forms $\mathsf{T}_{\limfunc{far}\limfunc{above}}\left(
f,g\right) $ and $\mathsf{T}_{\limfunc{disjoint}}\left( f,g\right) $ each
vanish just as in \cite{SaShUr7} and \cite{AlSaUr}, since there are no pairs 
$\left( I,J\right) \in \mathcal{C}_{F}\times \mathcal{C}_{G}^{\tau -\limfunc{%
shift}}$ with both (\textbf{i}) $J\Subset _{\rho ,\varepsilon }I$ and (%
\textbf{ii}) either $F\subsetneqq G$ or $G\cap F=\emptyset $. The far below
form $\mathsf{T}_{\limfunc{far}\limfunc{below}}\left( f,g\right) $ is then
further split into two forms $\mathsf{T}_{\limfunc{far}\limfunc{below}%
}^{1}\left( f,g\right) $ and $\mathsf{T}_{\limfunc{far}\limfunc{below}%
}^{2}\left( f,g\right) $ as in \cite{SaShUr7} and \cite{AlSaUr},%
\begin{align}
\mathsf{T}_{\limfunc{far}\limfunc{below}}\left( f,g\right) & =\sum_{G\in 
\mathcal{F}}\sum_{F\in \mathcal{F}:\ G\subsetneqq F}\sum_{\substack{ I\in 
\mathcal{C}_{F}\text{ and }J\in \mathcal{C}_{G}^{\tau -\limfunc{shift}}  \\ %
J\Subset _{\rho ,\varepsilon }I}}\left\langle T_{\sigma }^{\lambda
}\bigtriangleup _{I}^{\sigma }f,\bigtriangleup _{J}^{\omega }g\right\rangle
_{\omega }  \label{far below decomp} \\
& =\sum_{G\in \mathcal{F}}\sum_{F\in \mathcal{F}:\ G\subsetneqq F}\sum_{J\in 
\mathcal{C}_{G}^{\tau -\limfunc{shift}}}\sum_{I\in \mathcal{C}_{F}\text{ and 
}J\subset I}\left\langle T_{\sigma }^{\lambda }\bigtriangleup _{I}^{\sigma
}f,\bigtriangleup _{J}^{\omega }g\right\rangle _{\omega }  \notag \\
& -\sum_{F\in \mathcal{F}}\sum_{G\in \mathcal{F}:\ G\subsetneqq F}\sum_{J\in 
\mathcal{C}_{G}^{\tau -\limfunc{shift}}}\sum_{I\in \mathcal{C}_{F}\text{ and 
}J\subset I\text{ but }J\not\Subset _{\rho ,\varepsilon }I}\left\langle
T_{\sigma }^{\lambda }\bigtriangleup _{I}^{\sigma }f,\bigtriangleup
_{J}^{\omega }g\right\rangle _{\omega }  \notag \\
& \equiv \mathsf{T}_{\limfunc{far}\limfunc{below}}^{1}\left( f,g\right) -%
\mathsf{T}_{\limfunc{far}\limfunc{below}}^{2}\left( f,g\right) .  \notag
\end{align}%
The second far below form $\mathsf{T}_{\limfunc{far}\limfunc{below}%
}^{2}\left( f,g\right) $ satisfies%
\begin{equation}
\left\vert \mathsf{T}_{\limfunc{far}\limfunc{below}}^{2}\left( f,g\right)
\right\vert \lesssim \left( A_{p}^{\lambda ,\ell ^{2},\limfunc{offset}%
}\left( \sigma ,\omega \right) +\mathcal{WBP}_{T^{\lambda },p}^{\ell
^{2}}\left( \sigma ,\omega \right) \right) \left\Vert f\right\Vert
_{L^{p}\left( \sigma \right) }\left\Vert g\right\Vert _{L^{p^{\prime
}}\left( \omega \right) },  \label{second far below}
\end{equation}%
which follows in an easy way from (\ref{routine'}) and (\ref{routine''}) and
their porisms - see below. To control the first and main far below form $%
\mathsf{T}_{\limfunc{far}\limfunc{below}}^{1}\left( f,g\right) $, we will
use some new quadratic arguments exploiting Carleson measure conditions to
establish%
\begin{equation}
\left\vert \mathsf{T}_{\limfunc{far}\limfunc{below}}^{1}\left( f,g\right)
\right\vert \lesssim A_{p}^{\lambda ,\ell ^{2},\limfunc{offset}}\left(
\sigma ,\omega \right) \left\Vert f\right\Vert _{L^{p}\left( \sigma \right)
}\left\Vert g\right\Vert _{L^{p^{\prime }}\left( \omega \right) }.
\label{first far below}
\end{equation}

To handle the diagonal term $\mathsf{T}_{\limfunc{diagonal}}\left(
f,g\right) $, we further decompose according to the stopping times $\mathcal{%
F}$,%
\begin{equation}
\mathsf{T}_{\limfunc{diagonal}}\left( f,g\right) =\sum_{F\in \mathcal{F}}%
\mathsf{B}_{\Subset _{\rho ,\varepsilon }}^{F}\left( f,g\right) ,\text{
where }\mathsf{B}_{\Subset _{\rho ,\varepsilon }}^{F}\left( f,g\right)
\equiv \left\langle T_{\sigma }^{\lambda }\left( \mathsf{P}_{\mathcal{C}%
_{F}}^{\sigma }f\right) ,\mathsf{P}_{\mathcal{C}_{F}^{\tau -\limfunc{shift}%
}}^{\omega }g\right\rangle _{\omega }^{\Subset _{\rho }},  \label{def block}
\end{equation}%
where we recall that in \cite{AlSaUr} for $p=2$, the following estimate was
obtained,%
\begin{equation}
\left\vert \mathsf{B}_{\Subset _{\rho }}^{F}\left( f,g\right) \right\vert
\lesssim \left( \mathfrak{T}_{T^{\lambda }}+\sqrt{A_{2}^{\lambda }}\right) \
\left( \left\Vert \mathbb{E}_{F;\kappa }^{\sigma }f\right\Vert _{\infty }%
\sqrt{\left\vert F\right\vert _{\sigma }}+\left\Vert \mathsf{P}_{\mathcal{C}%
_{F}}^{\sigma }f\right\Vert _{L^{2}\left( \sigma \right) }\right) \
\left\Vert \mathsf{P}_{\mathcal{C}_{F}^{\tau -\limfunc{shift}}}^{\omega
}g\right\Vert _{L^{2}\left( \omega \right) }.  \label{below form bound'}
\end{equation}%
This was achieved by implementing the classical \emph{reach} of Nazarov,
Treil and Volberg using Haar wavelet projections $\bigtriangleup
_{I}^{\sigma }$, where by `reach' we mean the ingenious `thinking outside
the box' idea of the paraproduct / stopping / neighbour decomposition of
Nazarov, Treil and Volberg \cite{NTV4}.

\subsection{The Nazarov, Treil and Volberg reach}

Here is the Nazarov, Treil and Volberg decomposition, or reach. We have that 
$\mathsf{B}_{\Subset _{\rho ,\varepsilon };\kappa }^{F}\left( f,g\right) $
equals%
\begin{align*}
& \sum_{\substack{ I\in \mathcal{C}_{F}\text{ and }J\in \mathcal{C}%
_{F}^{\tau -\limfunc{shift}}  \\ J\Subset _{\rho ,\varepsilon }I}}%
\left\langle T_{\sigma }^{\lambda }\left( \mathbf{1}_{I_{J}}\bigtriangleup
_{I}^{\sigma }f\right) ,\bigtriangleup _{J}^{\omega }g\right\rangle _{\omega
}+\sum_{\substack{ I\in \mathcal{C}_{F}\text{ and }J\in \mathcal{C}%
_{F}^{\tau -\limfunc{shift}}  \\ J\Subset _{\rho ,\varepsilon }I}}%
\sum_{\theta \left( I_{J}\right) \in \mathfrak{C}_{\mathcal{D}}\left(
I\right) \setminus \left\{ I_{J}\right\} }\left\langle T_{\sigma }^{\lambda
}\left( \mathbf{1}_{\theta \left( I_{J}\right) }\bigtriangleup _{I}^{\sigma
}f\right) ,\bigtriangleup _{J}^{\omega }g\right\rangle _{\omega } \\
& \equiv \mathsf{B}_{\func{home}}^{F}\left( f,g\right) +\mathsf{B}_{\limfunc{%
neighbour}}^{F}\left( f,g\right) ,
\end{align*}%
and we further decompose the home form using the constant 
\begin{equation}
M_{I^{\prime }}\equiv \mathbf{1}_{I^{\prime }}\bigtriangleup _{I}^{\sigma }f=%
\mathbb{E}_{I^{\prime }}^{\sigma }\bigtriangleup _{I}^{\sigma }f,
\label{def M}
\end{equation}%
to obtain%
\begin{align*}
\mathsf{B}_{\func{home}}^{F}\left( f,g\right) & =\sum_{\substack{ I\in 
\mathcal{C}_{F}\text{ and }J\in \mathcal{C}_{F}^{\tau -\limfunc{shift}}  \\ %
J\Subset _{\rho ,\varepsilon }I}}\left\langle M_{I_{J}}T_{\sigma }^{\lambda }%
\mathbf{1}_{F},\bigtriangleup _{J}^{\omega }g\right\rangle _{\omega }-\sum 
_{\substack{ I\in \mathcal{C}_{F}\text{ and }J\in \mathcal{C}_{F}^{\tau -%
\limfunc{shift}}  \\ J\Subset _{\rho ,\varepsilon }I}}\left\langle
M_{I_{J}}T_{\sigma }^{\lambda }\mathbf{1}_{F\setminus I_{J}},\bigtriangleup
_{J}^{\omega }g\right\rangle _{\omega } \\
& \equiv \mathsf{B}_{\limfunc{paraproduct}}^{F}\left( f,g\right) +\mathsf{B}%
_{\limfunc{stop}}^{F}\left( f,g\right) .
\end{align*}

Altogether then we have the the Nazarov, Treil and Volberg paraproduct
decomposition,%
\begin{equation*}
\mathsf{B}_{\Subset _{\rho ,\varepsilon }}^{F}\left( f,g\right) =\mathsf{B}_{%
\limfunc{paraproduct}}^{F}\left( f,g\right) +\mathsf{B}_{\limfunc{stop}%
}^{F}\left( f,g\right) +\mathsf{B}_{\limfunc{neighbour}}^{F}\left(
f,g\right) .
\end{equation*}%
Several points of departure can now be identified in the following
description of the remainder of the paper. While we use here terminology yet
to be defined, the reader is nevertheless encouraged to keep these seven
points in mind while reading.

\begin{enumerate}
\item In order to obtain an estimate such as (\ref{below form bound'}) for $%
p\neq 2$, we will need to use square functions and vector-valued
inequalities as motivated by \cite{HyVu}, that in turn will require the
quadratic Muckenhoupt condition in place of the classical one, and we turn
to these issues in the next section.

\item A guiding principle will be to apply the pointwise $\ell ^{2}$
Cauchy-Schwarz inequality early in the proof, and then manipulate the
resulting vector-valued inequalities into a form where application of the
hypotheses reduce matters to the Fefferman-Stein inequalities for the vector
maximal function, and square function estimates.

\item After that we will prove necessity of quadratic testing and
Muckenhoupt conditions in Section 4. We also introduce a quadratic \emph{Haar%
} weak boundedness property that helps clarify the role of weak boundedness,
and show that it is controlled by quadratic weak boundedness and quadratic
offset Muckenhoupt.

\item The first forms we choose to control in Section 5 are the comparable
form and the paraproduct form, called the `difficult' form in \cite{NTV4},
each of which use only the local quadratic testing conditions.

\item Following that we first consider in Section 6 the disjoint, stopping,
far below and neighbour forms, all of which require what we call a `Pivotal
Lemma' that originated in \cite{NTV4}, as well as the quadratic Muckenhoupt
conditions. The stopping form requires in addition a new argument exploiting
an extreme energy reversal property of vector Riesz transforms.

\item Next we consider the commutator form\ in Section 7, which requires a
new pigeon-holing of the tower of dyadic cubes lying above a fixed point in
space, as well as Taylor expansions and quadratic offset Muckenhoupt
conditions, thus constituting another of the difficult new arguments in the
paper. The proof of the main theorem is wrapped up here as well.

\item Finally, the Appendix in Section 8 contains an example for $p\neq 2$
of radially decreasing weights on the real line for which $A_{p}<\infty $
but $A_{p}^{\lambda ,\ell ^{2},\limfunc{offset}}=\infty $.
\end{enumerate}

\subsection{A quadratic Carleson measure inequality}

We end this section with a quadratic Carleson measure inequality we will
need for bounding the stopping form below.

\begin{theorem}
\label{using Carleson}Suppose that the triple $\left( C_{0},\mathcal{F}%
,\alpha _{\mathcal{F}}\right) $ constitutes \emph{stopping data} for a
function $f\in L_{loc}^{1}\left( \mu \right) $, and for $\kappa \in \mathbb{Z%
}_{+}$, set 
\begin{equation*}
\alpha _{\mathcal{F}}^{\kappa }\left( x\right) \equiv \left\{ \alpha _{%
\mathcal{F}}\left( F\right) \mathbf{1}_{F^{\kappa }}\left( x\right) \right\}
_{F\in \mathcal{F}}\ \text{where }F^{\kappa }\equiv \bigcup_{G\in \mathfrak{C%
}_{\mathcal{F}}^{\left( \kappa \right) }\left( F\right) }G\text{ }.
\end{equation*}%
Then for $1<p<\infty $, 
\begin{equation}
\int_{\mathbb{R}^{n}}\left\vert \alpha _{\mathcal{F}}^{\kappa }\left(
x\right) \right\vert _{\ell ^{2}}^{p}d\mu \left( x\right) =\int_{\mathbb{R}%
^{n}}\left( \sum_{F\in \mathcal{F}}\left\vert \alpha _{\mathcal{F}}\left(
F\right) \right\vert ^{2}\mathbf{1}_{F^{\kappa }}\left( x\right) \right) ^{%
\frac{p}{2}}d\mu \left( x\right) \leq C_{\delta }2^{-\delta \kappa
}\sum_{F\in \mathcal{F}}\alpha _{\mathcal{F}}\left( F\right) ^{p}\left\vert
F\right\vert _{\mu }\ ,  \label{The claim}
\end{equation}%
where $\delta >0$ is the constant in (\ref{geom decay}). The inequality can
be reversed for $\kappa =0$ and $2\leq p<\infty $.
\end{theorem}

\begin{proof}[Proof of Theorem \protect\ref{using Carleson}]
We claim that for $1<p<\infty $, i.e.%
\begin{equation}
\int_{\mathbb{R}^{n}}\left( \sum_{F\in \mathcal{F}}\left\vert \alpha _{%
\mathcal{F}}\left( F\right) \right\vert ^{2}\mathbf{1}_{F}\left( x\right)
\right) ^{\frac{p}{2}}d\mu \left( x\right) \leq C_{\delta }\sum_{F\in 
\mathcal{F}}\alpha _{\mathcal{F}}\left( F\right) ^{p}\left\vert F\right\vert
_{\mu }.  \label{obs}
\end{equation}%
Indeed, for $1<p\leq 2$ (and even for $0<p\leq 2$), the inequality follows
from the trivial inequality $\left\Vert \cdot \right\Vert _{\ell ^{q}}\leq
\left\Vert \cdot \right\Vert _{\ell ^{1}}$ for $0<q\leq 1$,%
\begin{eqnarray*}
&&\int_{\mathbb{R}^{n}}\left( \sum_{F\in \mathcal{F}}\left\vert \alpha _{%
\mathcal{F}}\left( F\right) \right\vert ^{2}\mathbf{1}_{F}\left( x\right)
\right) ^{\frac{p}{2}}d\mu \left( x\right) \leq \int_{\mathbb{R}%
^{n}}\sum_{F\in \mathcal{F}}\left\vert \alpha _{\mathcal{F}}\left( F\right)
\right\vert ^{p}\mathbf{1}_{F}\left( x\right) d\mu \left( x\right) \\
&&\ \ \ \ \ \ \ \ \ \ \ \ \ \ \ =\sum_{F\in \mathcal{F}}\alpha _{\mathcal{F}%
}\left( F\right) ^{p}\left\vert F\right\vert _{\mu }\leq C_{\delta
}\sum_{F\in \mathcal{F}}\alpha _{\mathcal{F}}\left( F\right) ^{p}\left\vert
F\right\vert _{\mu },
\end{eqnarray*}%
where $\delta >0$ is the geometric decay in generations exponent in (\ref%
{geom decay}).

Now we turn to the case $p\geq 2$. When $p=2m$ is an even positive integer,
we will set%
\begin{equation*}
\mathcal{F}_{\ast }^{2m}\equiv \left\{ \left( F_{1},...,F_{2m}\right) \in 
\mathcal{F}\times ...\times \mathcal{F}:F_{i}\subset F_{j}\text{ for }1\leq
i\leq j\leq 2m\text{, and }F_{i}=F_{i+1}\text{ for all odd }i\right\} ,
\end{equation*}%
and then by symmetry we can arrange the intervals below in nondecreasing
order to obtain%
\begin{eqnarray*}
&&\int_{\mathbb{R}^{n}}\left( \sum_{F\in \mathcal{F}}\left\vert \alpha _{%
\mathcal{F}}\left( F\right) \mathbf{1}_{F}\left( x\right) \right\vert
^{2}\right) ^{\frac{p}{2}}d\mu \left( x\right) =\int_{\mathbb{R}^{n}}\left(
\sum_{F\in \mathcal{F}}\left\vert \alpha _{\mathcal{F}}\left( F\right) 
\mathbf{1}_{F}\left( x\right) \right\vert ^{2}\right) ^{m}d\mu \left(
x\right) \\
&=&\int_{\mathbb{R}^{n}}\sum_{\left( F_{1},...,F_{2m}\right) \in \mathcal{F}%
^{2m}}\alpha _{\mathcal{F}}\left( F_{1}\right) ...\alpha _{\mathcal{F}%
}\left( F_{2m}\right) \mathbf{1}_{F_{1}\cap ...\cap F_{2m}}d\mu \left(
x\right) \\
&=&C_{m}\int_{\mathbb{R}^{n}}\sum_{\left( F_{1},...,F_{2m}\right) \in 
\mathcal{F}_{\ast }^{2m}}\alpha _{\mathcal{F}}\left( F_{1}\right) ...\alpha
_{\mathcal{F}}\left( F_{2m}\right) \mathbf{1}_{F_{1}\cap ...\cap F_{2m}}d\mu
\left( x\right) \\
&=&C_{m}\sum_{\left( F_{1},...,F_{2m}\right) \in \mathcal{F}_{\ast
}^{2m}}\alpha _{\mathcal{F}}\left( F_{1}\right) ...\alpha _{\mathcal{F}%
}\left( F_{2m}\right) \left\vert F_{1}\right\vert _{\mu }=C_{m}\func{Int}%
\left( m\right) ,
\end{eqnarray*}%
where from the geometric decay in (\ref{geom decay}), we obtain%
\begin{eqnarray}
\func{Int}\left( m\right) &\equiv &\sum_{\left( F_{1},...,F_{2m}\right) \in 
\mathcal{F}_{\ast }^{2m}}\alpha _{\mathcal{F}}\left( F_{1}\right) ...\alpha
_{\mathcal{F}}\left( F_{2m}\right) \left\vert F_{1}\right\vert _{\mu
}\lesssim \func{Int}\left( m\right) ,  \label{Int} \\
\text{where }\func{Int}\left( m\right) &\equiv &\sum_{\left(
F_{1},...,F_{2m}\right) \in \mathcal{F}_{\ast }^{2m}}\alpha _{\mathcal{F}%
}\left( F_{1}\right) ...\alpha _{\mathcal{F}}\left( F_{2m}\right) \left\vert
F_{1}\right\vert _{\mu }.  \notag
\end{eqnarray}

We now set about showing that%
\begin{equation*}
\func{Int}\left( m\right) \lesssim \sum_{F\in \mathcal{F}}\left\vert \alpha
_{\mathcal{F}}\left( F\right) \right\vert ^{2m}\left\vert F\right\vert _{\mu
}\ .
\end{equation*}%
For this, we first prove (\ref{q orth}) in order to outline the main idea.
Using the geometric decay in (\ref{geom decay}) once more we obtain 
\begin{equation*}
\sum_{F^{\prime }\in \mathcal{F}\left[ F\right] :\ }\alpha _{\mathcal{F}%
}\left( F^{\prime }\right) \left\vert F^{\prime }\right\vert _{\mu }\leq
\sum_{n=0}^{\infty }\sqrt{\sum_{F^{\prime }\in \mathcal{F}\left[ F\right]
}\alpha _{\mathcal{F}}\left( F^{\prime }\right) ^{2}\left\vert F^{\prime
}\right\vert _{\mu }}C_{\delta }\sqrt{\left\vert F\right\vert _{\mu }}\leq
C_{\delta }\sqrt{\left\vert F\right\vert _{\mu }}\sqrt{\sum_{F^{\prime }\in 
\mathcal{F}\left[ F\right] }\alpha _{\mathcal{F}}\left( F^{\prime }\right)
^{2}\left\vert F^{\prime }\right\vert _{\mu }},
\end{equation*}%
and hence that$\ $%
\begin{eqnarray*}
&&\sum_{F\in \mathcal{F}}\alpha _{\mathcal{F}}\left( F\right) \left\{
\sum_{F^{\prime }\in \mathcal{F}\left[ F\right] }\alpha _{\mathcal{F}}\left(
F^{\prime }\right) \left\vert F^{\prime }\right\vert _{\mu }\right\}
\lesssim \sum_{F\in \mathcal{F}}\alpha _{\mathcal{F}}\left( F\right) \sqrt{%
\left\vert F\right\vert _{\mu }}\sqrt{\sum_{F^{\prime }\in \mathcal{F}\left[
F\right] }\alpha _{\mathcal{F}}\left( F^{\prime }\right) ^{2}\left\vert
F^{\prime }\right\vert _{\mu }} \\
&\lesssim &\left( \sum_{F\in \mathcal{F}}\alpha _{\mathcal{F}}\left(
F\right) ^{2}\left\vert F\right\vert _{\mu }\right) ^{\frac{1}{2}}\left(
\sum_{F\in \mathcal{F}}\sum_{F^{\prime }\in \mathcal{F}\left[ F\right]
}\alpha _{\mathcal{F}}\left( F^{\prime }\right) ^{2}\left\vert F^{\prime
}\right\vert _{\mu }\right) ^{\frac{1}{2}}\lesssim \left\Vert f\right\Vert
_{L^{2}\left( \mu \right) }\left( \sum_{F^{\prime }\in \mathcal{F}}\alpha _{%
\mathcal{F}}\left( F^{\prime }\right) ^{2}\left\vert F^{\prime }\right\vert
_{\mu }\right) ^{\frac{1}{2}}\lesssim \left\Vert f\right\Vert _{L^{2}\left(
\mu \right) }^{2}.
\end{eqnarray*}%
This proves (\ref{q orth}) since $\left\Vert \sum_{F\in \mathcal{F}}\alpha _{%
\mathcal{F}}\left( F\right) \mathbf{1}_{F}\right\Vert _{L^{2}\left( \mu
\right) }^{2}$ is dominated by twice the left hand side above.

We now adapt this last argument to apply to (\ref{Int}). For example, in the
case $m=2$, we have that%
\begin{eqnarray*}
&&\func{Int}\left( 2\right) =\sum_{F_{4}\in \mathcal{F}}\alpha _{\mathcal{F}%
}\left( F_{4}\right) \sum_{F_{3}\subset F_{4}}\alpha _{\mathcal{F}}\left(
F_{3}\right) \sum_{F_{2}\subset F_{3}}\alpha _{\mathcal{F}}\left(
F_{2}\right) \sum_{F_{1}\subset F_{2}}\alpha _{\mathcal{F}}\left(
F_{1}\right) \left\vert F_{1}\right\vert _{\mu } \\
&=&\sum_{F_{4}\in \mathcal{F}}\alpha _{\mathcal{F}}\left( F_{4}\right)
\left( \sum_{n_{3}=0}^{\infty }\sum_{F_{3}\in \mathfrak{C}_{\mathcal{F}%
}^{\left( n_{3}\right) }\left( F_{4}\right) }\alpha _{\mathcal{F}}\left(
F_{3}\right) \left( \sum_{n_{2}=0}^{\infty }\sum_{F_{2}\in \mathfrak{C}_{%
\mathcal{F}}^{\left( n_{2}\right) }\left( F_{3}\right) }\alpha _{\mathcal{F}%
}\left( F_{2}\right) \left( \sum_{n_{1}=0}^{\infty }\sum_{F_{1}\in \mathfrak{%
C}_{\mathcal{F}}^{\left( n_{1}\right) }\left( F_{2}\right) }\alpha _{%
\mathcal{F}}\left( F_{1}\right) \left\vert F_{1}\right\vert _{\mu }\right)
\right) \right)
\end{eqnarray*}%
which is at most (we continue to write $m$ in place of $2$ until the very
end of the argument)%
\begin{eqnarray*}
&&C_{\delta }\sum_{n_{3}=0}^{\infty }\sum_{n_{2}=0}^{\infty
}\sum_{n_{1}=0}^{\infty }\sum_{F_{4}\in \mathcal{F}}\alpha _{\mathcal{F}%
}\left( F_{4}\right) \sum_{F_{3}\in \mathfrak{C}_{\mathcal{F}}^{\left(
n_{3}\right) }\left( F_{4}\right) }\alpha _{\mathcal{F}}\left( F_{3}\right)
\\
&&\ \ \ \ \ \ \ \ \ \ \ \ \ \ \ \ \ \ \ \ \ \ \ \ \ \ \ \ \ \ \times
\sum_{F_{2}\in \mathfrak{C}_{\mathcal{F}}^{\left( n_{2}\right) }\left(
F_{3}\right) }\alpha _{\mathcal{F}}\left( F_{2}\right) \left( 2^{-\delta
n_{1}}\left\vert F_{2}\right\vert _{\mu }\right) ^{\frac{2m-1}{2m}}\left(
\sum_{F_{1}\in \mathfrak{C}_{\mathcal{F}}^{\left( n_{1}\right) }\left(
F_{2}\right) }\alpha _{\mathcal{F}}\left( F_{1}\right) ^{2m}\left\vert
F_{1}\right\vert _{\mu }\right) ^{\frac{1}{2m}} \\
&=&C_{\delta }\sum_{n_{3}=0}^{\infty }\sum_{n_{2}=0}^{\infty
}\sum_{n_{1}=0}^{\infty }2^{-\delta \frac{2m-1}{2m}n_{1}}\sum_{F_{4}\in 
\mathcal{F}}\alpha _{\mathcal{F}}\left( F_{4}\right) \sum_{F_{3}\in 
\mathfrak{C}_{\mathcal{F}}^{\left( n_{3}\right) }\left( F_{4}\right) }\alpha
_{\mathcal{F}}\left( F_{3}\right) \\
&&\ \ \ \ \ \ \ \ \ \ \ \ \ \ \ \ \ \ \ \ \ \ \ \ \ \ \ \ \ \ \times
\sum_{F_{2}\in \mathfrak{C}_{\mathcal{F}}^{\left( n_{2}\right) }\left(
F_{3}\right) }\alpha _{\mathcal{F}}\left( F_{2}\right) \left\vert
F_{2}\right\vert _{\mu }^{\frac{1}{2m}}\left( \sum_{F_{1}\in \mathfrak{C}_{%
\mathcal{F}}^{\left( n_{1}\right) }\left( F_{2}\right) }\alpha _{\mathcal{F}%
}\left( F_{1}\right) ^{2m}\left\vert F_{1}\right\vert _{\mu }\right) ^{\frac{%
1}{2m}}\left\vert F_{2}\right\vert _{\mu }^{1-\frac{2}{2m}},
\end{eqnarray*}%
which is in turn dominated by%
\begin{eqnarray*}
&&C_{\delta }\sum_{n_{3}=0}^{\infty }\sum_{n_{2}=0}^{\infty
}\sum_{n_{1}=0}^{\infty }2^{-\delta \frac{2m-1}{2m}n_{1}}\sum_{F_{4}\in 
\mathcal{F}}\alpha _{\mathcal{F}}\left( F_{4}\right) \sum_{F_{3}\in 
\mathfrak{C}_{\mathcal{F}}^{\left( n_{3}\right) }\left( F_{4}\right) }\alpha
_{\mathcal{F}}\left( F_{3}\right) \\
&&\times \left( \sum_{F_{2}\in \mathfrak{C}_{\mathcal{F}}^{\left(
n_{2}\right) }\left( F_{3}\right) }\alpha _{\mathcal{F}}\left( F_{2}\right)
^{2m}\left\vert F_{2}\right\vert _{\mu }\right) ^{\frac{1}{2m}}\left(
\sum_{F_{2}\in \mathfrak{C}_{\mathcal{F}}^{\left( n_{2}\right) }\left(
F_{3}\right) }\sum_{F_{1}\in \mathfrak{C}_{\mathcal{F}}^{\left( n_{1}\right)
}\left( F_{2}\right) }\alpha _{\mathcal{F}}\left( F_{1}\right)
^{2m}\left\vert F_{1}\right\vert _{\mu }\right) ^{\frac{1}{2m}}\left(
2^{-\delta n_{2}}\left\vert F_{3}\right\vert _{\mu }\right) ^{\frac{2m-2}{2m}%
},
\end{eqnarray*}%
where in the last line we have applied H\"{o}lder's inequality with
exponents $\left( 2m,2m,\frac{2m}{2m-2}\right) $, and then used that $%
\sum_{F_{2}\in \mathfrak{C}_{\mathcal{F}}^{\left( n_{2}\right) }\left(
F_{3}\right) }\left\vert F_{2}\right\vert _{\mu }\leq C_{\delta }2^{-\delta
n_{2}}\left\vert F_{3}\right\vert _{\mu }$.

Continuing in this way, we dominate the sum above by%
\begin{eqnarray*}
&\lesssim &\sum_{n_{3}=0}^{\infty }\sum_{n_{2}=0}^{\infty
}\sum_{n_{1}=0}^{\infty }2^{-\delta \frac{2m-1}{2m}n_{1}}\sum_{F_{4}\in 
\mathcal{F}}\alpha _{\mathcal{F}}\left( F_{4}\right) \sum_{F_{3}\in 
\mathfrak{C}_{\mathcal{F}}^{\left( n_{3}\right) }\left( F_{4}\right) }\alpha
_{\mathcal{F}}\left( F_{3}\right) \\
&&\times \left( \sum_{F_{2}\in \mathfrak{C}_{\mathcal{F}}^{\left(
n_{2}\right) }\left( F_{3}\right) }\alpha _{\mathcal{F}}\left( F_{2}\right)
^{2m}\left\vert F_{2}\right\vert _{\mu }\right) ^{\frac{1}{2m}}\left(
\sum_{F_{2}\in \mathfrak{C}_{\mathcal{F}}^{\left( n_{2}\right) }\left(
F_{3}\right) }\sum_{F_{1}\in \mathfrak{C}_{\mathcal{F}}^{\left( n_{1}\right)
}\left( F_{2}\right) }\alpha _{\mathcal{F}}\left( F_{1}\right)
^{2m}\left\vert F_{1}\right\vert _{\mu }\right) ^{\frac{1}{2m}}\left(
2^{-\delta n_{2}}\left\vert F_{3}\right\vert _{\mu }\right) ^{1-\frac{2}{2m}}
\\
&=&\sum_{n_{3}=0}^{\infty }\sum_{n_{2}=0}^{\infty }\sum_{n_{1}=0}^{\infty
}2^{-\delta \left( 1-\frac{1}{2m}\right) n_{1}-\delta \left( 1-\frac{2}{2m}%
\right) n_{2}}\sum_{F_{4}\in \mathcal{F}}\alpha _{\mathcal{F}}\left(
F_{4}\right) \\
&&\times \sum_{F_{3}\in \mathfrak{C}_{\mathcal{F}}^{\left( n_{3}\right)
}\left( F_{4}\right) }\alpha _{\mathcal{F}}\left( F_{3}\right) \left\vert
F_{3}\right\vert _{\mu }^{\frac{1}{2m}}\left( \sum_{F_{2}\in \mathfrak{C}_{%
\mathcal{F}}^{\left( n_{2}\right) }\left( F_{3}\right) }\alpha _{\mathcal{F}%
}\left( F_{2}\right) ^{2m}\left\vert F_{2}\right\vert _{\mu }\right) ^{\frac{%
1}{2m}} \\
&&\times \left( \sum_{F_{2}\in \mathfrak{C}_{\mathcal{F}}^{\left(
n_{2}\right) }\left( F_{3}\right) }\sum_{F_{1}\in \mathfrak{C}_{\mathcal{F}%
}^{\left( n_{1}\right) }\left( F_{2}\right) }\alpha _{\mathcal{F}}\left(
F_{1}\right) ^{2m}\left\vert F_{1}\right\vert _{\mu }\right) ^{\frac{1}{2m}%
}\left\vert F_{3}\right\vert _{\mu }^{1-\frac{3}{2m}}
\end{eqnarray*}%
and continuing with $\frac{2m-4}{2m}=0$ for $m=2$, we have the upper bound,%
\begin{eqnarray*}
&&\sum_{n_{3}=0}^{\infty }\sum_{n_{2}=0}^{\infty }\sum_{n_{1}=0}^{\infty
}2^{-\delta \left[ \left( 1-\frac{1}{2m}\right) n_{1}+\left( 1-\frac{2}{2m}%
\right) n_{2}+\left( 1-\frac{3}{2m}\right) n_{3}\right] }\sum_{F_{4}\in 
\mathcal{F}}\alpha _{\mathcal{F}}\left( F_{4}\right) \left\vert
F_{4}\right\vert _{\mu }^{\frac{1}{2m}}\left( \sum_{F_{3}\in \mathfrak{C}_{%
\mathcal{F}}^{\left( n_{3}\right) }\left( F_{4}\right) }\alpha _{\mathcal{F}%
}\left( F_{3}\right) ^{2m}\left\vert F_{3}\right\vert _{\mu }\right) ^{\frac{%
1}{2m}} \\
&&\times \left( \sum_{F_{3}\in \mathfrak{C}_{\mathcal{F}}^{\left(
n_{3}\right) }\left( F_{4}\right) }\sum_{F_{2}\in \mathfrak{C}_{\mathcal{F}%
}^{\left( n_{2}\right) }\left( F_{3}\right) }\alpha _{\mathcal{F}}\left(
F_{2}\right) ^{2m}\left\vert F_{2}\right\vert _{\mu }\right) ^{\frac{1}{2m}}
\\
&&\times \left( \sum_{F_{3}\in \mathfrak{C}_{\mathcal{F}}^{\left(
n_{3}\right) }\left( F_{4}\right) }\sum_{F_{2}\in \mathfrak{C}_{\mathcal{F}%
}^{\left( n_{2}\right) }\left( F_{3}\right) }\sum_{F_{1}\in \mathfrak{C}_{%
\mathcal{F}}^{\left( n_{1}\right) }\left( F_{2}\right) }\alpha _{\mathcal{F}%
}\left( F_{1}\right) ^{2m}\left\vert F_{1}\right\vert _{\mu }\right) ^{\frac{%
1}{2m}}\left\vert F_{4}\right\vert _{\mu }^{\frac{2m-4}{2m}},
\end{eqnarray*}%
which is at most%
\begin{eqnarray*}
&&\sum_{n_{3}=0}^{\infty }\sum_{n_{2}=0}^{\infty }\sum_{n_{1}=0}^{\infty
}2^{-\delta \left[ \left( 1-\frac{1}{2m}\right) n_{1}+\left( 1-\frac{2}{2m}%
\right) n_{2}+\left( 1-\frac{3}{2m}\right) n_{3}\right] }\left(
\sum_{F_{4}\in \mathcal{F}}\alpha _{\mathcal{F}}\left( F_{4}\right)
^{2m}\left\vert F_{4}\right\vert _{\mu }\right) ^{\frac{1}{2m}} \\
&&\times \left( \sum_{F_{4}\in \mathcal{F}}\sum_{F_{3}\in \mathfrak{C}_{%
\mathcal{F}}^{\left( n_{3}\right) }\left( F_{4}\right) }\alpha _{\mathcal{F}%
}\left( F_{3}\right) ^{2m}\left\vert F_{3}\right\vert _{\mu }\right) ^{\frac{%
1}{2m}}\left( \sum_{F_{4}\in \mathcal{F}}\sum_{F_{3}\in \mathfrak{C}_{%
\mathcal{F}}^{\left( n_{3}\right) }\left( F_{4}\right) }\sum_{F_{2}\in 
\mathfrak{C}_{\mathcal{F}}^{\left( n_{2}\right) }\left( F_{3}\right) }\alpha
_{\mathcal{F}}\left( F_{2}\right) ^{2m}\left\vert F_{2}\right\vert _{\mu
}\right) ^{\frac{1}{2m}} \\
&&\times \left( \sum_{F_{4}\in \mathcal{F}}\sum_{F_{3}\in \mathfrak{C}_{%
\mathcal{F}}^{\left( n_{3}\right) }\left( F_{4}\right) }\sum_{F_{2}\in 
\mathfrak{C}_{\mathcal{F}}^{\left( n_{2}\right) }\left( F_{3}\right)
}\sum_{F_{1}\in \mathfrak{C}_{\mathcal{F}}^{\left( n_{1}\right) }\left(
F_{2}\right) }\alpha _{\mathcal{F}}\left( F_{1}\right) ^{2m}\left\vert
F_{1}\right\vert _{\mu }\right) ^{\frac{1}{2m}}.
\end{eqnarray*}%
Finally, since 
\begin{eqnarray*}
\sum_{F_{4}\in \mathcal{F}}\sum_{F_{3}\in \mathfrak{C}_{\mathcal{F}}^{\left(
n_{3}\right) }\left( F_{4}\right) }\sum_{F_{2}\in \mathfrak{C}_{\mathcal{F}%
}^{\left( n_{2}\right) }\left( F_{3}\right) }\sum_{F_{1}\in \mathfrak{C}_{%
\mathcal{F}}^{\left( n_{1}\right) }\left( F_{2}\right) }\alpha _{\mathcal{F}%
}\left( F_{1}\right) ^{2m}\left\vert F_{1}\right\vert _{\mu } &\leq
&\sum_{F\in \mathcal{F}}\alpha _{\mathcal{F}}\left( F\right) ^{2m}\left\vert
F\right\vert _{\mu }, \\
\sum_{F_{4}\in \mathcal{F}}\sum_{F_{3}\in \mathfrak{C}_{\mathcal{F}}^{\left(
n_{3}\right) }\left( F_{4}\right) }\sum_{F_{2}\in \mathfrak{C}_{\mathcal{F}%
}^{\left( n_{2}\right) }\left( F_{3}\right) }\alpha _{\mathcal{F}}\left(
F_{2}\right) ^{2m}\left\vert F_{2}\right\vert _{\mu } &\leq &\sum_{F\in 
\mathcal{F}}\alpha _{\mathcal{F}}\left( F\right) ^{2m}\left\vert
F\right\vert _{\mu }, \\
\sum_{F_{4}\in \mathcal{F}}\sum_{F_{3}\in \mathfrak{C}_{\mathcal{F}}^{\left(
n_{3}\right) }\left( F_{4}\right) }\alpha _{\mathcal{F}}\left( F_{3}\right)
^{2m}\left\vert F_{3}\right\vert _{\mu } &\leq &\sum_{F\in \mathcal{F}%
}\alpha _{\mathcal{F}}\left( F\right) ^{2m}\left\vert F\right\vert _{\mu },
\end{eqnarray*}%
we obtain that $\func{Int}\left( 2\right) $ is dominated by%
\begin{equation*}
\sum_{n_{3}=0}^{\infty }\sum_{n_{2}=0}^{\infty }\sum_{n_{1}=0}^{\infty
}2^{-\delta \left[ \left( 1-\frac{1}{2m}\right) n_{1}+\left( 1-\frac{2}{2m}%
\right) n_{2}+\left( 1-\frac{3}{2m}\right) n_{3}\right] }\sum_{F\in \mathcal{%
F}}\alpha _{\mathcal{F}}\left( F\right) ^{2m}\left\vert F\right\vert _{\mu
}=C_{\delta ,p}\sum_{F\in \mathcal{F}}\alpha _{\mathcal{F}}\left( F\right)
^{2m}\left\vert F\right\vert _{\mu }\ .
\end{equation*}%
This together with (\ref{Int}), proves%
\begin{equation*}
\int_{\mathbb{R}^{n}}\left\vert \alpha _{\mathcal{F}}\left( x\right)
\right\vert _{\ell ^{2}}^{4}d\mu \left( x\right) \lesssim \sum_{F\in 
\mathcal{F}}\alpha _{\mathcal{F}}\left( F\right) ^{4}\left\vert F\right\vert
_{\mu }\ .
\end{equation*}

Similarly, we can show for $m\geq 3$ that%
\begin{equation*}
\int_{\mathbb{R}^{n}}\left\vert \alpha _{\mathcal{F}}\left( x\right)
\right\vert _{\ell ^{2}}^{2m}d\mu \left( x\right) \lesssim \sum_{F\in 
\mathcal{F}}\alpha _{\mathcal{F}}\left( F\right) ^{2m}\left\vert
F\right\vert _{\mu }.
\end{equation*}%
Altogether then we have%
\begin{equation*}
\int_{\mathbb{R}}\left\vert \alpha _{\mathcal{F}}\left( x\right) \right\vert
_{\ell ^{2}}^{p}d\mu \left( x\right) \lesssim \sum_{F\in \mathcal{F}}\alpha
_{\mathcal{F}}\left( F\right) ^{p}\left\vert F\right\vert _{\mu },\ \ \ \ \ 
\text{for }p\in \left( 0,2\right] \cup \left\{ 2m\right\} _{m\in \mathbb{N}%
}\ ,
\end{equation*}%
where $\alpha _{\mathcal{F}}\left( x\right) \equiv \left\{ \alpha _{\mathcal{%
F}}\left( F\right) \mathbf{1}_{F}\left( x\right) \right\} _{F\in \mathcal{F}%
} $. Marcinkiewicz interpolation \cite[Theorem 1.18 on page 480]{GaRu}
applied with the linear operator taking sequences of numbers $\left\{ \alpha
_{\mathcal{F}}\left( F\right) \right\} _{F\in \mathcal{F}}\in \ell
^{p}\left( \mathcal{F},\left\vert F\right\vert _{\mu }\right) $ to sequences
of functions $\left\{ \alpha _{\mathcal{F}}\left( F\right) \mathbf{1}%
_{F}\left( x\right) \right\} _{F\in \mathcal{F}}\in L^{p}\left( \ell
^{2};\omega \right) $, now gives this inequality for all $1<p<\infty $, and
this completes the proof of (\ref{obs}), which is the inequality in (\ref%
{The claim}).

For the reverse inequality when $2\leq p<\infty $, we have with $\alpha _{%
\mathcal{F}}\left( x\right) =\alpha _{\mathcal{F}}^{0}\left( x\right) $ that%
\begin{eqnarray*}
&&\int_{\mathbb{R}^{n}}\left\vert \alpha _{\mathcal{F}}\left( x\right)
\right\vert _{\ell ^{2}}^{p}d\mu \left( x\right) =\int_{\mathbb{R}%
^{n}}\left( \sum_{F\in \mathcal{F}}\left\vert \alpha _{\mathcal{F}}\left(
F\right) \mathbf{1}_{F}\left( x\right) \right\vert ^{2}\right) ^{\frac{p}{2}%
}d\mu \left( x\right) \\
&\gtrsim &\int_{\mathbb{R}^{n}}\sum_{F\in \mathcal{F}}\left\vert \alpha _{%
\mathcal{F}}\left( F\right) \mathbf{1}_{F}\left( x\right) \right\vert
^{p}d\mu \left( x\right) =\sum_{F\in \mathcal{F}}\alpha _{\mathcal{F}}\left(
F\right) ^{p}\left\vert F\right\vert _{\mu }.
\end{eqnarray*}
\end{proof}

\section{Square functions and vector-valued inequalities}

Recall that the Haar square function%
\begin{equation*}
\mathcal{S}_{\limfunc{Haar}}f\left( x\right) \equiv \left( \sum_{I\in 
\mathcal{D}}\left\vert \bigtriangleup _{I}^{\mu }f\left( x\right)
\right\vert ^{2}\right) ^{\frac{1}{2}}
\end{equation*}%
is bounded on $L^{p}\left( \mu \right) $ for any $1<p<\infty $ and any
locally finite positive Borel measure $\mu $ - simply because $\mathcal{S}_{%
\limfunc{Haar}}$ is the martingale difference square function of an $L^{p}$
bounded martingale. We now extend this result to more complicated square
functions.

Fix a $\mathcal{D}$-dyadic cube $F_{0}$, let $\mu $\ be a locally finite
positive Borel measure on $F_{0}$, and suppose that $\mathcal{F}$ is a
subset of $\mathcal{D}_{F_{0}}\equiv \left\{ I\in \mathcal{D}:I\subset
F_{0}\right\} $. We say that $F^{\prime }\in \mathcal{F}$ is an $\mathcal{F}$%
-child of $F$ if $F^{\prime }\subset F$, and is maximal with respect to this
inclusion. The collection $\left\{ \mathcal{C}_{F}\right\} _{F\in \mathcal{F}%
}$ of subsets $\mathcal{C}_{F}\subset \mathcal{D}_{F_{0}}$ is defined by%
\begin{equation*}
\mathcal{C}_{F}\equiv \left\{ I\in \mathcal{D}:I\subset F\text{ and }%
I\not\subset F^{\prime }\text{ for any }\mathcal{F}\text{-child }F^{\prime }%
\text{ of }F\right\} ,\ \ \ \ \ F\in \mathcal{F},
\end{equation*}%
and satisfy the properties%
\begin{eqnarray*}
&&\mathcal{C}_{F}\text{ is connected for each }F\in \mathcal{F}, \\
&&F\in \mathcal{C}_{F}\text{ and }I\in \mathcal{C}_{F}\Longrightarrow
I\subset F\text{ for each }F\in \mathcal{F}, \\
&&\mathcal{C}_{F}\cap \mathcal{C}_{F^{\prime }}=\emptyset \text{ for all
distinct }F,F^{\prime }\in \mathcal{F}, \\
&&\mathcal{D}_{F_{0}}=\bigcup_{F\in \mathcal{F}}\mathcal{C}_{F}\ .
\end{eqnarray*}%
The subset $\mathcal{C}_{F}$ of $\mathcal{D}$ is referred to as the $%
\mathcal{F}$-corona with top $F$. Define the Haar corona projections $%
\mathsf{P}_{\mathcal{C}_{F}}^{\mu }\equiv \sum_{I\in \mathcal{C}%
_{F}}\bigtriangleup _{I}^{\mu }$ and group them together according to their
depth in the tree $\mathcal{F}$ into the projections%
\begin{equation*}
\mathsf{P}_{k}^{\mu }\equiv \sum_{F\in \mathfrak{C}_{\mathcal{F}}^{k}\left(
F_{0}\right) }\mathsf{P}_{\mathcal{C}_{F}}^{\mu }\ .
\end{equation*}%
Note that the $k^{th}$ grandchildren $F\in \mathfrak{C}_{\mathcal{F}%
}^{k}\left( F_{0}\right) $ are pairwise disjoint and hence so are the
supports of the functions $\mathsf{P}_{\mathcal{C}_{F}}^{\mu }f$ for $F\in 
\mathfrak{C}_{\mathcal{F}}^{k}\left( F_{0}\right) $. Define the $\mathcal{F}$%
-square function $\mathcal{S}_{\mathcal{F}}f$ by%
\begin{equation*}
\mathcal{S}_{\mathcal{F}}f\left( x\right) =\left( \sum_{k=0}^{\infty
}\left\vert \mathsf{P}_{k}^{\mu }f\left( x\right) \right\vert ^{2}\right) ^{%
\frac{1}{2}}=\left( \sum_{F\in \mathcal{F}}\left\vert \mathsf{P}_{\mathcal{C}%
_{F}}^{\mu }f\left( x\right) \right\vert ^{2}\right) ^{\frac{1}{2}}=\left(
\sum_{F\in \mathcal{F}}\left\vert \sum_{I\in \mathcal{C}_{F}}\bigtriangleup
_{I}^{\mu }f\left( x\right) \right\vert ^{2}\right) ^{\frac{1}{2}}.
\end{equation*}

Now note that the sequence $\left\{ \mathsf{P}_{k}^{\mu }f\left( x\right)
\right\} _{F\in \mathcal{F}}$ of functions is the \emph{martingale
difference sequence} of the $L^{p}$ bounded martingale $\left\{ \mathsf{E}%
_{k}^{\mu }f\left( x\right) \right\} _{F\in \mathcal{F}}$ with respect to
the increasing sequence $\left\{ \mathcal{E}_{k}\right\} _{k=0}^{\infty }$
of $\sigma $-algebras, where $\mathcal{E}_{k}$ is the $\sigma $-algebra
generated by the `atoms' $F\in \mathfrak{C}_{\mathcal{F}}^{\left( k\right)
}\left( F_{0}\right) $, i.e.%
\begin{equation*}
\mathcal{E}_{k}\equiv \left\{ E\text{ Borel }\subset F_{0}:E\cap F\in
\left\{ \emptyset ,F\right\} \text{ for all }F\in \mathfrak{C}_{\mathcal{F}%
}^{\left( k\right) }\left( F_{0}\right) \right\} ,
\end{equation*}%
and where 
\begin{eqnarray*}
\mathsf{E}_{k}^{\mu }f\left( x\right) &\equiv &\left\{ 
\begin{array}{ccc}
E_{F}^{\mu }f & \text{ if } & x\in F\text{ for some }F\in \mathfrak{C}_{%
\mathcal{F}}^{\left( k\right) }\left( F_{0}\right) \\ 
f\left( x\right) & \text{ if } & x\in F_{0}\setminus \bigcup \mathfrak{C}_{%
\mathcal{F}}^{\left( k\right) }\left( F_{0}\right)%
\end{array}%
\right. ; \\
&\text{ }&\text{where }\bigcup \mathfrak{C}_{\mathcal{F}}^{\left( k\right)
}\left( F_{0}\right) \equiv \bigcup_{F\in \mathfrak{C}_{\mathcal{F}}^{\left(
k\right) }\left( F_{0}\right) }F.
\end{eqnarray*}%
Indeed, if $E\in \mathcal{E}_{k-1}$, then 
\begin{eqnarray*}
&&\int_{E}\mathsf{E}_{k}^{\mu }f\left( x\right) d\mu \left( x\right)
=\int_{E\setminus \bigcup \mathfrak{C}_{\mathcal{F}}^{\left( k\right)
}\left( F_{0}\right) }\mathsf{E}_{k}^{\mu }f\left( x\right) d\mu \left(
x\right) +\sum_{F\in \mathfrak{C}_{\mathcal{F}}^{\left( k\right) }\left(
F_{0}\right) :\ F\subset E}\int_{F}\mathsf{E}_{k}^{\mu }f\left( x\right)
d\mu \left( x\right) \\
&=&\int_{E\setminus \bigcup \mathfrak{C}_{\mathcal{F}}^{\left( k-1\right)
}\left( F_{0}\right) }f\left( x\right) d\mu \left( x\right) +\sum_{F\in 
\mathfrak{C}_{\mathcal{F}}^{\left( k-1\right) }\left( F_{0}\right) :\
F\subset E}\int_{F\setminus \bigcup \mathfrak{C}_{\mathcal{F}}^{\left(
k\right) }\left( F_{0}\right) }f\left( x\right) d\mu \left( x\right)
+\sum_{F^{\prime }\in \mathfrak{C}_{\mathcal{F}}^{\left( k\right) }\left(
F_{0}\right) :\ F^{\prime }\subset E}\int_{F^{\prime }}f\left( x\right) d\mu
\left( x\right) \\
&=&\int_{E\setminus \bigcup \mathfrak{C}_{\mathcal{F}}^{\left( k-1\right)
}\left( F_{0}\right) }\mathsf{E}_{k-1}^{\mu }f\left( x\right) d\mu \left(
x\right) +\sum_{F\in \mathfrak{C}_{\mathcal{F}}^{\left( k-1\right) }\left(
F_{0}\right) :\ F\subset E}\int_{F}f\left( x\right) d\mu \left( x\right) \\
&=&\int_{E\setminus \bigcup \mathfrak{C}_{\mathcal{F}}^{\left( k-1\right)
}\left( F_{0}\right) }\mathsf{E}_{k-1}^{\mu }f\left( x\right) d\mu \left(
x\right) +\sum_{F\in \mathfrak{C}_{\mathcal{F}}^{\left( k-1\right) }\left(
F_{0}\right) :\ F\subset E}\int_{F}\mathsf{E}_{k-1}^{\mu }f\left( x\right)
d\mu \left( x\right) =\int_{E}\mathsf{E}_{k-1}^{\mu }f\left( x\right) d\mu
\left( x\right) ,
\end{eqnarray*}%
shows that $\left\{ \mathsf{E}_{k}^{\mu }f\left( x\right) \right\} _{F\in 
\mathcal{F}}$ is a martingale. Finally, it is easy to check that the Haar
support of the function $\mathsf{P}_{k}^{\mu }f=\mathsf{E}_{k}^{\mu }f-%
\mathsf{E}_{k-1}^{\mu }f$ is precisely $\bigcup_{F\in \mathfrak{C}_{\mathcal{%
F}}^{\left( k\right) }\left( F_{0}\right) }\mathcal{C}_{F}$, the union of
the coronas associated to the $k$-grandchildren of $F_{0}$.

From Burkholder's martingale transform theorem, for a nice treatment see
Hytonen \cite{Hyt2}, we obtain the inequality%
\begin{equation*}
\left\Vert \sum_{k=0}^{\infty }v_{k}\mathsf{P}_{k}^{\mu }f\right\Vert
_{L^{p}\left( \mu \right) }\leq C_{p}\left( \sup_{0\leq k<\infty }\left\vert
v_{k}\right\vert \right) \left\Vert f\right\Vert _{L^{p}\left( \mu \right) },
\end{equation*}%
for all sequences $v_{k}$ of predictable functions. Now we take $v_{k}=\pm 1$
randomly on $\bigtriangleup _{F}^{\mu }f\equiv \mathbf{1}_{F}\mathsf{P}%
_{k}^{\mu }$ for $F\in \mathfrak{C}_{\mathcal{F}}^{\left( k\right) }\left(
F_{0}\right) $, and then an application of Khintchine's inequality, for
which see \cite[Lemma 5.5 page 114]{MuSc} and \cite[Proposition 4.5 page 28]%
{Wol}, allows us to conclude that the square function satisfies the
following $L^{p}\left( \mu \right) $ bound,%
\begin{equation*}
\left\Vert \mathcal{S}_{\mathcal{F}}f\right\Vert _{L^{p}\left( \mu \right)
}\leq C_{p}\left\Vert f\right\Vert _{L^{p}\left( \mu \right) },\ \ \ \ \ 
\text{for all }1<p<\infty .
\end{equation*}

We now note that from this result, we can obtain the square function bounds
we need for the nearby and paraproduct forms treated below, which include
both of the square funcitons $\mathcal{S}_{\mathcal{F}}$ and 
\begin{equation*}
\mathcal{S}_{\mathcal{F}^{\tau -\func{shift}}}f\left( x\right) \equiv \left(
\sum_{F\in \mathcal{F}}\left\vert \mathsf{P}_{\mathcal{C}_{F}^{\mu ,\tau -%
\func{shift}}}^{\mu }f\left( x\right) \right\vert ^{2}\right) ^{\frac{1}{2}}.
\end{equation*}%
Indeed, we first note that if we take $\mathcal{F}=\mathcal{D}_{F_{0}}$,
then we obtain the bound%
\begin{eqnarray*}
\left\Vert \mathcal{S}_{\limfunc{Haar}}f\right\Vert _{L^{p}\left( \mu
\right) } &\leq &C_{p}\left\Vert f\right\Vert _{L^{p}\left( \mu \right) },\
\ \ \ \ \text{for all }1<p<\infty ; \\
\mathcal{S}_{\limfunc{Haar}}f\left( x\right) &\equiv &\left( \sum_{I\in 
\mathcal{D}_{F_{0}}}\left\vert \bigtriangleup _{I}^{\mu }f\left( x\right)
\right\vert ^{2}\right) ^{\frac{1}{2}}.
\end{eqnarray*}%
Then using,%
\begin{equation*}
\mathcal{C}_{F}\setminus \mathcal{C}_{F}^{\mu ,\tau -\func{shift}}\subset 
\mathcal{N}_{F}\text{ and }\mathcal{C}_{F}^{\mu ,\tau -\func{shift}%
}\setminus \mathcal{C}_{F}\subset \bigcup_{F^{\prime }\in \mathfrak{C}_{%
\mathcal{F}}\left( F\right) }\mathcal{N}_{F^{\prime }}\ ,
\end{equation*}%
we conclude that the symmetric difference of $\mathcal{C}_{F}$ and $\mathcal{%
C}_{F}^{\mu ,\tau -\func{shift}}$ is contained in $\mathcal{N}_{F}\cup
\bigcup_{F^{\prime }\in \mathfrak{C}_{\mathcal{F}}\left( F\right) }\mathcal{N%
}_{F^{\prime }}$, where $\mathcal{N}_{F}$ denotes the set of cubes $I$ near $%
F$ in the corona $\mathcal{C}_{F}$, i.e. $\ell \left( I\right) \geq 2^{-\tau
}\ell \left( F\right) $. But since the children $F^{\prime }\in \mathfrak{C}%
_{\mathcal{F}}\left( F\right) $ are pairwise disjoint, and the cardinality
of the nearby sets $\mathcal{N}_{F}$ and $\mathcal{N}_{F^{\prime }}$ are
each $2^{n\tau }$, we see that%
\begin{equation*}
\left\Vert \mathcal{S}_{\mathcal{F}^{\tau -\func{shift}}}f\right\Vert
_{L^{p}\left( \mu \right) }\leq \left\Vert \mathcal{S}_{\mathcal{F}%
}f\right\Vert _{L^{p}\left( \mu \right) }+C_{\tau ,n}\left\Vert \mathcal{S}_{%
\limfunc{Haar}}f\right\Vert _{L^{p}\left( \mu \right) },
\end{equation*}%
since each of the square functions $\mathcal{S}_{\mathcal{F}}$ and $\mathcal{%
S}_{\limfunc{Haar}}$ have already been shown to be bounded on $L^{p}\left(
\mu \right) $. We have thus proved the following theorem.

\begin{theorem}
\label{square thm}Suppose $\mu $ is a locally finite positive Borel measure
on $\mathbb{R}^{n}$. Then for $1<p<\infty $,%
\begin{equation*}
\left\Vert \mathcal{S}_{\mathcal{F}^{\tau -\func{shift}}}f\right\Vert
_{L^{p}\left( \mu \right) }\leq C_{p,\tau }\left\Vert f\right\Vert
_{L^{p}\left( \mu \right) }.
\end{equation*}
\end{theorem}

Another square function that will arise in the nearby and related forms is 
\begin{eqnarray*}
\mathcal{S}_{\rho ,\delta }f\left( x\right) &\equiv &\left( \sum_{I\in 
\mathcal{D}\ :x\in I}\left\vert \mathsf{P}_{I}^{\rho ,\delta }f\left(
x\right) \right\vert ^{2}\right) ^{\frac{1}{2}}, \\
\text{where }\mathsf{P}_{I}^{\rho ,\delta }f\left( x\right) &\equiv
&\sum_{J\in \mathcal{D}:\ 2^{-\rho }\ell \left( I\right) \leq \ell \left(
J\right) \leq 2^{\rho }\ell \left( I\right) }2^{-\delta \limfunc{dist}\left(
J,I\right) }\bigtriangleup _{J}^{\mu }f\left( x\right) .
\end{eqnarray*}

\begin{theorem}
\label{square thm nearby}Suppose $\mu $ is a locally finite positive Borel
measure on $\mathbb{R}^{n}$, and let $0<\rho ,\delta <1$. Then for $%
1<p<\infty $,%
\begin{equation*}
\left\Vert \mathcal{S}_{\rho ,\delta }f\right\Vert _{L^{p}\left( \mu \right)
}\leq C_{p,\rho ,\delta }\left\Vert f\right\Vert _{L^{p}\left( \mu \right) }.
\end{equation*}
\end{theorem}

\begin{proof}
It is easy to see that $\mathcal{S}_{\rho ,\delta }f\left( x\right) \leq
C_{\rho ,\delta }\mathcal{S}_{\limfunc{Haar}}f\left( x\right) $, and the
boundedness of $\mathcal{S}_{\rho ,\delta }$ now follows from the
boundedness of the Haar square function $\mathcal{S}_{\limfunc{Haar}}$.
\end{proof}

\subsection{Alpert square functions}

This subsection will not be used in this paper, but we include it due to its
likely use in extensions of the current paper, and its utility in other
situations as well. We extend the Haar square function inequalities to
Alpert square functions that use weighted Alpert wavelets in place of Haar
wavelets, but only for doubling measures now. Recall from \cite{RaSaWi} that
if $\mathbb{E}_{I;\kappa }^{\mu }$ denotes orthogonal projection
in\thinspace $L^{2}\left( \mu \right) $ onto the finite dimensional space of
restrictions to $I$ of polynomials of degree less than $\kappa $, then the
weighted Alpert projection $\bigtriangleup _{I;\kappa }^{\mu }$ is given by 
\begin{equation*}
\bigtriangleup _{I;\kappa }^{\mu }=\left( \sum_{I^{\prime }\in \mathfrak{C}_{%
\mathcal{D}}\left( I\right) }\mathbb{E}_{I^{\prime };\kappa }^{\mu }\right) -%
\mathbb{E}_{I;\kappa }^{\mu }.
\end{equation*}%
These weighted Alpert projections $\left\{ \bigtriangleup _{I;\kappa }^{\mu
}\right\} _{I\in \mathcal{D}}$ are orthogonal and span $L^{2}\left( \mu
\right) $ for measures $\mu $ that are infinite on all dyadic tops, and in
particular for doubling measures, see \cite{RaSaWi} and \cite{AlSaUr2} for
teminology and proofs.

We begin by showing that the Alpert square function%
\begin{equation*}
\mathcal{S}_{\limfunc{Alpert};\kappa }f\left( x\right) \equiv \left(
\sum_{I\in \mathcal{D}}\left\vert \bigtriangleup _{I;\kappa }^{\mu }f\left(
x\right) \right\vert ^{2}\right) ^{\frac{1}{2}}
\end{equation*}%
is bounded on $L^{p}\left( \mu \right) $ for any $1<p<\infty $ and any
doubling measure $\mu $. We thank a referee of a previous version of this
paper, for pointing out to us that it is \textbf{not} the case that $%
\mathcal{S}_{\limfunc{Alpert};\kappa }$ is a martingale difference square
function of an $L^{p}\,$bounded martingale, and so we cannot apply
Burkholder's martingale transform theorem as we did for the Haar square
function. On the other hand, in the case the measure $\mu $ is doubling, the
sequence of projections of Alpert wavelets satisfies all of the properties
needed by Burkholder's proof, as we now demonstrate. For the convenience of
the reader, we repeat Burkholder's beautiful argument, following the
treatment in Hyt\"{o}nen \cite{Hyt2}.

Recall that $\mathcal{D}_{k}\equiv \left\{ Q\in \mathcal{D}:\ell \left(
Q\right) =2^{k}\right\} $ is the tiling of $\mathbb{R}^{n}$ with dyadic
cubes of side length $2^{k}$. For each $k\in \mathbb{Z}$ define the
projections%
\begin{equation*}
\mathsf{P}_{k;\kappa }^{\mu }f\left( x\right) \equiv \sum_{Q\in \mathcal{D}%
_{k}}\mathbb{E}_{Q;\kappa }^{\mu }f
\end{equation*}%
of $f$ onto the linear space of functions whose restrictions to cubes in $%
\mathcal{D}_{k}$ are polynomials of degree at most $\kappa $. While there is
no conditional expectation result in the current setting, we can show the
key inequality needed by appealing to the properties of Alpert projections.
Indeed, we show that the functions $\mathsf{P}_{k;\kappa }^{\mu }$ and $%
\mathsf{P}_{k+1;\kappa }^{\mu }$ have the same integral over all $P\in 
\mathcal{D}_{k}$, and this holds because $\bigtriangleup _{P;\kappa }^{\mu
}f $ has vanishing mean on $P$: 
\begin{eqnarray*}
&&\int_{P}\mathsf{P}_{k+1;\kappa }^{\mu }f\left( x\right) d\mu \left(
x\right) -\int_{P}\mathsf{P}_{k;\kappa }^{\mu }f\left( x\right) d\mu \left(
x\right) =\int_{P}\left( \mathsf{P}_{k+1;\kappa }^{\mu }f\left( x\right) -%
\mathsf{P}_{k;\kappa }^{\mu }f\left( x\right) \right) d\mu \left( x\right) \\
&=&\int_{P}\left( \sum_{Q\in \mathcal{D}_{k+1}}\mathbb{E}_{Q;\kappa }^{\mu
}f-\sum_{Q\in \mathcal{D}_{k}}\mathbb{E}_{Q;\kappa }^{\mu }f\right) d\mu
\left( x\right) =\int_{P}\sum_{Q\in \mathcal{D}_{k+1}:Q\subset P}\left( 
\mathbb{E}_{Q;\kappa }^{\mu }f-\mathbb{E}_{P;\kappa }^{\mu }f\right) d\mu
\left( x\right) \\
&=&\int_{P}\left( \bigtriangleup _{P;\kappa }^{\mu }f\left( x\right) \right)
d\mu \left( x\right) =0.
\end{eqnarray*}%
The $L^{p}$ boundedness $\left\Vert \mathbb{E}_{k;\kappa }^{\mu
}f\right\Vert _{L^{p}\left( \mu \right) }\leq C$ follows easily from the
estimate $\left\Vert \mathbb{E}_{Q;\kappa }^{\mu }f\right\Vert _{\infty
}\lesssim E_{Q}^{\mu }\left\vert f\right\vert $ in (\ref{analogue}) for $%
Q\in \mathcal{D}_{k}$.

We will in fact establish the analogue of Burkholder's martingale tranform
theorem for a new class of what we call $L^{p}\left( \mu \right) $%
-quasimartingales, that share all of the formal properties of martingales
except for the presence of sigma algebras and measurability.

\begin{definition}
We say that $\left\{ f_{k}\left( x\right) \right\} _{k\in \mathbb{Z}}\subset
L^{p}\left( \mu \right) $ is an $L^{p}\left( \mu \right) $\emph{%
-quasimartingale} if there is a collection of $L^{2}$ projections $\left\{ 
\mathbb{E}_{I}^{\mu }\right\} _{I\in \mathcal{F}}$ such that 
\begin{eqnarray}
&&f_{k}\left( x\right) =\sum_{I\in \mathcal{F}_{k}}\mathbb{E}_{I}^{\mu
}f\left( x\right) ,\ \ \ \ \ \ k\in \mathbb{Z}\text{, with convergence }a.e.%
\text{ and in }L^{p}\left( \mu \right) ,  \label{q mart} \\
&&d_{k}\left( x\right) \equiv f_{k}\left( x\right) -f_{k-1}\left( x\right)
=\sum_{I\in \mathcal{D}_{k}}\mathsf{P}_{I}^{\mu }f\left( x\right) ,  \notag
\\
&&\mathsf{P}_{I}^{\mu }\mathsf{P}_{J}^{\mu }=\left\{ 
\begin{array}{ccc}
\mathsf{P}_{I}^{\mu } & \text{ if } & I=J \\ 
0 & \text{ if } & I\not=J%
\end{array}%
\right. ,\ \ \ \ \ I,J\in \mathcal{F},  \notag
\end{eqnarray}%
where $\mathsf{P}_{I}^{\mu }=\sum_{I^{\prime }\in \mathfrak{C}_{\mathcal{F}%
}\left( I\right) }\mathbb{E}_{I^{\prime }}^{\mu }-\mathbf{1}_{I^{\prime }}%
\mathbb{E}_{I}^{\mu }$.
\end{definition}

For convenience, we restrict our attention from now on to the case $\mathcal{%
F}=\mathcal{D}$, but the reader can easily extend the analysis below to the
case of an arbitrary subset $\mathcal{F}\subset \mathcal{D}$. Define%
\begin{equation*}
T_{\beta }f\left( x\right) \equiv \sum_{I\in \mathcal{D}}\beta
_{I}\bigtriangleup _{I;\kappa }^{\mu }f\left( x\right) .
\end{equation*}%
Consider the $L^{p}\left( \mu \right) $-quasimartingale,%
\begin{equation*}
\left\{ \mathsf{P}_{k;\kappa }^{\mu }f\left( x\right) \right\} _{k\in 
\mathbb{Z}}\equiv \left\{ \sum_{Q\in \mathcal{D}_{k}}\mathbb{E}_{Q;\kappa
}^{\mu }f\right\} _{k\in \mathbb{Z}}\equiv \left\{ \sum_{I\in \mathcal{D}%
:\ell \left( I\right) \geq 2^{k}}\bigtriangleup _{I;\kappa }^{\mu }f\left(
x\right) \right\} _{k\in \mathbb{Z}},
\end{equation*}%
and its associated $L^{p}\left( \mu \right) $\emph{-quasimartingale
difference sequence},%
\begin{equation*}
\left\{ \mathsf{d}_{k;\kappa }^{\mu }f\left( x\right) \right\} _{k\in 
\mathbb{Z}}=\left\{ \mathsf{P}_{k;\kappa }^{\mu }f\left( x\right) -\mathsf{P}%
_{k-1;\kappa }^{\mu }f\left( x\right) \right\} _{k\in \mathbb{Z}}\equiv
\left\{ \sum_{I\in \mathcal{D}:\ell \left( I\right) =2^{k}}\bigtriangleup
_{I;\kappa }^{\mu }f\left( x\right) \right\} _{k\in \mathbb{Z}}.
\end{equation*}%
Note that 
\begin{equation*}
\int_{P}\mathsf{P}_{k+1;\kappa }^{\mu }f\left( x\right) d\mu \left( x\right)
-\int_{P}\mathsf{P}_{k;\kappa }^{\mu }f\left( x\right) d\mu \left( x\right)
=\int_{P}\left( \bigtriangleup _{P;\kappa }^{\mu }f\left( x\right) \right)
d\mu \left( x\right) =0,
\end{equation*}%
shows that,%
\begin{equation*}
\int_{P}\mathsf{P}_{k+1;\kappa }^{\mu }T_{\beta }f\left( x\right) d\mu
\left( x\right) -\int_{P}\mathsf{P}_{k;\kappa }^{\mu }T_{\beta }f\left(
x\right) d\mu \left( x\right) =\int_{P}\left( \beta _{P}\bigtriangleup
_{P;\kappa }^{\mu }f\left( x\right) \right) d\mu \left( x\right) =0,
\end{equation*}%
and so $\left\{ \mathsf{P}_{k;\kappa }^{\mu }f\left( x\right) \right\}
_{k\in \mathbb{Z}}$ is also an $L^{p}\left( \mu \right) $-quasimartingale.

\begin{definition}
We say that a sequence of functions $\left\{ v_{k}\right\} _{k\in \mathbb{Z}%
} $ is \emph{predictable} if $v_{k}$ is constant on every cube $Q\in 
\mathcal{D}_{k}$.
\end{definition}

\begin{notation}
In order to conform to the notation used in \cite{Hyt2}, we write 
\begin{eqnarray*}
v_{k}\left( x\right) &=&\sum_{Q\in \mathcal{D}_{k}}\beta _{Q}\mathbf{1}%
_{Q}\left( x\right) , \\
f_{k}\left( x\right) &=&\mathbb{E}_{k;\kappa }^{\mu }f\left( x\right)
=\sum_{Q\in \mathcal{D}_{k}}\mathbb{E}_{Q;\kappa }^{\mu }f\left( x\right) ,
\\
d_{k}\left( x\right) &=&\mathsf{P}_{k;\kappa }^{\mu }f\left( x\right) -%
\mathsf{P}_{k-1;\kappa }^{\mu }f\left( x\right) , \\
\left( T_{\beta }f\right) _{n}\left( x\right) &=&\sum_{Q\in \mathcal{D}_{k}}%
\mathbb{E}_{Q;\kappa }^{\mu }T_{\beta }f\left( x\right) =T_{\beta
}\sum_{Q\in \mathcal{D}_{k}}\mathbb{E}_{Q;\kappa }^{\mu }f\left( x\right) .
\end{eqnarray*}
\end{notation}

\begin{theorem}
\label{quasi thm}Let $\left\{ f_{k}\right\} _{k=0}^{n}$ be an $L^{p}\left(
\mu \right) $-quasimartingale with $\mu $ doubling, and let $\left\{
v_{k}\right\} _{k=0}^{n}$ be a bounded predictable sequence, and define
numbers $\beta _{Q}$ by $v_{k}=\sum_{Q\in \mathcal{D}_{k}}\beta _{Q}\mathbf{1%
}_{Q}$. Then%
\begin{eqnarray*}
\left\Vert \mathbb{E}_{n;\kappa }^{\mu }T_{\beta }f\right\Vert _{L^{p}\left(
\mu \right) } &\leq &C_{p}\left\Vert \mathbb{E}_{n;\kappa }^{\mu
}f\right\Vert _{L^{p}\left( \mu \right) }, \\
\text{and }\left\Vert T_{\beta }f\right\Vert _{L^{p}\left( \mu \right) }
&\leq &C_{p}\left\Vert f\right\Vert _{L^{p}\left( \mu \right) }.
\end{eqnarray*}
\end{theorem}

\begin{proof}
By interpolation and duality, it suffices to show that if the Theorem holds
for some index $p\in \left( 0,\infty \right) $, then it also holds for the
index $2p$. We start with%
\begin{eqnarray*}
\left( T_{\beta }f\right) _{n}^{2} &=&\left( \sum_{k=0}^{n}v_{k}d_{k}\right)
^{2}=\sum_{k=0}^{n}v_{k}^{2}d_{k}^{2}+2\sum_{k=0}^{n}%
\sum_{j=0}^{k-1}v_{j}d_{j}v_{k}d_{k} \\
&=&\sum_{k=0}^{n}v_{k}^{2}d_{k}^{2}+2\sum_{k=0}^{n}\left( T_{\beta }f\right)
_{k-1}v_{k}d_{k},
\end{eqnarray*}%
which gives%
\begin{equation*}
\left\Vert \left( T_{\beta }f\right) _{n}\right\Vert _{L^{2p}\left( \mu
\right) }^{2}=\left\Vert \left( T_{\beta }f\right) _{n}^{2}\right\Vert
_{L^{p}\left( \mu \right) }\leq \left\Vert
\sum_{k=0}^{n}v_{k}^{2}d_{k}^{2}\right\Vert _{L^{p}\left( \mu \right)
}+2\left\Vert \sum_{k=0}^{n}\left( T_{\beta }f\right)
_{k-1}v_{k}d_{k}\right\Vert _{L^{p}\left( \mu \right) }.
\end{equation*}%
Now we write%
\begin{equation*}
\left( T_{\beta }f\right) _{k-1}v_{k}=\frac{\left( T_{\beta }f\right)
_{k-1}v_{k}}{\left( T_{\beta }f\right) _{k-1}^{\ast }}\left( T_{\beta
}f\right) _{k-1}^{\ast }\equiv u_{k}\cdot \left( T_{\beta }f\right)
_{k-1}^{\ast }
\end{equation*}%
where%
\begin{equation*}
\left( T_{\beta }f\right) _{k-1}^{\ast }\equiv \max_{j\leq k-1}\left\vert
\left( T_{\beta }f\right) _{j}\right\vert .
\end{equation*}%
Note that $u_{k}$ is predictable and bounded by $1$, and that $\left(
T_{\beta }f\right) _{k-1}^{\ast }$ is increasing in $k$. Now $\left\{ \left(
T_{\beta }f\right) _{k-1}d_{k}\right\} _{k=1}^{n}$ is also a quasimartingale
difference sequence, and so by our induction hypothesis,%
\begin{equation*}
\left\Vert \sum_{k=0}^{n}\left( T_{\beta }f\right)
_{k-1}v_{k}d_{k}\right\Vert _{L^{p}\left( \mu \right) }=\left\Vert
\sum_{k=0}^{n}u_{k}\cdot \left( T_{\beta }f\right) _{k-1}^{\ast
}d_{k}\right\Vert _{L^{p}\left( \mu \right) }\leq C_{p}\left\Vert
\sum_{k=0}^{n}\left( T_{\beta }f\right) _{k-1}^{\ast }d_{k}\right\Vert
_{L^{p}\left( \mu \right) }.
\end{equation*}%
We now consider the following pointwise estimate using summation by parts,%
\begin{eqnarray*}
\sum_{k=0}^{n}\left( T_{\beta }f\right) _{k-1}^{\ast }d_{k}
&=&\sum_{k=0}^{n}\left( T_{\beta }f\right) _{k-1}^{\ast }\left(
f_{k}-f_{k-1}\right) =\sum_{k=0}^{n}\left( T_{\beta }f\right) _{k-1}^{\ast
}f_{k}-\sum_{k=-1}^{n-1}\left( T_{\beta }f\right) _{k}^{\ast }f_{k} \\
&=&\left( T_{\beta }f\right) _{n-1}^{\ast }f_{n}+\sum_{k=0}^{n-1}\left[
\left( T_{\beta }f\right) _{k-1}^{\ast }-\left( T_{\beta }f\right)
_{k}^{\ast }\right] f_{k}-\left( T_{\beta }f\right) _{-1}^{\ast }f_{0}\ ,
\end{eqnarray*}%
where the final term vanishes by our convention about the quasimartingale at 
$-1$. Recalling that $\left( T_{\beta }f\right) _{k}^{\ast }$ is increasing
in $k$, we have%
\begin{eqnarray*}
\left\vert \sum_{k=0}^{n}\left( T_{\beta }f\right) _{k-1}^{\ast
}d_{k}\right\vert &\leq &\left( T_{\beta }f\right) _{n-1}^{\ast }\left\vert
f_{n}\right\vert +\sum_{k=0}^{n-1}\left[ \left( T_{\beta }f\right)
_{k}^{\ast }-\left( T_{\beta }f\right) _{k-1}^{\ast }\right] \left\vert
f_{k}\right\vert \\
&\leq &\left( T_{\beta }f\right) _{n-1}^{\ast }f_{n}^{\ast
}+\sum_{k=0}^{n-1} \left[ \left( T_{\beta }f\right) _{k}^{\ast }-\left(
T_{\beta }f\right) _{k-1}^{\ast }\right] f_{n}^{\ast } \\
&\leq &\left( T_{\beta }f\right) _{n-1}^{\ast }f_{n}^{\ast }+\left( T_{\beta
}f\right) _{n-1}^{\ast }f_{n}^{\ast }\leq 2\left( T_{\beta }f\right)
_{n}^{\ast }f_{n}^{\ast }\ .
\end{eqnarray*}%
We then have%
\begin{eqnarray*}
&&\left\Vert 2\left( T_{\beta }f\right) _{n}^{\ast }f_{n}^{\ast }\right\Vert
_{L^{p}\left( \mu \right) }\leq 2\left\Vert \left( T_{\beta }f\right)
_{n}^{\ast }\right\Vert _{L^{2p}\left( \mu \right) }\left\Vert f_{n}^{\ast
}\right\Vert _{L^{2p}\left( \mu \right) } \\
&\leq &2\left( A_{p}B_{p}\right) ^{2}\left\Vert \left( T_{\beta }f\right)
_{n}\right\Vert _{L^{2p}\left( \mu \right) }\left\Vert f_{n}\right\Vert
_{L^{p}\left( \mu \right) }\leq 2\left( A_{p}B_{p}\right) ^{2}\left\Vert
\left( T_{\beta }f\right) _{n}\right\Vert _{L^{2p}\left( \mu \right)
}\left\Vert f_{n}\right\Vert _{L^{p}\left( \mu \right) }\ ,
\end{eqnarray*}%
by the dyadic maximal theorem, whose bound is $A_{p}$, i.e.%
\begin{eqnarray*}
f_{k}^{\ast }\left( x\right) &=&\max_{j\leq k-1}\left\vert f_{j}\left(
x\right) \right\vert =\max_{j\leq k-1}\left\vert \sum_{Q\in \mathcal{D}_{j}}%
\mathbb{E}_{Q;\kappa }^{\mu }f\left( x\right) \right\vert \leq \sup_{Q\in 
\mathcal{D}:\ x\in Q}\left\vert \mathbb{E}_{Q;\kappa }^{\mu }f\left(
x\right) \right\vert \\
&\leq &B_{p}\sup_{Q\in \mathcal{D}:\ x\in Q}\frac{1}{\left\vert Q\right\vert
_{\mu }}\int_{Q}\left\vert f\left( x\right) \right\vert d\mu \left( x\right)
=B_{p}M_{\mu }^{\func{dy}}f\left( x\right)
\end{eqnarray*}%
by an inequality in [Saw6] since $\mu $ is doubling, and where $\left\Vert
M_{\mu }^{\func{dy}}f\right\Vert _{L^{p}\left( \mu \right) }\leq
A_{p}\left\Vert f\right\Vert _{L^{p}\left( \mu \right) }$.

So far we have%
\begin{equation*}
\left\Vert \left( T_{\beta }f\right) _{n}\right\Vert _{L^{2p}\left( \mu
\right) }^{2}=\left\Vert \left( T_{\beta }f\right) _{n}^{2}\right\Vert
_{L^{p}\left( \mu \right) }\leq \left\Vert
\sum_{k=0}^{n}v_{k}^{2}d_{k}^{2}\right\Vert _{L^{p}\left( \mu \right)
}+2\cdot 2\left( A_{p}B_{p}\right) ^{2}C_{p}\left\Vert \left( T_{\beta
}f\right) _{n}\right\Vert _{L^{2p}\left( \mu \right) }\left\Vert
f_{n}\right\Vert _{L^{p}\left( \mu \right) }\ .
\end{equation*}%
To bound the remaining term $\left\Vert
\sum_{k=0}^{n}v_{k}^{2}d_{k}^{2}\right\Vert _{L^{p}\left( \mu \right) }$ on
the right, we use $\left\vert v_{k}\right\vert \leq 1$, and then follow some
of the earlier steps in reverse order to obtain,%
\begin{equation*}
\sum_{k=0}^{n}v_{k}^{2}d_{k}^{2}\leq \sum_{k=0}^{n}d_{k}^{2}=\left(
\sum_{k=0}^{n}d_{k}\right)
^{2}-2\sum_{k=0}^{n}\sum_{j=0}^{k-1}d_{j}d_{k}=f_{n}^{2}-2%
\sum_{k=0}^{n}f_{k-1}d_{k}\ .
\end{equation*}%
Thus we have%
\begin{equation*}
\left\Vert \sum_{k=0}^{n}v_{k}^{2}d_{k}^{2}\right\Vert _{L^{p}\left( \mu
\right) }\leq \left\Vert f_{n}^{2}\right\Vert _{L^{p}\left( \mu \right)
}+2\left\Vert \sum_{k=0}^{n}f_{k-1}d_{k}\right\Vert _{L^{p}\left( \mu
\right) },
\end{equation*}%
where the term $\left\Vert \sum_{k=0}^{n}f_{k-1}d_{k}\right\Vert
_{L^{p}\left( \mu \right) }$ is exacty the term $\left\Vert
\sum_{k=0}^{n}f_{k-1}v_{k}d_{k}\right\Vert _{L^{p}\left( \mu \right) }$
handled above, but with $v_{k}\equiv 1$, so that%
\begin{equation*}
\left\Vert \sum_{k=0}^{n}f_{k-1}d_{k}\right\Vert _{L^{p}\left( \mu \right)
}\leq 2\left( A_{p}B_{p}\right) ^{2}C_{p}\left\Vert f_{n}\right\Vert
_{L^{2p}\left( \mu \right) }^{2}\ .
\end{equation*}

Altogether then we have%
\begin{eqnarray*}
\left\Vert \left( T_{\beta }f\right) _{n}\right\Vert _{L^{2p}\left( \mu
\right) }^{2} &\leq &\left\Vert f_{n}\right\Vert _{L^{2p}\left( \mu \right)
}^{2}+4\left( A_{p}B_{p}\right) ^{2}C_{p}\left\Vert f_{n}\right\Vert
_{L^{2p}\left( \mu \right) }^{2}+4\left( A_{p}B_{p}\right)
^{2}C_{p}\left\Vert \left( T_{\beta }f\right) _{n}\right\Vert _{L^{2p}\left(
\mu \right) }\left\Vert f_{n}\right\Vert _{L^{p}\left( \mu \right) } \\
&=&\left\{ \left[ 1+4\left( A_{p}B_{p}\right) ^{2}C_{p}\right] \left\Vert
f_{n}\right\Vert _{L^{2p}\left( \mu \right) }+4\left( A_{p}B_{p}\right)
^{2}C_{p}\left\Vert \left( T_{\beta }f\right) _{n}\right\Vert _{L^{2p}\left(
\mu \right) }\right\} \left\Vert f_{n}\right\Vert _{L^{p}\left( \mu \right)
}\ ,
\end{eqnarray*}%
and a simple divide and conquer argument, considering the dominant summand
inside the braces, finishes the proof of the induction step. Indeed, if $%
\left[ 1+4\left( A_{p}B_{p}\right) ^{2}C_{p}\right] \left\Vert
f_{n}\right\Vert _{L^{2p}\left( \mu \right) }$ dominates, then%
\begin{equation*}
\left\Vert \left( T_{\beta }f\right) _{n}\right\Vert _{L^{2p}\left( \mu
\right) }^{2}\leq 2\left[ 1+4\left( A_{p}B_{p}\right) ^{2}C_{p}\right]
\left\Vert f_{n}\right\Vert _{L^{2p}\left( \mu \right) }^{2},
\end{equation*}%
while if $4\left( A_{p}B_{p}\right) ^{2}C_{p}\left\Vert \left( T_{\beta
}f\right) _{n}\right\Vert _{L^{2p}\left( \mu \right) }$ dominates, then%
\begin{equation*}
\left\Vert \left( T_{\beta }f\right) _{n}\right\Vert _{L^{2p}\left( \mu
\right) }^{2}\leq 4\left( A_{p}B_{p}\right) ^{2}C_{p}\left\Vert \left(
T_{\beta }f\right) _{n}\right\Vert _{L^{2p}\left( \mu \right) }\left\Vert
f_{n}\right\Vert _{L^{p}\left( \mu \right) }.
\end{equation*}%
Altogether then,%
\begin{equation*}
\left\Vert \left( T_{\beta }f\right) _{n}\right\Vert _{L^{2p}\left( \mu
\right) }\leq C_{2p}\left\Vert f_{n}\right\Vert _{L^{2p}\left( \mu \right) },
\end{equation*}%
where%
\begin{equation*}
C_{2p}\equiv \sqrt{2}\max \left\{ \sqrt{1+4\left( A_{p}B_{p}\right) ^{2}C_{p}%
},4\left( A_{p}B_{p}\right) ^{2}C_{p}\right\} ,
\end{equation*}%
which proves the first line in the statement of the theorem, and the second
line then follows by a limiting argument.
\end{proof}

Now we can prove the Alpert square function equivalence for doubling
measures in the standard way.

\begin{definition}
Given a quasimartingale $f\sim \left\{ f_{k}\right\} _{k=0}^{n}$ with
respect to a doubling measure $\mu $, and with difference sequence $\left\{
d_{k}\right\} _{k=0}^{n}$, define the associated square function by,%
\begin{equation*}
\mathcal{S}_{\mu }f\left( x\right) \equiv \sqrt{\sum_{k=0}^{n}d_{k}\left(
x\right) ^{2}}.
\end{equation*}
\end{definition}

\begin{corollary}
Let $\left\{ f_{k}\right\} _{k=0}^{n}$ be an $L^{p}\left( \mu \right) $%
-quasimartingale with respect to a doubling measure $\mu $, and let $%
1<p<\infty $. Then for $1<p<\infty $, we have the square function
equivalence,%
\begin{equation*}
\left\Vert \mathcal{S}_{\mu }f\right\Vert _{L^{p}\left( \mu \right) }\approx
\left\Vert f\right\Vert _{L^{p}\left( \mu \right) },\ \ \ \ \ f\in
L^{p}\left( \mu \right) .
\end{equation*}
\end{corollary}

\begin{proof}
Combining the Theorem \ref{quasi thm} with Khintchine's inequality, yields
the boundedness 
\begin{equation*}
\left\Vert \mathcal{S}_{\mu }f\right\Vert _{L^{p}\left( \mu \right)
}\lesssim \left\Vert f\right\Vert _{L^{p}\left( \mu \right) },
\end{equation*}%
of the square function $\mathcal{S}_{\mu }$.

For $\beta _{I}=\pm 1$, we have%
\begin{eqnarray*}
\left\langle f,g\right\rangle _{\mu } &=&\int_{\mathbb{R}^{n}}\left(
\sum_{I\in \mathcal{D}}\bigtriangleup _{I}^{\mu }f\right) \left( \sum_{J\in 
\mathcal{D}}\bigtriangleup _{J}^{\mu }g\right) d\mu =\int_{\mathbb{R}%
^{n}}\sum_{I\in \mathcal{D}}\left( \bigtriangleup _{I}^{\mu }f\right) \left(
\bigtriangleup _{I}^{\mu }g\right) d\mu \\
&=&\int_{\mathbb{R}^{n}}\sum_{I\in \mathcal{D}}\left( \beta
_{I}\bigtriangleup _{I}^{\mu }f\right) \left( \beta _{I}\bigtriangleup
_{I}^{\mu }g\right) d\mu =\int_{\mathbb{R}^{n}}\left( \sum_{I\in \mathcal{D}%
}\beta _{I}\bigtriangleup _{I}^{\mu }f\right) \left( \sum_{J\in \mathcal{D}%
}\beta _{I}\bigtriangleup _{J}^{\mu }g\right) d\mu =\left\langle T_{\beta
}f,T_{\beta }g\right\rangle _{\mu },
\end{eqnarray*}%
and so duality then gives%
\begin{eqnarray*}
\left\Vert f\right\Vert _{L^{p}\left( \mu \right) } &=&\sup_{\left\Vert
g\right\Vert _{L^{p^{\prime }}\left( \mu \right) }\leq 1}\left\vert
\left\langle f,g\right\rangle _{\mu }\right\vert =\sup_{\left\Vert
g\right\Vert _{L^{p^{\prime }}\left( \mu \right) }\leq 1}\left\vert
\left\langle T_{\beta }f,T_{\beta }g\right\rangle _{\mu }\right\vert \\
&\leq &\left\Vert T_{\beta }f\right\Vert _{L^{p}\left( \mu \right)
}\sup_{\left\Vert g\right\Vert _{L^{p^{\prime }}\left( \mu \right) }\leq
1}\left\Vert T_{\beta }g\right\Vert _{L^{p^{\prime }}\left( \mu \right)
}\leq C_{p^{\prime }}\left\Vert T_{\beta }f\right\Vert _{L^{p}\left( \mu
\right) }\ ,
\end{eqnarray*}%
\emph{independently} of $\beta _{I}=\pm 1$. Another application of
Khintchine's inequality gives the reverse inequality%
\begin{equation*}
\left\Vert f\right\Vert _{L^{p}\left( \mu \right) }\lesssim \left\Vert 
\mathcal{S}_{\mu }f\right\Vert _{L^{p}\left( \mu \right) }.
\end{equation*}
\end{proof}

Define the square functions 
\begin{equation*}
\mathcal{S}_{\mathcal{F}^{\tau -\func{shift}};\kappa }f\left( x\right)
\equiv \left( \sum_{F\in \mathcal{F}}\left\vert \mathsf{P}_{\mathcal{C}%
_{F}^{\mu ,\tau -\func{shift}};\kappa }^{\mu }f\left( x\right) \right\vert
^{2}\right) ^{\frac{1}{2}},
\end{equation*}%
and%
\begin{eqnarray*}
\mathcal{S}_{\rho ,\delta ;\kappa }f\left( x\right) &\equiv &\left(
\sum_{I\in \mathcal{D}\ :x\in I}\left\vert \mathsf{P}_{I;\kappa }^{\rho
,\delta }f\left( x\right) \right\vert ^{2}\right) ^{\frac{1}{2}}, \\
\text{where }\mathsf{P}_{I;\kappa }^{\rho ,\delta }f\left( x\right) &\equiv
&\sum_{J\in \mathcal{D}:\ 2^{-\rho }\ell \left( I\right) \leq \ell \left(
J\right) \leq 2^{\rho }\ell \left( I\right) }2^{-\delta \limfunc{dist}\left(
J,I\right) }\bigtriangleup _{J;\kappa }^{\mu }f\left( x\right) .
\end{eqnarray*}%
Altogether we have the following theorem.

\begin{theorem}
\label{Alpert square thm}Suppose $\mu $ is a doubling measure on $\mathbb{R}%
^{n}$. Then for $\kappa \in \mathbb{N}$, $1<p<\infty $ and $0<\rho ,\delta
<1 $, we have%
\begin{eqnarray*}
&&\left\Vert \mathcal{S}_{\limfunc{Alpert};\kappa }f\right\Vert
_{L^{p}\left( \mu \right) }+\left\Vert \mathcal{S}_{\mathcal{F};\kappa
}f\right\Vert _{L^{p}\left( \mu \right) }+\left\Vert \mathcal{S}_{\mathcal{F}%
^{\tau -\func{shift}};\kappa }f\right\Vert _{L^{p}\left( \mu \right) }\leq
C_{p,n,\kappa ,\tau }\left\Vert f\right\Vert _{L^{p}\left( \mu \right) }, \\
&&\ \ \ \ \ \ \ \ \ \ \ \ \ \ \ \ \ \ \ \ \left\Vert \mathcal{S}_{\rho
,\delta ;\kappa }f\right\Vert _{L^{p}\left( \mu \right) }\leq C_{p,\rho
,\delta ,n,\kappa }\left\Vert f\right\Vert _{L^{p}\left( \mu \right) }.
\end{eqnarray*}
\end{theorem}

\subsection{Vector-valued inequalities}

We begin by reviewing the well-known $\ell ^{2}$-extension of a bounded
linear operator.\ We include the simple proof here as it sheds light on the
nature of the quadratic Muckenhoupt condition, in particular on its
necessity for the norm inequality - namely that one must test the norm
inequality over \emph{all} functions $\mathbf{f}_{\mathbf{u}}$ defined below.

Let $M\in \mathbb{N}$ be a large positive integer that we will send to $%
\infty $ later on. Suppose $T$ is bounded from $L^{p}\left( \sigma \right) $
to $L^{p}\left( \omega \right) $, $0<p<\infty $, and for $\mathbf{f}=\left\{
f_{j}\right\} _{j=1}^{M}$, define%
\begin{equation*}
T\mathbf{f}\equiv \left\{ Tf_{j}\right\} _{j=1}^{M}.\text{ }
\end{equation*}%
For any unit vector $\mathbf{u}$ $=\left( u_{j}\right) _{j=1}^{M}$ in $%
\mathbb{C}^{M}$ define%
\begin{equation*}
\mathbf{f}_{\mathbf{u}}\equiv \left\langle \mathbf{f},\mathbf{u}%
\right\rangle \text{ and }T_{\mathbf{u}}\mathbf{f}\equiv \left\langle T%
\mathbf{f},\mathbf{u}\right\rangle =T\left\langle \mathbf{f},\mathbf{u}%
\right\rangle =T\mathbf{f}_{\mathbf{u}}
\end{equation*}%
where the final equalities follow since $T$ is linear. We have%
\begin{equation*}
\int_{\mathbb{R}^{n}}\left\vert T_{\mathbf{u}}\mathbf{f}\left( x\right)
\right\vert ^{p}d\omega \left( x\right) =\int_{\mathbb{R}^{n}}\left\vert T%
\mathbf{f}_{\mathbf{u}}\left( x\right) \right\vert ^{p}d\omega \left(
x\right) \leq \left\Vert T\right\Vert _{L^{p}\left( \sigma \right)
\rightarrow L^{p}\left( \omega \right) }^{p}\int_{\mathbb{R}^{n}}\left\vert 
\mathbf{f}_{\mathbf{u}}\left( x\right) \right\vert ^{p}d\sigma \left(
x\right) ,
\end{equation*}%
where%
\begin{equation*}
T_{\mathbf{u}}\mathbf{f}\left( x\right) =\left\langle T\mathbf{f}\left(
x\right) ,\mathbf{u}\right\rangle =\left\vert T\mathbf{f}\left( x\right)
\right\vert _{\ell ^{2}}\left\langle \frac{T\mathbf{f}\left( x\right) }{%
\left\vert T\mathbf{f}\left( x\right) \right\vert _{\ell ^{2}}},\mathbf{u}%
\right\rangle =\left\vert T\mathbf{f}\left( x\right) \right\vert _{\ell
^{2}}\ \cos \theta ,
\end{equation*}%
if $\theta $ is the angle between $\frac{T\mathbf{f}\left( x\right) }{%
\left\vert T\mathbf{f}\left( x\right) \right\vert }$ and $\mathbf{u}$ in $%
\mathbb{C}^{M}$. Then using%
\begin{equation*}
\int_{\mathbb{S}^{M-1}}\left\vert \left\langle \mathbf{u},\mathbf{v}%
\right\rangle \right\vert ^{p}d\mathbf{u}=\gamma _{p}\text{ for }\left\Vert 
\mathbf{v}\right\Vert =1,
\end{equation*}%
we have%
\begin{eqnarray*}
&&\int_{\mathbb{S}^{M-1}}\left\{ \int_{\mathbb{R}^{n}}\left\vert T_{\mathbf{u%
}}\mathbf{f}\left( x\right) \right\vert ^{p}d\omega \left( x\right) \right\}
d\mathbf{u}=\int_{\mathbb{R}^{n}}\left\{ \int_{\mathbb{S}^{M-1}}\left\vert
T_{\mathbf{u}}\mathbf{f}\left( x\right) \right\vert ^{p}d\mathbf{u}\right\}
d\omega \left( x\right) \\
&=&\int_{\mathbb{R}^{n}}\left\vert T\mathbf{f}\left( x\right) \right\vert
_{\ell ^{2}}^{p}\left\{ \int_{\mathbb{S}^{M-1}}\left\vert \cos \theta
\right\vert ^{p}d\mathbf{u}\right\} d\omega \left( x\right) =\gamma
_{p}\int_{\mathbb{R}^{n}}\left\vert T\mathbf{f}\left( x\right) \right\vert
_{\ell ^{2}}^{p}d\omega \left( x\right) ,
\end{eqnarray*}%
and similarly,%
\begin{equation*}
\int_{\mathbb{S}^{M-1}}\left\{ \int_{\mathbb{R}^{n}}\left\vert \mathbf{f}_{%
\mathbf{u}}\left( x\right) \right\vert ^{p}d\sigma \left( x\right) \right\} d%
\mathbf{u}=\gamma _{p}\int_{\mathbb{R}^{n}}\left\vert \mathbf{f}\left(
x\right) \right\vert _{\ell ^{2}}^{p}d\sigma \left( x\right) .
\end{equation*}%
Altogether then,%
\begin{eqnarray*}
&&\gamma _{p}\int_{\mathbb{R}^{n}}\left\vert T\mathbf{f}\left( x\right)
\right\vert _{\ell ^{2}}^{p}d\omega \left( x\right) =\int_{\mathbb{S}%
^{M-1}}\left\{ \int_{\mathbb{R}^{n}}\left\vert T_{\mathbf{u}}\mathbf{f}%
\left( x\right) \right\vert ^{p}d\omega \left( x\right) \right\} d\mathbf{u}%
=\int_{\mathbb{S}^{M-1}}\left\{ \int_{\mathbb{R}^{n}}\left\vert T\mathbf{f}_{%
\mathbf{u}}\left( x\right) \right\vert ^{p}d\omega \left( x\right) \right\} d%
\mathbf{u} \\
&\leq &\int_{\mathbb{S}^{M-1}}\left\{ \left\Vert T\right\Vert _{L^{p}\left(
\sigma \right) \rightarrow L^{p}\left( \omega \right) }^{p}\int_{\mathbb{R}%
^{n}}\left\vert \mathbf{f}_{\mathbf{u}}\left( x\right) \right\vert
^{p}d\sigma \left( x\right) \right\} d\mathbf{u}=\gamma _{p}\left\Vert
T\right\Vert _{L^{p}\left( \sigma \right) \rightarrow L^{p}\left( \omega
\right) }^{p}\int_{\mathbb{R}^{n}}\left\vert \mathbf{f}\left( x\right)
\right\vert _{\ell ^{2}}^{p}d\sigma \left( x\right) ,
\end{eqnarray*}%
and upon dividing both sides by $\gamma _{p}$ we conclude that%
\begin{equation*}
\int_{\mathbb{R}^{n}}\left\vert T\mathbf{f}\left( x\right) \right\vert
_{\ell ^{2}}^{p}d\omega \left( x\right) \leq \left\Vert T\right\Vert
_{L^{p}\left( \sigma \right) \rightarrow L^{p}\left( \omega \right)
}^{p}\int_{\mathbb{R}^{n}}\left\vert \mathbf{f}\left( x\right) \right\vert
_{\ell ^{2}}^{p}d\sigma \left( x\right) .
\end{equation*}%
Finally we can let $M\nearrow \infty $ to obtain the desired vector-valued
extension,%
\begin{equation}
\left( \int_{\mathbb{R}^{n}}\left( \sqrt{\sum_{j=1}^{\infty }\left\vert
Tf_{j}\left( x\right) \right\vert ^{2}}\right) ^{p}d\omega \left( x\right)
\right) ^{\frac{1}{p}}\leq \left\Vert T\right\Vert _{L^{p}\left( \sigma
\right) \rightarrow L^{p}\left( \omega \right) }\left( \int_{\mathbb{R}%
^{n}}\left( \sqrt{\sum_{j=1}^{\infty }\left\vert f_{j}\left( x\right)
\right\vert ^{2}}\right) ^{p}d\sigma \left( x\right) \right) ^{\frac{1}{p}}.
\label{in full}
\end{equation}

\section{Necessity of quadratic testing and $A_{p}$ conditions}

We can use the vector-valued inequality (\ref{in full}) to obtain the
necessity of the quadratic testing inequality, namely%
\begin{equation}
\left\Vert \left( \sum_{i=1}^{\infty }\left( a_{i}\mathbf{1}%
_{I_{i}}T_{\sigma }^{\lambda }\mathbf{1}_{I_{i}}\right) ^{2}\right) ^{\frac{1%
}{2}}\right\Vert _{L^{p}\left( \omega \right) }\leq \mathfrak{T}_{T^{\lambda
}}^{\ell ^{2}}\left( \sigma ,\omega \right) \left\Vert \left(
\sum_{i=1}^{\infty }\left( a_{i}\mathbf{1}_{I_{i}}\right) ^{2}\right) ^{%
\frac{1}{2}}\right\Vert _{L^{p}\left( \sigma \right) },
\label{quad cube testing}
\end{equation}%
for the boundedness of $T^{\lambda }$ from $L^{p}\left( \sigma \right) $ to $%
L^{p}\left( \omega \right) $, i.e. $\mathfrak{T}_{T^{\lambda }}^{\ell
^{2}}\left( \sigma ,\omega \right) \lesssim \left\Vert T^{\lambda
}\right\Vert _{L^{p}\left( \sigma \right) \rightarrow L^{p}\left( \omega
\right) }$. Indeed, we simply set $f_{i}\equiv a_{i}\mathbf{1}_{I_{i}}$ in (%
\ref{in full}) to obtain the \emph{global} quadratic testing inequality,%
\begin{equation}
\left\Vert \left( \sum_{i=1}^{\infty }\left( a_{i}T_{\sigma }^{\lambda }%
\mathbf{1}_{I_{i}}\right) ^{2}\right) ^{\frac{1}{2}}\right\Vert
_{L^{p}\left( \omega \right) }\leq \mathfrak{T}_{T^{\lambda },p}^{\ell ^{2},%
\func{global}}\left( \sigma ,\omega \right) \left\Vert \left(
\sum_{i=1}^{\infty }\left( a_{i}\mathbf{1}_{I_{i}}\right) ^{2}\right) ^{%
\frac{1}{2}}\right\Vert _{L^{p}\left( \sigma \right) },
\label{global cube testing}
\end{equation}%
and then we simply note the pointwise inequality%
\begin{equation*}
\sum_{i=1}^{\infty }\left( a_{i}\mathbf{1}_{I_{i}}T_{\sigma }^{\lambda }%
\mathbf{1}_{I_{i}}\right) \left( x\right) ^{2}=\sum_{i=1}^{\infty
}\left\vert a_{i}\right\vert ^{2}\left\vert T_{\sigma }^{\lambda }\mathbf{1}%
_{I_{i}}\left( x\right) \right\vert ^{2}\mathbf{1}_{I_{i}}\left( x\right)
\leq \sum_{i=1}^{\infty }\left\vert a_{i}\right\vert ^{2}\left\vert
T_{\sigma }^{\lambda }\mathbf{1}_{I_{i}}\left( x\right) \right\vert ^{2}%
\mathbf{,}
\end{equation*}%
to obtain the local version (\ref{quad cube testing}).

Now we turn to the necessity of the quadratic offset $A_{p}^{\lambda ,\ell
^{2},\limfunc{offset}}$ condition, namely%
\begin{equation*}
\left\Vert \left( \sum_{i=1}^{\infty }\left( a_{i}\mathbf{1}_{I_{i}^{\ast }}%
\frac{\left\vert I_{i}\right\vert _{\sigma }}{\left\vert I_{i}\right\vert
^{1-\frac{\lambda }{n}}}\right) ^{2}\right) ^{\frac{1}{2}}\right\Vert
_{L^{p}\left( \omega \right) }\leq A_{p}^{\lambda ,\ell ^{2},\limfunc{offset}%
}\left( \sigma ,\omega \right) \left\Vert \left( \sum_{i=1}^{\infty
}\left\vert a_{i}\mathbf{1}_{I_{i}}\right\vert ^{2}\right) ^{\frac{1}{2}%
}\right\Vert _{L^{p}\left( \sigma \right) }.
\end{equation*}%
Suppose that $T^{\lambda }$ is Stein elliptic, and fix appropriate sequences 
$\left\{ I_{i}\right\} _{i=1}^{\infty }$ and $\left\{ a_{i}\right\}
_{i=1}^{\infty }$ of cubes and numbers respectively. Then there is a choice
of constant $C$ and appropriate cubes $I_{i}^{\ast }$ such that%
\begin{equation*}
\left\vert T_{\sigma }^{\lambda }\mathbf{1}_{I_{i}}\left( x\right)
\right\vert \geq c\frac{\left\vert I_{i}\right\vert _{\sigma }}{\left\vert
I_{i}\right\vert ^{1-\frac{\lambda }{n}}}\text{ for }x\in I_{i}^{\ast },\ \
\ \ \ 1\leq i\leq \infty .
\end{equation*}%
Now we simply apply (\ref{global cube testing}) to obtain 
\begin{equation*}
A_{p}^{\lambda ,\ell ^{2},\limfunc{offset}}\left( \sigma ,\omega \right)
\lesssim \mathfrak{T}_{T^{\lambda },p}^{\ell ^{2},\func{global}}\left(
\sigma ,\omega \right) \leq \mathfrak{N}_{T^{\lambda },p}\left( \sigma
,\omega \right) .
\end{equation*}%
It should be noticed that while the necessity of the quadratic Muckenhoupt
condition $\mathcal{A}_{p}^{\lambda ,\ell ^{2}}\left( \sigma ,\omega \right) 
$ itself is easily shown for the Hilbert transform, the necessity for even
nice operators in higher dimensions is much more difficult.

\subsection{Quadratic Haar weak boundedness property}

It is convenient in our proof to introduce the quadratic \emph{Haar weak
boundedness property} constant $\mathcal{HWBP}_{T^{\lambda },p}^{\ell
^{2},\rho }\left( \sigma ,\omega \right) $ as the least constant in the
inequality,%
\begin{equation}
\left\Vert \left( \sum_{I\in \mathcal{D}}\sum_{J\in \func{Adj}_{\rho }\left(
I\right) }\left\vert \bigtriangleup _{J}^{\omega }T_{\sigma }^{\lambda
}\bigtriangleup _{I}^{\sigma }f\right\vert ^{2}\right) ^{\frac{1}{2}%
}\right\Vert _{L^{p}\left( \omega \right) }\leq \mathcal{HWBP}_{T^{\lambda
},p}^{\ell ^{2},\rho }\left( \sigma ,\omega \right) \left\Vert f\right\Vert
_{L^{p}\left( \sigma \right) }.  \label{quad Haar test}
\end{equation}

There is only one quadratic Haar inequality in the weak boundedness
condition (\ref{quad Haar test}), since we show below in Proposition \ref%
{diag form} that (\ref{quad Haar test}) is \emph{equivalent} to the bilinear
inequality%
\begin{equation}
\left\vert \sum_{I\in \mathcal{D}}\sum_{J\in \func{Adj}_{\rho }\left(
I\right) }\left\langle T_{\sigma }^{\lambda }\bigtriangleup _{I}^{\sigma
}f,\bigtriangleup _{J}^{\omega }g\right\rangle _{\omega }\right\vert \leq
C\left\Vert f\right\Vert _{L^{p}\left( \sigma \right) }\left\Vert
g\right\Vert _{L^{p^{\prime }}\left( \omega \right) },\ \ \ \ \ f\in
L^{p}\left( \sigma \right) ,g\in L^{p^{\prime }}\left( \omega \right) ,
\label{bil in}
\end{equation}%
which is then also equivalent to the inequality dual to that appearing in (%
\ref{quad Haar test}). In fact, this bilinear inequality is a `quadratic
analogue' of a scalar weak boundedness property, which points to the
relative `weakness' of (\ref{quad Haar test}). Of course, one can use the $%
L^{\infty }$ estimates (\ref{analogue'}) on Haar wavelets, together with the
vector-valued maximal theorem of Fefferman and Stein in a space of
homogeneous type \cite[Theorem 2.1]{GrLiYa}, to show that for doubling
measures, we actually have $\mathcal{HWBP}_{T^{\lambda },p}^{\ell ^{2},\rho
}\lesssim \mathfrak{T}_{T^{\lambda },p}^{\ell ^{2},\func{global}}$, but we
will instead show a stronger result in Lemma \ref{stronger} below.

The property (\ref{quad Haar test}) appears at first glance to be much
stronger than the corresponding scalar testing and weak boundedness
conditions, mainly because the standard proof of necessity of these
conditions involves testing the scalar norm inequality over a dense set of
functions $\sum_{i=1}^{\infty }u_{i}a_{i}\mathbf{1}_{I_{i}}$ with $%
\sum_{i=1}^{\infty }u_{i}^{2}=1$, see the subsection on vector-valued
inequalities above. However, the minimal nature of the role of the quadratic
Haar weak boundedness property (\ref{quad Haar test}) is demonstrated by
considering the adjacent diagonal bilinear form $\mathsf{B}_{\func{Adj},\rho
}\left( f,g\right) $ associated with the form%
\begin{equation*}
\left\langle T_{\sigma }^{\lambda }f,g\right\rangle _{\omega }=\left\langle
T_{\sigma }^{\lambda }\left( \sum_{I\in \mathcal{D}}\bigtriangleup
_{I}^{\sigma }f\right) ,\sum_{J\in \mathcal{D}}\bigtriangleup _{J}^{\omega
}g\right\rangle _{\omega }=\sum_{I,J\in \mathcal{D}}\left\langle T_{\sigma
}^{\lambda }\bigtriangleup _{I}^{\sigma }f,\bigtriangleup _{J}^{\omega
}g\right\rangle _{\omega }\ ,
\end{equation*}%
where $f=\sum_{I\in \mathcal{D}}\bigtriangleup _{I}^{\sigma }f$ and $%
g=\sum_{J\in \mathcal{D}}\bigtriangleup _{J}^{\omega }g$ are the weighted
Haar expansions of $f$ and $g$ respectively. Here the adjacent diagonal form 
$\mathsf{B}_{\func{Adj},\rho }\left( f,g\right) $ is given by%
\begin{equation*}
\mathsf{B}_{\func{Adj},\rho }\left( f,g\right) \equiv \sum_{I\in \mathcal{D}%
}\sum_{J\in \func{Adj}_{\rho }\left( I\right) }\left\langle T_{\sigma
}^{\lambda }\bigtriangleup _{I}^{\sigma }f,\bigtriangleup _{J}^{\omega
}g\right\rangle _{\omega },
\end{equation*}%
where $\func{Adj}_{\rho }\left( I\right) $ is defined in (\ref{def Adj}). We
now demonstrate that the norm of $\mathsf{B}_{\func{Adj},\rho }\left(
f,g\right) $ as a bilinear form is comparable to the quadratic Alpert weak
boundedness constant $\mathcal{HWBP}_{T^{\lambda },p}^{\ell ^{2},\rho }$.

\begin{proposition}
\label{diag form}Suppose $1<p<\infty $, $0\leq \rho <\infty $, and $\sigma $
and $\omega $ are positive locally finite Borel measures on $\mathbb{R}^{n}$%
. If $\mathfrak{N}_{L^{p}\left( \sigma \right) \times L^{p^{\prime }}\left(
\omega \right) }$ denotes the smallest constant $C$ in the bilinear
inequality.%
\begin{equation*}
\left\vert \mathsf{B}_{\func{Adj},\rho }\left( f,g\right) \right\vert \leq
C\left\Vert f\right\Vert _{L^{p}\left( \sigma \right) }\left\Vert
g\right\Vert _{L^{p^{\prime }}\left( \omega \right) },
\end{equation*}%
then%
\begin{equation*}
\mathfrak{N}_{L^{p}\left( \sigma \right) \times L^{p^{\prime }}\left( \omega
\right) }\approx \mathcal{HWBP}_{T^{\lambda },p}^{\ell ^{2},\rho }\left(
\sigma ,\omega \right) .
\end{equation*}
\end{proposition}

\begin{proof}
We have%
\begin{eqnarray*}
&&\mathfrak{N}_{L^{p}\left( \sigma \right) \times L^{p^{\prime }}\left(
\omega \right) }=\sup_{\left\Vert f\right\Vert _{L^{p}\left( \sigma \right)
}=\left\Vert g\right\Vert _{L^{p^{\prime }}\left( \omega \right)
}=1}\left\vert \mathsf{B}_{\func{Adj},\rho }\left( f,g\right) \right\vert \\
&=&\sup_{\left\Vert f\right\Vert _{L^{p}\left( \sigma \right) }=\left\Vert
g\right\Vert _{L^{p^{\prime }}\left( \omega \right) }=1}\left\vert \int_{%
\mathbb{R}^{n}}\left[ \sum_{I\in \mathcal{D}}\sum_{J\in \func{Adj},\rho
\left( I\right) }\bigtriangleup _{J}^{\omega }T_{\sigma }^{\lambda
}\bigtriangleup _{I}^{\sigma }f\left( x\right) \right] \ g\left( x\right) \
d\omega \left( x\right) \right\vert \\
&=&\sup_{\left\Vert f\right\Vert _{L^{p}\left( \sigma \right) }=1}\left(
\int_{\mathbb{R}^{n}}\left\vert \sum_{I\in \mathcal{D}}\sum_{J\in \func{Adj}%
,\rho \left( I\right) }\bigtriangleup _{J}^{\omega }T_{\sigma }^{\lambda
}\bigtriangleup _{I}^{\sigma }f\left( x\right) \right\vert ^{p}d\omega
\left( x\right) \right) ^{\frac{1}{p}}.
\end{eqnarray*}%
Now we use the fact that Haar multipliers are bounded on $L^{p}\left( \sigma
\right) $ to obtain%
\begin{equation*}
\left\Vert \sum_{I\in \mathcal{D}}\pm \bigtriangleup _{I}^{\sigma
}f\right\Vert _{L^{p}\left( \sigma \right) }\approx \left\Vert \sum_{I\in 
\mathcal{D}}\bigtriangleup _{I}^{\sigma }f\right\Vert _{L^{p}\left( \sigma
\right) }=\left\Vert f\right\Vert _{L^{p}\left( \sigma \right) }\ .
\end{equation*}%
Hence by the equivalence $\left\Vert \mathcal{S}_{\func{Haar}}f\right\Vert
_{L^{p}\left( \sigma \right) }\approx \left\Vert f\right\Vert _{L^{p}\left(
\sigma \right) }$, we have%
\begin{eqnarray*}
&&\mathfrak{N}_{L^{p}\left( \sigma \right) \times L^{p^{\prime }}\left(
\omega \right) }\approx \sup_{\left\Vert f\right\Vert _{L^{p}\left( \sigma
\right) }=1}\mathbb{E}_{\pm }\left( \int_{\mathbb{R}^{n}}\left\vert
\sum_{I\in \mathcal{D}}\sum_{J\in \func{Adj},\rho \left( I\right)
}\bigtriangleup _{J}^{\omega }T_{\sigma }^{\lambda }\left( \pm
\bigtriangleup _{I}^{\sigma }f\right) \left( x\right) \right\vert
^{p}d\omega \left( x\right) \right) ^{\frac{1}{p}} \\
&\approx &\sup_{\left\Vert f\right\Vert _{L^{p}\left( \sigma \right)
}=1}\left( \int_{\mathbb{R}^{n}}\left( \sum_{I\in \mathcal{D}}\sum_{J\in 
\func{Adj},\rho \left( I\right) }\left\vert \bigtriangleup _{J}^{\omega
}T_{\sigma }^{\lambda }\left( \bigtriangleup _{I}^{\sigma }f\right) \left(
x\right) \right\vert ^{2}\right) ^{\frac{p}{2}}d\omega \left( x\right)
\right) ^{\frac{1}{p}}=\mathcal{HWBP}_{T^{\lambda },p}^{\ell ^{2},\rho
}\left( \sigma ,\omega \right) .
\end{eqnarray*}
\end{proof}

Note that when $\rho =0$, we have $\func{Adj}_{\rho }\left( I\right) =\func{%
Adj}_{0}\left( I\right) =\func{Adj}\left( I\right) $ as defined in the
introduction.

\begin{lemma}
\label{stronger}Suppose that $\sigma $ and $\omega $ are doubling measures, $%
1<p<\infty $, and that $T^{\lambda }$ is a smooth $\lambda $-fractional
Calderon-Zygmund operator. Then for $0<\varepsilon <1$,%
\begin{equation*}
\mathcal{HWBP}_{T^{\lambda },p}^{\ell ^{2},\rho }\left( \sigma ,\omega
\right) \leq C_{\varepsilon }\left[ \mathcal{WBP}_{T^{\lambda },p}^{\ell
^{2}}\left( \sigma ,\omega \right) +A_{p}^{\lambda ,\ell ^{2},\limfunc{offset%
}}\left( \sigma ,\omega \right) \right] +\varepsilon \mathfrak{N}%
_{T^{\lambda },p}\left( \sigma ,\omega \right) .
\end{equation*}
\end{lemma}

\begin{proof}
Fix a dyadic cube $I$. We can write 
\begin{eqnarray*}
\bigtriangleup _{I}^{\sigma }f &=&\sum_{I^{\prime }\in \mathfrak{C}_{%
\mathcal{D}}\left( I\right) }\sum_{I^{\prime \prime }\in \mathfrak{C}_{%
\mathcal{D}}^{\left( m\right) }\left( I^{\prime }\right) }a_{I^{\prime
\prime }}\mathbf{1}_{I^{\prime \prime }}\ , \\
\bigtriangleup _{J}^{\omega }g &=&\sum_{J^{\prime }\in \mathfrak{C}_{%
\mathcal{D}}\left( J\right) }\sum_{J^{\prime \prime }\in \mathfrak{C}_{%
\mathcal{D}}^{\left( m\right) }\left( J^{\prime }\right) }b_{J^{\prime
\prime }}\mathbf{1}_{J^{\prime \prime }}\ ,
\end{eqnarray*}%
where the constants $a_{I^{\prime \prime }}$ and $b_{J^{\prime \prime }}$
are controlled by $\left\Vert \mathbb{E}_{I^{\prime }}\left( \bigtriangleup
_{I;\kappa }^{\sigma }f\right) \right\Vert _{\infty }$ and $\left\Vert 
\mathbb{E}_{J^{\prime }}\left( \bigtriangleup _{J;\kappa }^{\omega }g\right)
\right\Vert _{\infty }$ respectively. Thus we have%
\begin{eqnarray*}
&&\mathsf{B}_{\func{Adj},\rho }\left( f,g\right) \equiv \sum_{I\in \mathcal{D%
}}\sum_{J\in \func{Adj}_{\rho }\left( I\right) }\left\langle T_{\sigma
}^{\lambda }\bigtriangleup _{I;\kappa }^{\sigma }f,\bigtriangleup _{J;\kappa
}^{\omega }g\right\rangle _{\omega } \\
&=&\left\{ \sum_{I\in \mathcal{D}}\sum_{I^{\prime }\in \mathfrak{C}_{%
\mathcal{D}}\left( I\right) }\sum_{I^{\prime \prime }\in \mathfrak{C}_{%
\mathcal{D}}^{\left( m\right) }\left( I^{\prime }\right) }\right\} \left\{
\sum_{J\in \func{Adj}_{\rho }\left( I\right) }\sum_{J^{\prime }\in \mathfrak{%
C}_{\mathcal{D}}\left( I\right) }\sum_{J^{\prime \prime }\in \mathfrak{C}_{%
\mathcal{D}}^{\left( m\right) }\left( J^{\prime }\right) }\right\}
a_{I^{\prime \prime }}b_{J^{\prime \prime }}\left\langle T_{\sigma
}^{\lambda }\mathbf{1}_{I^{\prime \prime }},\mathbf{1}_{J^{\prime \prime
}}\right\rangle _{\omega } \\
&=&\left\{ \sum_{\overline{J^{\prime \prime }}\cap \overline{I^{\prime
\prime }}=\emptyset \text{ and }J\in \func{Adj}_{\rho }\left( I\right)
}+\sum_{\overline{J^{\prime \prime }}\cap \overline{I^{\prime \prime }}%
\not=\emptyset \text{ and }J\in \func{Adj}_{\rho }\left( I\right) }\right\}
a_{I^{\prime \prime }}b_{J^{\prime \prime }}\left\langle T_{\sigma
}^{\lambda }\mathbf{1}_{I^{\prime \prime }},\mathbf{1}_{J^{\prime \prime
}}\right\rangle _{\omega }\equiv T_{\func{sep}}+T_{\func{touch}},
\end{eqnarray*}%
where we have suppressed many of the conditions governing the dyadic cubes $%
I^{\prime \prime }$ and $J^{\prime \prime }$, including the fact that $\ell
\left( J^{\prime \prime }\right) =2^{-m-1}\ell \left( J\right) =2^{-m-1}\ell
\left( I\right) =\ell \left( I^{\prime \prime }\right) $. Thus the cubes $%
J^{\prime \prime }$ and $I^{\prime \prime }$ arising in term $T_{\func{sep}}$
are separated and it is then an easy matter to see that%
\begin{equation*}
\left\vert T_{\func{sep}}\right\vert \leq C_{m}A_{p}^{\lambda ,\ell ^{2},%
\limfunc{offset}}\left( \sigma ,\omega \right) \left\Vert f\right\Vert
_{L^{p}\left( \sigma \right) }\left\Vert g\right\Vert _{L^{p^{\prime
}}\left( \omega \right) }.
\end{equation*}%
As for the term $T_{\func{touch}}$, it is controlled by the weak boundedness
constant,

\begin{equation*}
\left\vert T_{\func{touch}}\right\vert \leq C_{m,\rho }\mathcal{WBP}%
_{T^{\lambda },p}^{\ell ^{2}}\left( \sigma ,\omega \right) \left\Vert
f\right\Vert _{L^{p}\left( \sigma \right) }\left\Vert g\right\Vert
_{L^{p^{\prime }}\left( \omega \right) },
\end{equation*}%
since since the cubes $J^{\prime \prime }$ and $I^{\prime \prime }$ are
adjacent in this sum.
\end{proof}

\section{Forms requiring testing conditions}

The three forms requiring conditions other than those of Muckenhoupt type,
are the Haar adjacent diagonal form, which uses only the quadratic Haar weak
boundedness constant $\mathcal{HWBP}_{T^{\lambda },p}^{\ell ^{2},\rho }$,
and the two dual paraproduct forms, which each use only the appropriate
scalar testing condition $\mathfrak{T}_{T^{\lambda },p}$ or $\mathfrak{T}%
_{T^{\lambda ,\ast },p^{\prime }}$.

\subsection{Adjacent diagonal form}

Here we control the quadratic adjacent form by%
\begin{eqnarray*}
\left\vert \mathsf{B}_{\func{Adj},\rho }\left( f,g\right) \right\vert
&=&\left\vert \dsum\limits_{I\in \mathcal{D},\ J\in \func{Adj}_{\rho }\left(
I\right) }\left\langle T_{\sigma }^{\lambda }\left( \bigtriangleup
_{I}^{\sigma }f\right) ,\left( \bigtriangleup _{J}^{\omega }g\right)
\right\rangle _{\omega }\right\vert \\
&=&\left\vert \int_{\mathbb{R}^{n}}\dsum\limits_{I\in \mathcal{D},\ J\in 
\func{Adj}_{\rho }\left( I\right) }\bigtriangleup _{J}^{\omega }T_{\sigma
}^{\lambda }\left( \bigtriangleup _{I}^{\sigma }f\right) \left( x\right) \
\bigtriangleup _{J}^{\omega }g\left( x\right) \ d\omega \left( x\right)
\right\vert \\
&\leq &\int_{\mathbb{R}^{n}}\left( \dsum\limits_{I\in \mathcal{D},\ J\in 
\func{Adj}_{\rho }\left( I\right) }\left\vert \bigtriangleup _{J}^{\omega
}T_{\sigma }^{\lambda }\left( \bigtriangleup _{I}^{\sigma }f\right) \left(
x\right) \right\vert ^{2}\right) ^{\frac{1}{2}}\ \left( \dsum\limits_{I\in 
\mathcal{D},\ J\in \func{Adj}_{\rho }\left( I\right) }\left\vert
\bigtriangleup _{J}^{\omega }g\left( x\right) \right\vert ^{2}\right) ^{%
\frac{1}{2}}\ d\omega \left( x\right) \\
&\lesssim &\left\Vert \left( \dsum\limits_{I\in \mathcal{D},\ J\in \func{Adj}%
_{\rho }\left( I\right) }\left\vert \bigtriangleup _{J}^{\omega }T_{\sigma
}^{\lambda }\left( \bigtriangleup _{I}^{\sigma }f\right) \left( x\right)
\right\vert ^{2}\right) ^{\frac{1}{2}}\right\Vert _{L^{p}\left( \omega
\right) }\ \left\Vert \mathcal{S}_{\func{Haar}}g\right\Vert _{L^{p^{\prime
}}\left( \omega \right) }.
\end{eqnarray*}%
We have $\left\Vert \mathcal{S}_{\func{Haar}}g\right\Vert _{L^{p^{\prime
}}\left( \omega \right) }\approx \left\Vert g\right\Vert _{L^{p^{\prime
}}\left( \omega \right) }$ by a square function estimate, and using the
quadratic Haar weak boundedness property, we obtain%
\begin{equation*}
\left( \int_{\mathbb{R}^{n}}\left( \dsum\limits_{I\in \mathcal{D},\ J\in 
\func{Adj}_{\rho }\left( I\right) }\left\vert \bigtriangleup _{J}^{\omega
}T_{\sigma }^{\lambda }\left( \bigtriangleup _{I}^{\sigma }f\right) \left(
x\right) \right\vert ^{2}\right) ^{\frac{p}{2}}d\omega \left( x\right)
\right) ^{\frac{1}{p}}\lesssim \mathcal{HWBP}_{T^{\lambda },p}^{\ell
^{2},\rho }\left( \sigma ,\omega \right) \left\Vert f\right\Vert
_{L^{p}\left( \sigma \right) }\ ,
\end{equation*}%
and so altogether that%
\begin{equation*}
\left\vert \mathsf{B}_{\func{Adj},\rho }\left( f,g\right) \right\vert
\lesssim \mathcal{HWBP}_{T^{\lambda },p}^{\ell ^{2},\rho }\left( \sigma
,\omega \right) \ \left\Vert f\right\Vert _{L^{p}\left( \sigma \right)
}\left\Vert g\right\Vert _{L^{p^{\prime }}\left( \omega \right) }.
\end{equation*}

Recall that in Lemma \ref{stronger}, we have controlled the Haar weak
boundedness property constant $\mathcal{HWBP}_{T^{\lambda },p}^{\ell
^{2},\rho }$ by the adjacent weak boundedness property constant $\mathcal{WBP%
}_{T^{\lambda },p}^{\ell ^{2}}$ and the offset Muckenhoupt constant $%
A_{p}^{\lambda ,\ell ^{2},\limfunc{offset}}$, plus a small multiple of the
operator norm. This will be used at the end of the proof to eliminate the
use of $\mathcal{HWBP}_{T^{\lambda },p}^{\ell ^{2},\rho }$.

\subsection{Paraproduct form}

Here we must bound the paraproduct form,%
\begin{eqnarray*}
\mathsf{B}_{\limfunc{paraproduct}}\left( f,g\right) &=&\sum_{F\in \mathcal{F}%
}\mathsf{B}_{\limfunc{paraproduct}}^{F}\left( f,g\right) =\sum_{F\in 
\mathcal{F}}\sum_{\substack{ I\in \mathcal{C}_{F}\text{ and }J\in \mathcal{C}%
_{F}^{\tau -\limfunc{shift}}  \\ J\Subset _{\rho ,\varepsilon }I}}%
\left\langle M_{I_{J}}T_{\sigma }^{\lambda }\mathbf{1}_{F},\bigtriangleup
_{J}^{\omega }g\right\rangle _{\omega } \\
&=&\sum_{F\in \mathcal{F}}\sum_{J\in \mathcal{C}_{F}^{\tau -\limfunc{shift}%
}}\sum_{I\in \mathcal{C}_{F}\text{ }J\Subset _{\rho ,\varepsilon
}I}\left\langle M_{I_{J}}T_{\sigma }^{\lambda }\mathbf{1}_{F},\bigtriangleup
_{J}^{\omega }g\right\rangle _{\omega } \\
&=&\sum_{F\in \mathcal{F}}\sum_{J\in \mathcal{C}_{F}^{\tau -\limfunc{shift}%
}}\left\langle \left[ \mathbf{1}_{\widehat{J}_{J}}\left[ \left( \mathbb{E}_{%
\widehat{J}}^{\sigma }f\right) -\left( \mathbb{E}_{F}^{\sigma }f\right) %
\right] \right] T_{\sigma }^{\lambda }\mathbf{1}_{F},\bigtriangleup
_{J}^{\omega }g\right\rangle _{\omega },
\end{eqnarray*}%
where $\widehat{J}$ is the smallest $I\in \mathcal{C}_{F}$ for which $%
J\Subset _{\rho ,\varepsilon }I$, in Theorem \ref{main}. Note that because
of the projection $\bigtriangleup _{J}^{\omega }g$ the telescoping sum in
the second line above is restricted to $J$. Define $\widetilde{g}=\sum_{J\in 
\mathcal{D}}\frac{\mathbf{1}_{\widehat{J}_{J}}\left[ \left( \mathbb{E}_{%
\widehat{J}}^{\sigma }f\right) -\left( \mathbb{E}_{F}^{\sigma }f\right) %
\right] }{E_{F}^{\sigma }\left\vert f\right\vert }\bigtriangleup
_{J}^{\omega }g$, and noting that $\left\vert \mathbb{E}_{\widehat{J}%
}^{\sigma }f\right\vert +\left\vert \mathbb{E}_{F}^{\sigma }f\right\vert
\lesssim E_{F}^{\sigma }\left\vert f\right\vert $ by (\ref{analogue}), we
obtain%
\begin{eqnarray*}
&&\left\vert \mathsf{B}_{\limfunc{paraproduct}}\left( f,g\right) \right\vert
=\left\vert \sum_{F\in \mathcal{F}}\mathsf{B}_{\limfunc{paraproduct}%
}^{F}\left( f,g\right) \right\vert =\left\vert \sum_{F\in \mathcal{F}%
}\sum_{J\in \mathcal{C}_{F}^{\tau -\limfunc{shift}}}\left\langle \left( 
\mathbf{1}_{\widehat{J}_{J}}\left[ \left( \mathbb{E}_{\widehat{J}}^{\sigma
}f\right) -\left( \mathbb{E}_{F}^{\sigma }f\right) \right] \right) T_{\sigma
}^{\lambda }\mathbf{1}_{F},\bigtriangleup _{J}^{\omega }g\right\rangle
_{\omega }\right\vert \\
&=&\left\vert \sum_{F\in \mathcal{F}}\sum_{J\in \mathcal{C}_{F}^{\tau -%
\limfunc{shift}}}\left\langle T_{\sigma }^{\lambda }\mathbf{1}_{F},\left( 
\mathbf{1}_{\widehat{J}_{J}}\left[ \left( \mathbb{E}_{\widehat{J}}^{\sigma
}f\right) -\left( \mathbb{E}_{F}^{\sigma }f\right) \right] \right)
\bigtriangleup _{J}^{\omega }g\right\rangle _{\omega }\right\vert \\
&=&\left\vert \sum_{F\in \mathcal{F}}\left\langle T_{\sigma }^{\lambda }%
\mathbf{1}_{F},\sum_{J\in \mathcal{C}_{F}^{\tau -\limfunc{shift}}}\left( 
\mathbf{1}_{\widehat{J}_{J}}\left[ \left( \mathbb{E}_{\widehat{J}}^{\sigma
}f\right) -\left( \mathbb{E}_{F}^{\sigma }f\right) \right] \right)
\bigtriangleup _{J}^{\omega }g\right\rangle _{\omega }\right\vert \\
&=&\left\vert \sum_{F\in \mathcal{F}}\left\langle \mathbf{1}_{F}T_{\sigma
}^{\lambda }\mathbf{1}_{F},\sum_{J\in \mathcal{C}_{F}^{\tau -\limfunc{shift}%
}}\left( \mathbf{1}_{\widehat{J}_{J}}\left[ \left( \mathbb{E}_{\widehat{J}%
}^{\sigma }f\right) -\left( \mathbb{E}_{F}^{\sigma }f\right) \right] \right)
\bigtriangleup _{J}^{\omega }g\right\rangle _{\omega }\right\vert
=\left\vert \int_{\mathbb{R}^{n}}\sum_{F\in \mathcal{F}}\mathbf{1}%
_{F}T_{\sigma }^{\lambda }\mathbf{1}_{F}\left( x\right) \sum_{J\in \mathcal{C%
}_{F}^{\tau -\limfunc{shift}}}P_{J}\bigtriangleup _{J}^{\omega }g\left(
x\right) d\omega \left( x\right) \right\vert ,
\end{eqnarray*}%
where $P_{J}$ is the constant $\mathbf{1}_{\widehat{J}_{J}}\left[ \left( 
\mathbb{E}_{\widehat{J}}^{\sigma }f\right) -\left( \mathbb{E}_{F}^{\sigma
}f\right) \right] $. The final term above is dominated by%
\begin{eqnarray*}
&&\int_{\mathbb{R}^{n}}\left( \sum_{F\in \mathcal{F}}E_{F}^{\sigma
}\left\vert f\right\vert ^{2}\left\vert \mathbf{1}_{F}T_{\sigma }^{\lambda }%
\mathbf{1}_{F}\left( x\right) \right\vert ^{2}\right) ^{\frac{1}{2}}\left(
\sum_{F\in \mathcal{F}}\left\vert \sum_{J\in \mathcal{C}_{F}^{\tau -\limfunc{%
shift}}}\frac{P_{J}}{E_{F}^{\sigma }\left\vert f\right\vert }\bigtriangleup
_{J}^{\omega }g\left( x\right) \right\vert ^{2}\right) ^{\frac{1}{2}}d\omega
\left( x\right) \\
&\leq &\left( \int_{\mathbb{R}^{n}}\left( \sum_{F\in \mathcal{F}%
}E_{F}^{\sigma }\left\vert f\right\vert ^{2}\left\vert \mathbf{1}%
_{F}T_{\sigma }^{\lambda }\mathbf{1}_{F}\left( x\right) \right\vert
^{2}\right) ^{\frac{p}{2}}d\omega \left( x\right) \right) ^{\frac{1}{p}%
}\left( \int_{\mathbb{R}^{n}}\left( \sum_{F\in \mathcal{F}}\left\vert
\sum_{J\in \mathcal{C}_{F}^{\tau -\limfunc{shift}}}\frac{P_{J}}{%
E_{F}^{\sigma }\left\vert f\right\vert }\bigtriangleup _{J}^{\omega }g\left(
x\right) \right\vert ^{2}\right) ^{\frac{p^{\prime }}{2}}d\omega \left(
x\right) \right) ^{\frac{1}{p^{\prime }}}.
\end{eqnarray*}%
The first factor above is controlled by the local quadratic testing
characteristic,%
\begin{eqnarray*}
\left( \int_{\mathbb{R}^{n}}\left( \sum_{F\in \mathcal{F}}\left(
E_{F}^{\sigma }\left\vert f\right\vert \right) ^{2}\left\vert \mathbf{1}%
_{F}T_{\sigma }^{\lambda }\mathbf{1}_{F}\left( x\right) \right\vert
^{2}\right) ^{\frac{p}{2}}d\omega \left( x\right) \right) ^{\frac{1}{p}}
&\leq &\mathfrak{T}_{T^{\lambda },p}^{\ell ^{2},\limfunc{loc}}\left( \sigma
,\omega \right) \left( \int_{\mathbb{R}^{n}}\left( \sum_{F\in \mathcal{F}%
}\left( E_{F}^{\sigma }\left\vert f\right\vert \right) ^{2}\mathbf{1}%
_{F}\left( x\right) \right) ^{\frac{p}{2}}d\sigma \left( x\right) \right) ^{%
\frac{1}{p}} \\
&\lesssim &\mathfrak{T}_{T^{\lambda },p}^{\ell ^{2},\limfunc{loc}}\left(
\sigma ,\omega \right) \left\Vert f\right\Vert _{L^{p}\left( \sigma \right)
},
\end{eqnarray*}%
and the second factor above is controlled by the square function estimates
and the inequality $\left\vert \frac{P_{J}}{E_{F}^{\sigma }\left\vert
f\right\vert }\right\vert \lesssim 1$. Indeed with $\widetilde{g}\equiv
\sum_{J\in \mathcal{C}_{F}^{\tau -\limfunc{shift}}}\frac{P_{J}}{%
E_{F}^{\sigma }\left\vert f\right\vert }\bigtriangleup _{J}^{\omega }g\left(
x\right) $ we have%
\begin{equation}
\int_{\mathbb{R}^{n}}\left( \sum_{F\in \mathcal{F}}\left\vert \sum_{J\in 
\mathcal{C}_{F}^{\tau -\limfunc{shift}}}\frac{P_{J}}{E_{F}^{\sigma
}\left\vert f\right\vert }\bigtriangleup _{J}^{\omega }\left( g\right)
\left( x\right) \right\vert ^{2}\right) ^{\frac{p^{\prime }}{2}}d\omega
\left( x\right) \lesssim \left\Vert \widetilde{g}\right\Vert _{L^{p^{\prime
}}\left( \omega \right) }^{p^{\prime }}\lesssim \left\Vert g\right\Vert
_{L^{p^{\prime }}\left( \omega \right) }^{p^{\prime }},\ \ \ \ \ 1<p^{\prime
}<\infty .  \label{turn to}
\end{equation}

\section{Forms requiring quadratic offset Muckenhoupt conditions}

To bound the disjoint $\mathsf{B}_{\cap }\left( f,g\right) $, comparable $%
\mathsf{B}_{/}\left( f,g\right) $, stopping $\mathsf{B}_{\limfunc{stop}%
}\left( f,g\right) $, far below $\mathsf{B}_{\func{far}\func{below}}\left(
f,g\right) $, and neighbour $\mathsf{B}_{\limfunc{neighbour}}\left(
f,g\right) $ forms, we will need the quadratic offset Muckenhoupt
conditions, as well as a Pivotal Lemma, which originated in \cite{NTV4}. For 
$0\leq \lambda <n$ and $t\in \mathbb{R}_{+}$, recall the $t^{th}$-order
fractional Poisson integral%
\begin{equation*}
\mathrm{P}_{t}^{\lambda }\left( J,\mu \right) \equiv \int_{\mathbb{R}^{n}}%
\frac{\ell \left( J\right) ^{t}}{\left( \ell \left( J\right) +\left\vert
y-c_{J}\right\vert \right) ^{t+n-\lambda }}d\mu \left( y\right) ,
\end{equation*}%
where $\mathrm{P}_{1}^{\lambda }\left( J,\mu \right) =\mathrm{P}^{\lambda
}\left( J,\mu \right) $ is the standard Poisson integral of order $\lambda $%
. The following Poisson estimate from \cite[Lemma 33]{Saw6} is a
straightforward extension of the case $\kappa =1$ due to Nazarov, Treil and
Volberg in \cite{NTV4}, which provided the vehicle through which geometric
gain was derived from their groundbreaking notion of goodness.

\begin{lemma}
\label{Poisson inequality}Suppose that $J\subset I\subset K$ and that $%
\limfunc{dist}\left( J,\partial I\right) >2\sqrt{n}\ell \left( J\right)
^{\varepsilon }\ell \left( I\right) ^{1-\varepsilon }$. Then%
\begin{equation}
\mathrm{P}^{\lambda }(J,\sigma \mathbf{1}_{K\setminus I})\lesssim \left( 
\frac{\ell \left( J\right) }{\ell \left( I\right) }\right) ^{1-\varepsilon
\left( n+1-\lambda \right) }\mathrm{P}^{\lambda }(I,\sigma \mathbf{1}%
_{K\setminus I}).  \label{e.Jsimeq}
\end{equation}
\end{lemma}

\begin{lemma}[\textbf{Pivotal Lemma}]
\label{ener}Let $J\ $be a cube in $\mathcal{D}$, and let $\Psi _{J}$ be an $%
L^{2}\left( \omega \right) $ function supported in $J$ with vanishing $%
\omega $-mean. Let $\nu $ be a positive measure supported in $\mathbb{R}%
^{n}\setminus \gamma J$ with $\gamma >1$, and let $T^{\lambda }$ be a
standard $\lambda $-fractional singular integral operator with $0\leq
\lambda <n$. Then we have the `pivotal' bound%
\begin{equation}
\left\vert \left\langle T^{\lambda }\left( \varphi \nu \right) ,\Psi
_{J}\right\rangle _{L^{2}\left( \omega \right) }\right\vert \lesssim
C_{\gamma }\mathrm{P}^{\lambda }\left( J,\nu \right) \left\Vert \Psi
_{J}\right\Vert _{L^{1}\left( \omega \right) }\leq C_{\gamma }\mathrm{P}%
^{\lambda }\left( J,\nu \right) \sqrt{\left\vert J\right\vert _{\omega }}%
\left\Vert \Psi _{J}\right\Vert _{L^{2}\left( \omega \right) }\ ,
\label{piv bound}
\end{equation}%
for any function $\varphi $ with $\left\vert \varphi \right\vert \leq 1$.
\end{lemma}

This form of the lemma is proved in many places in the literature, but
usually with only the far right estimate. However, all of the proofs can be
stopped one line short to give the first estimate, see e.g. \cite{NTV4}
where it originates.

\subsection{Disjoint form}

We decompose the disjoint form into two pieces,%
\begin{eqnarray*}
\mathsf{B}_{\cap }\left( f,g\right) &=&\dsum\limits_{I,J\in \mathcal{D}\
:J\cap I=\emptyset \text{ and }\frac{\ell \left( J\right) }{\ell \left(
I\right) }\notin \left[ 2^{-\rho },2^{\rho }\right] }\left\langle T_{\sigma
}^{\lambda }\left( \bigtriangleup _{I}^{\sigma }f\right) ,\left(
\bigtriangleup _{J}^{\omega }g\right) \right\rangle _{\omega } \\
&=&\left\{ \dsum\limits_{\substack{ I,J\in \mathcal{D}\ :J\cap I=\emptyset 
\\ \ell \left( J\right) <2^{-\rho }\ell \left( I\right) }}+\dsum\limits 
_{\substack{ I,J\in \mathcal{D}\ :J\cap I=\emptyset  \\ \ell \left( J\right)
>2^{\rho }\ell \left( I\right) }}\right\} \left\langle T_{\sigma }^{\lambda
}\left( \bigtriangleup _{I}^{\sigma }f\right) ,\left( \bigtriangleup
_{J}^{\omega }g\right) \right\rangle _{\omega } \\
&\equiv &\mathsf{B}_{\cap }^{\limfunc{down}}\left( f,g\right) +\mathsf{B}%
_{\cap }^{\limfunc{up}}\left( f,g\right) .
\end{eqnarray*}%
Since the up form is dual to the down form, we consider only $\mathsf{B}%
_{\cap }^{\limfunc{down}}\left( f,g\right) $, and we will prove the
following estimate: 
\begin{equation}
\left\vert \mathsf{B}_{\cap }^{\limfunc{down}}\left( f,g\right) \right\vert
\lesssim A_{p}^{\lambda ,\ell ^{2},\limfunc{offset}}\left( \sigma ,\omega
\right) \left\Vert f\right\Vert _{L^{p}\left( \sigma \right) }\left\Vert
g\right\Vert _{L^{p^{\prime }}\left( \omega \right) }.  \label{routine'}
\end{equation}

\begin{description}
\item[Porism] It is important to note that from the proof given, we may
replace the sum $\dsum\limits_{\substack{ I,J\in \mathcal{D}\ :J\cap
I=\emptyset  \\ \ell \left( J\right) <2^{-\rho }\ell \left( I\right) }}$ in
the left hand side of (\ref{routine'}) with a sum over any \emph{subset} of
the pairs $I,J\,$arising in $\mathsf{B}_{\cap }^{\limfunc{down}}\left(
f,g\right) $. A similar remark of course applies to $\mathsf{B}_{\cap }^{%
\limfunc{up}}\left( f,g\right) $.
\end{description}

\begin{proof}[Proof of (\protect\ref{routine'})]
Denote by $\limfunc{dist}$ the $\ell ^{\infty }$ distance in $\mathbb{R}^{n}$%
: $\limfunc{dist}\left( x,y\right) =\max_{1\leq j\leq n}\left\vert
x_{j}-y_{j}\right\vert $. We now estimate separately the long-range and
mid-range cases in $\mathsf{B}_{\cap }^{\limfunc{down}}\left( f,g\right) $
where $\limfunc{dist}\left( J,I\right) \geq \ell \left( I\right) $ holds or
not, and we decompose $\mathsf{B}_{\cap }^{\limfunc{down}}\left( f,g\right) $
accordingly:%
\begin{equation*}
\mathsf{B}_{\cap }^{\limfunc{down}}\left( f,g\right) =\mathcal{A}^{\limfunc{%
long}}\left( f,g\right) +\mathcal{A}^{\limfunc{mid}}\left( f,g\right) .
\end{equation*}

\bigskip

\textbf{The long-range case}: We begin with the case where $\limfunc{dist}%
\left( J,I\right) $ is at least $\ell \left( I\right) $, i.e. $J\cap
3I=\emptyset $. With $A\left( f,g\right) =\mathcal{A}^{\limfunc{long}}\left(
f,g\right) $ we have%
\begin{equation*}
A\left( f,g\right) =\sum_{\substack{ I,J\in \mathcal{D}:\ \limfunc{dist}%
\left( J,I\right) \geq \ell \left( I\right)  \\ \ell \left( J\right) \leq
2^{-\rho }\ell \left( I\right) }}\left\langle T_{\sigma }^{\lambda }\left(
\bigtriangleup _{I}^{\sigma }f\right) ,\bigtriangleup _{J}^{\omega
}g\right\rangle _{\omega }=\sum_{s=\rho }^{\infty }\sum_{m=1}^{\infty
}A_{s,m}\left( f,g\right) ,
\end{equation*}%
where%
\begin{eqnarray*}
A_{s,m}\left( f,g\right) &=&\sum_{\substack{ I,J\in \mathcal{D}:\ \limfunc{%
dist}\left( J,I\right) \approx \ell \left( I\right) ^{m}  \\ \ell \left(
J\right) =2^{-s}\ell \left( I\right) }}\left\langle T_{\sigma }^{\lambda
}\left( \bigtriangleup _{I}^{\sigma }f\right) ,\bigtriangleup _{J}^{\omega
}g\right\rangle _{\omega } \\
&=&\sum_{J\in \mathcal{D}}\sum_{I\in \mathcal{F}_{s,m}\left( J\right)
}\left\langle T_{\sigma }^{\lambda }\left( \bigtriangleup _{I}^{\sigma
}f\right) ,\bigtriangleup _{J}^{\omega }g\right\rangle _{\omega }=\sum_{J\in 
\mathcal{D}}\left\langle T_{\sigma }^{\lambda }\left( \mathsf{Q}%
_{J,s,m}^{\sigma }f\right) ,\bigtriangleup _{J}^{\omega }g\right\rangle
_{\omega },
\end{eqnarray*}%
with 
\begin{equation*}
\mathcal{F}_{s,m}\left( J\right) \equiv \left\{ I\in \mathcal{D}:\ \limfunc{%
dist}\left( J,I\right) \approx 2^{m}\ell \left( I\right) \text{, }\ell
\left( I\right) =2^{s}\ell \left( J\right) \right\} \text{ and }\mathsf{Q}%
_{J,s,m}^{\sigma }\equiv \sum_{I\in \mathcal{F}_{s,m}\left( J\right)
}\bigtriangleup _{I}^{\sigma }.
\end{equation*}%
Then from the Pivotal Lemma \ref{ener} we have%
\begin{equation*}
\left\vert \left\langle T_{\sigma }^{\lambda }\left( \mathsf{Q}%
_{J,s,m}^{\sigma }f\right) ,\bigtriangleup _{J}^{\omega }g\right\rangle
_{\omega }\right\vert \lesssim \mathrm{P}^{\lambda }\left( J,\left\vert 
\mathsf{Q}_{J,s,m}^{\sigma }f\right\vert \sigma \right) \int_{J}\left\vert
\bigtriangleup _{J}^{\omega }g\right\vert d\omega ,
\end{equation*}%
where%
\begin{eqnarray*}
\mathrm{P}^{\lambda }\left( J,\left\vert \mathsf{Q}_{J,s,m}^{\sigma
}f\right\vert \sigma \right) &\mathbf{=}&\int_{\mathbb{R}^{n}}\frac{\ell
\left( J\right) }{\left\vert \ell \left( J\right) +\limfunc{dist}\left(
y,J\right) \right\vert ^{n+1-\lambda }}\left\vert \mathsf{Q}_{J,s,m}^{\sigma
}f\left( y\right) \right\vert d\sigma \left( y\right) \\
&\lesssim &2^{-\left( s+m\right) }\int_{\mathbb{R}^{n}\setminus 3J}\frac{1}{%
\left\vert \ell \left( J\right) +\limfunc{dist}\left( y,J\right) \right\vert
^{n-\lambda }}\left\vert \mathsf{Q}_{J,s,m}^{\sigma }f\left( y\right)
\right\vert d\sigma \left( y\right) ,
\end{eqnarray*}%
by the definition of $\mathsf{Q}_{J,s,m}^{\sigma }$ since 
\begin{equation}
\ell \left( J\right) =2^{-s}\ell \left( I\right) \approx 2^{-s-m}\limfunc{%
dist}\left( y,J\right) .  \label{pigeon s}
\end{equation}%
Thus we have%
\begin{eqnarray*}
&&\left\vert A_{s,m}\left( f,g\right) \right\vert \lesssim 2^{-\left(
s+m\right) }\int_{\mathbb{R}^{n}}\sum_{J\in \mathcal{D}}\left( \int_{\mathbb{%
R}^{n}\setminus 3J}\frac{1}{\left\vert c_{J}-y\right\vert ^{n-\lambda }}%
\left\vert \mathsf{Q}_{J,s,m}^{\sigma }f\left( y\right) \right\vert d\sigma
\left( y\right) \right) \mathbf{1}_{J}\left( x\right) \left\vert
\bigtriangleup _{J}^{\omega }g\left( x\right) \right\vert d\omega \left(
x\right) \\
&\lesssim &2^{-\left( s+m\right) }\int_{\mathbb{R}^{n}}\left( \sum_{J\in 
\mathcal{D}}\left( \int_{\mathbb{R}^{n}\setminus 3J}\frac{1}{\left\vert
c_{J}-y\right\vert ^{n-\lambda }}\left\vert \mathsf{Q}_{J,s,m}^{\sigma
}f\left( y\right) \right\vert d\sigma \left( y\right) \mathbf{1}_{J}\left(
x\right) \right) ^{2}\right) ^{\frac{1}{2}}\left( \sum_{J\in \mathcal{D}%
}\left\vert \bigtriangleup _{J}^{\omega }g\left( x\right) \right\vert
^{2}\right) ^{\frac{1}{2}}d\omega \left( x\right) \\
&\leq &2^{-\left( s+m\right) }\left( \int_{\mathbb{R}^{n}}\left( \sum_{J\in 
\mathcal{D}}\left( \int_{\mathbb{R}^{n}\setminus 3J}\frac{1}{\left\vert
c_{J}-y\right\vert ^{n-\lambda }}\left\vert \mathsf{Q}_{J,s,m}^{\sigma
}f\left( y\right) \right\vert d\sigma \left( y\right) \mathbf{1}_{J}\left(
x\right) \right) ^{2}\right) ^{\frac{p}{2}}d\omega \left( x\right) \right) ^{%
\frac{1}{p}}\left\Vert \mathcal{S}_{\func{Haar}}^{\omega }g\right\Vert
_{L^{p^{\prime }}\left( \omega \right) }.
\end{eqnarray*}

Now $\mathcal{S}_{\func{Haar}}^{\omega }$ is bounded on $L^{p^{\prime
}}\left( \omega \right) $, and so by the geometric decay in $s$ and $m$, it
remains to show that for each $s,m\in \mathbb{N}$, 
\begin{equation}
\left( \int_{\mathbb{R}^{n}}\left( \sum_{J\in \mathcal{D}}\left( \int_{%
\mathbb{R}^{n}\setminus 3J}\frac{1}{\left\vert c_{J}-y\right\vert
^{n-\lambda }}\left\vert \mathsf{Q}_{J,s,m}^{\sigma }f\left( y\right)
\right\vert d\sigma \left( y\right) \right) ^{2}\mathbf{1}_{J}\left(
x\right) \right) ^{\frac{p}{2}}d\omega \left( x\right) \right) ^{\frac{1}{p}%
}\lesssim A_{p}^{\lambda ,\ell ^{2},\limfunc{offset}}\left\Vert f\right\Vert
_{L^{p}\left( \sigma \right) }.  \label{RTS}
\end{equation}%
For this we use (\ref{pigeon s}) to write%
\begin{equation*}
\int_{\mathbb{R}^{n}\setminus 3J}\frac{1}{\left\vert c_{J}-y\right\vert
^{n-\lambda }}\left\vert \mathsf{Q}_{J,s,m}^{\sigma }f\left( y\right)
\right\vert d\sigma \left( y\right) \approx \frac{1}{\left( 2^{\left(
s+m\right) }\ell \left( J\right) \right) ^{n-\lambda }}\int_{\mathbb{R}%
^{n}\setminus 3J}\left\vert \mathsf{Q}_{J,s,m}^{\sigma }f\left( y\right)
\right\vert d\sigma \left( y\right) ,
\end{equation*}%
and then obtain with $K_{s,m}\left( J\right) \approx
\bigcup_{j=1}^{c2^{n}}K_{s,m}^{j}\left( J\right) $ roughly equal to the
support of $\mathsf{Q}_{J,s,m}^{\sigma }$, that 
\begin{eqnarray*}
&&\int_{\mathbb{R}^{n}}\left( \sum_{J\in \mathcal{D}}\left( \int_{\mathbb{R}%
^{n}\setminus 3J}\frac{1}{\left\vert c_{J}-y\right\vert ^{n-\lambda }}%
\left\vert \mathsf{Q}_{J,s,m}^{\sigma }f\left( y\right) \right\vert d\sigma
\left( y\right) \right) ^{2}\mathbf{1}_{J}\left( x\right) \right) ^{\frac{p}{%
2}}d\omega \left( x\right) \\
&\approx &\int_{\mathbb{R}^{n}}\left( \sum_{J\in \mathcal{D}}\left( \frac{1}{%
\left( 2^{\left( s+m\right) }\ell \left( J\right) \right) ^{n-\lambda }}%
\int_{K_{s,m}\left( J\right) }\left\vert \mathsf{Q}_{J,s,m}^{\sigma }f\left(
y\right) \right\vert d\sigma \left( y\right) \right) ^{2}\mathbf{1}%
_{J}\left( x\right) \right) ^{\frac{p}{2}}d\omega \left( x\right) \\
&\approx &\int_{\mathbb{R}^{n}}\left( \sum_{K\in \mathcal{D}}\sum_{J\in 
\mathcal{D}:\ J\subset K\text{ and }K_{s,m}\left( J\right) \approx
K}\sum_{j=1}^{c2^{n}}\mathbf{1}_{J}\left( x\right) \left( \frac{1}{\left(
2^{\left( s+m\right) }\ell \left( J\right) \right) ^{n-\lambda }}%
\int_{K_{s,m}^{j}\left( J\right) }\left\vert \mathsf{Q}_{J,s,m}^{\sigma
}f\left( y\right) \right\vert d\sigma \left( y\right) \right) ^{2}\right) ^{%
\frac{p}{2}}d\omega \left( x\right) \\
&\lesssim &\int_{\mathbb{R}^{n}}\left( \sum_{K\in \mathcal{D}}\mathbf{1}_{%
\widetilde{K}}\left( x\right) \left( \frac{\left\vert K\right\vert _{\sigma }%
}{\ell \left( K\right) ^{n-\lambda }}\frac{1}{\left\vert K\right\vert
_{\sigma }}\int_{K}\left\vert \mathsf{Q}_{K}^{\sigma }f\left( y\right)
\right\vert d\sigma \left( y\right) \right) ^{2}\right) ^{\frac{p}{2}%
}d\omega \left( x\right) ,
\end{eqnarray*}%
where $\mathsf{Q}_{K}^{\sigma }\equiv \sum_{K_{s,m}\left( J\right) \approx K}%
\mathsf{Q}_{J,s,m}^{\sigma }$ and $\widetilde{K}=\bigcup_{j=1}^{c2^{n}}%
\widetilde{K}^{j}\left( J\right) $ is a union of dyadic cubes $\widetilde{K}%
^{j}$ surrounding $K$ with $0<\limfunc{dist}\left( \widetilde{K}%
^{j},K\right) \lesssim \ell \left( K\right) \approx \ell \left( \widetilde{K}%
^{j}\right) $. Now we use first the quadratic offset condition $%
A_{p}^{\lambda ,\ell ^{2},\limfunc{offset}}\left( \sigma ,\omega \right) $,
and then the Fefferman-Stein vector-valued inequality for the maximal
function, to obtain the following vector-valued inequality for each fixed $%
s,m\in \mathbb{N}$,%
\begin{eqnarray*}
&&\int_{\mathbb{R}^{n}}\left( \sum_{K\in \mathcal{D}}\mathbf{1}_{\widetilde{K%
}}\left( x\right) \left( \frac{\left\vert K\right\vert _{\sigma }}{\ell
\left( K\right) ^{n-\lambda }}\frac{1}{\left\vert K\right\vert _{\sigma }}%
\int_{K}\left\vert \mathsf{Q}_{K}^{\sigma }f\left( y\right) \right\vert
d\sigma \left( y\right) \right) ^{2}\right) ^{\frac{p}{2}}d\omega \left(
x\right) \\
&\lesssim &A_{p}^{\lambda ,\ell ^{2},\limfunc{offset}}\left( \sigma ,\omega
\right) ^{p}\int_{\mathbb{R}^{n}}\left( \sum_{K\in \mathcal{D}}\mathbf{1}%
_{K}\left( x\right) \left( \frac{1}{\left\vert K\right\vert _{\sigma }}%
\int_{K}\left\vert \mathsf{Q}_{K}^{\sigma }f\left( y\right) \right\vert
d\sigma \left( y\right) \right) ^{2}\right) ^{\frac{p}{2}}d\sigma \left(
x\right) \\
&\lesssim &A_{p}^{\lambda ,\ell ^{2},\limfunc{offset}}\left( \sigma ,\omega
\right) ^{p}\int_{\mathbb{R}^{n}}\left( \sum_{K\in \mathcal{D}}\left\vert 
\mathsf{Q}_{K}^{\sigma }f\left( x\right) \right\vert ^{2}\right) ^{\frac{p}{2%
}}d\sigma \left( x\right) \lesssim A_{p}^{\lambda ,\ell ^{2},\limfunc{offset}%
}\left( \sigma ,\omega \right) ^{p}\left\Vert f\right\Vert _{L^{p}\left(
\sigma \right) }^{p}.
\end{eqnarray*}%
As mentioned above, this completes the proof of the long range case by the
geometric decay in $s$ and $m$.

\bigskip

\textbf{The mid range case}: Let%
\begin{equation*}
\mathcal{P}\equiv \left\{ \left( I,J\right) \in \mathcal{D}\times \mathcal{D}%
:J\text{ is good},\ \ell \left( J\right) \leq 2^{-\rho }\ell \left( I\right)
,\text{ }J\subset 3I\setminus I\right\} .
\end{equation*}%
Now we pigeonhole the lengths of $I$ and $J$ and the distance between them
by defining%
\begin{equation*}
\mathcal{P}_{d}^{t}\equiv \left\{ \left( I,J\right) \in \mathcal{D}\times 
\mathcal{D}:J\text{ is good},\ \ell \left( I\right) =2^{t}\ell \left(
J\right) ,\text{ }J\subset 3I\setminus I,\ 2^{d-1}\ell \left( J\right) \leq 
\limfunc{dist}\left( I,J\right) \leq 2^{d}\ell \left( J\right) \right\} .
\end{equation*}%
Note that the closest a good cube $J$ can come to $I$ is determined by the
goodness inequality, which gives this bound: 
\begin{eqnarray}
&&2^{d}\ell \left( J\right) \geq \limfunc{dist}\left( I,J\right) \geq \frac{1%
}{2}\ell \left( I\right) ^{1-\varepsilon }\ell \left( J\right) ^{\varepsilon
}=\frac{1}{2}2^{t\left( 1-\varepsilon \right) }\ell \left( J\right) ;
\label{d below} \\
&&\text{which implies }d\geq t\left( 1-\varepsilon \right) -1.  \notag
\end{eqnarray}%
We write%
\begin{equation*}
\dsum\limits_{\left( I,J\right) \in \mathcal{P}}\left\langle T_{\sigma
}^{\lambda }\left( \bigtriangleup _{I}^{\sigma }f\right) ,\bigtriangleup
_{J}^{\omega }g\right\rangle _{\omega }=\dsum\limits_{t=\rho }^{\infty }\
\sum_{d=N-\varepsilon t-1}^{N}\sum_{\left( I,J\right) \in \mathcal{P}%
_{d}^{t}}\left\langle T_{\sigma }^{\lambda }\left( \bigtriangleup
_{I}^{\sigma }f\right) ,\bigtriangleup _{J}^{\omega }g\right\rangle _{\omega
},
\end{equation*}%
and for fixed $t$ and $d$, we estimate%
\begin{eqnarray*}
&&\left\vert \dsum\limits_{\left( I,J\right) \in \mathcal{P}%
_{d}^{t}}\left\langle T_{\sigma }^{\lambda }\left( \bigtriangleup
_{I}^{\sigma }f\right) ,\bigtriangleup _{J}^{\omega }g\right\rangle _{\omega
}\right\vert =\left\vert \int_{\mathbb{R}^{n}}\dsum\limits_{\left(
I,J\right) \in \mathcal{P}_{d}^{t}}T_{\sigma }^{\lambda }\left(
\bigtriangleup _{I}^{\sigma }f\right) \left( x\right) \ \bigtriangleup
_{J}^{\omega }g\left( x\right) \ d\omega \left( x\right) \right\vert \\
&=&\left\vert \int_{\mathbb{R}^{n}}\dsum\limits_{J\in \mathcal{D}%
}\bigtriangleup _{J}^{\omega }T_{\sigma }^{\lambda }\left(
\dsum\limits_{I\in \mathcal{D}:\ \left( I,J\right) \in \mathcal{P}%
_{d}^{t}}\bigtriangleup _{I}^{\sigma }f\right) \left( x\right) \
\bigtriangleup _{J}^{\omega }g\left( x\right) \ d\omega \left( x\right)
\right\vert \\
&\leq &\int_{\mathbb{R}^{n}}\left( \dsum\limits_{J\in \mathcal{D}}\left\vert
\bigtriangleup _{J}^{\omega }T_{\sigma }^{\lambda }\left( \dsum\limits_{I\in 
\mathcal{D}:\ \left( I,J\right) \in \mathcal{P}_{d}^{t}}\bigtriangleup
_{I}^{\sigma }f\right) \left( x\right) \right\vert ^{2}\right) ^{\frac{1}{2}%
}\left( \dsum\limits_{J\in \mathcal{D}}\left\vert \bigtriangleup
_{J}^{\omega }g\left( x\right) \right\vert ^{2}\right) ^{\frac{1}{2}}d\omega
\left( x\right) \\
&\lesssim &\left\{ \int_{\mathbb{R}^{n}}\left( \dsum\limits_{J\in \mathcal{D}%
}\left\vert \bigtriangleup _{J}^{\omega }T_{\sigma }^{\lambda }\left(
\dsum\limits_{I\in \mathcal{D}:\ \left( I,J\right) \in \mathcal{P}%
_{d}^{t}}\bigtriangleup _{I}^{\sigma }f\right) \left( x\right) \right\vert
^{2}\right) ^{\frac{p}{2}}d\omega \left( x\right) \right\} ^{\frac{1}{p}} \\
&&\ \ \ \ \ \ \ \ \ \ \ \ \ \ \ \ \ \ \ \ \ \ \ \ \ \times \left\{ \int_{%
\mathbb{R}^{n}}\left( \dsum\limits_{J\in \mathcal{D}}\left\vert
\bigtriangleup _{J}^{\omega }g\left( x\right) \right\vert ^{2}\right) ^{%
\frac{p^{\prime }}{2}}d\omega \left( x\right) \right\} ^{\frac{1}{p^{\prime }%
}}.
\end{eqnarray*}%
Now we use the fact that for a fixed $J$, there are only boundedly many $%
I\in \mathcal{D}$ with $\left( I,J\right) \in \mathcal{P}_{d}^{t}$, which
without loss of generality we can suppose is a single cube $I\left[ J\right] 
$, together with (\ref{d below}) to obtain the estimate 
\begin{eqnarray*}
\left\vert \bigtriangleup _{J}^{\omega }T_{\sigma }^{\alpha }\left(
\bigtriangleup _{I}^{\sigma }f\right) \left( x\right) \right\vert &\lesssim &%
\mathrm{P}^{\lambda }\left( J,\left\vert \bigtriangleup _{I}^{\sigma
}f\right\vert \sigma \right) \mathbf{1}_{J}\left( x\right) =\int_{I}\frac{%
\ell \left( J\right) }{\left( \ell \left( J\right) +\left\vert
y-c_{J}\right\vert \right) ^{n+1-\lambda }}\left\vert \bigtriangleup _{I%
\left[ J\right] }^{\sigma }f\left( y\right) \right\vert d\sigma \left(
y\right) \mathbf{1}_{J}\left( x\right) \\
&\lesssim &\frac{\ell \left( J\right) }{\left( 2^{d}\ell \left( J\right)
\right) ^{n+1-\lambda }}\sum_{I^{\prime }\in \mathfrak{C}_{\mathcal{D}%
}\left( I\left[ J\right] \right) }E_{I^{\prime }}^{\sigma }\left\vert
\bigtriangleup _{I}^{\sigma }f\right\vert \left\vert I^{\prime }\right\vert
_{\sigma }\mathbf{1}_{J}\left( x\right) \\
&\lesssim &\frac{2^{-t\left[ 1-\varepsilon \left( n+1-\lambda \right) \right]
}}{\ell \left( I\right) ^{n-\lambda }}\sum_{I^{\prime }\in \mathfrak{C}_{%
\mathcal{D}}\left( I\left[ J\right] \right) }E_{I^{\prime }}^{\sigma
}\left\vert \bigtriangleup _{I}^{\sigma }f\right\vert \left\vert I^{\prime
}\right\vert _{\sigma }\mathbf{1}_{J}\left( x\right) ,
\end{eqnarray*}%
since%
\begin{equation*}
\frac{\ell \left( J\right) }{\left( 2^{d}\ell \left( J\right) \right)
^{n+1-\lambda }}=\frac{2^{-t}2^{\left( t-d\right) n+1-\lambda }}{\ell \left(
I\left[ J\right] \right) ^{n-\lambda }}\leq \frac{2^{-t}2^{\left(
t\varepsilon +1\right) \left( n+1-\lambda \right) }}{\ell \left( I\left[ J%
\right] \right) ^{n-\lambda }}=2^{n+1-\lambda }\frac{2^{-t\left[
1-\varepsilon \left( n+1-\lambda \right) \right] }}{\ell \left( I\left[ J%
\right] \right) ^{n-\lambda }}.
\end{equation*}%
Thus we have%
\begin{eqnarray*}
&&\left\{ \int_{\mathbb{R}^{n}}\left( \dsum\limits_{J\in \mathcal{D}%
}\left\vert \bigtriangleup _{J}^{\omega }T_{\sigma }^{\lambda }\left(
\dsum\limits_{I\in \mathcal{D}:\ \left( I,J\right) \in \mathcal{P}%
_{d}^{t}}\bigtriangleup _{I}^{\sigma }f\right) \left( x\right) \right\vert
^{2}\right) ^{\frac{p}{2}}d\omega \left( x\right) \right\} ^{\frac{1}{p}} \\
&\lesssim &2^{-t\left[ 1-\varepsilon \left( n+1-\lambda \right) \right]
}\left\{ \int_{\mathbb{R}^{n}}\left( \dsum\limits_{J\in \mathcal{D}%
}\left\vert \sum_{I^{\prime }\in \mathfrak{C}_{\mathcal{D}}\left( I\left[ J%
\right] \right) }E_{I^{\prime }}^{\sigma }\left\vert \bigtriangleup
_{I}^{\sigma }f\right\vert \frac{\left\vert I^{\prime }\right\vert _{\sigma }%
}{\ell \left( I\right) ^{n-\lambda }}\mathbf{1}_{J}\left( x\right)
\right\vert ^{2}\right) ^{\frac{p}{2}}d\omega \left( x\right) \right\} ^{%
\frac{1}{p}} \\
&\lesssim &2^{-t\left[ 1-\varepsilon \left( n+1-\lambda \right) \right]
}\left\{ \int_{\mathbb{R}^{n}}\left( \dsum\limits_{I\in \mathcal{D}%
}\left\vert \sum_{I^{\prime }\in \mathfrak{C}_{\mathcal{D}}\left( I\left[ J%
\right] \right) }E_{I^{\prime }}^{\sigma }\left\vert \bigtriangleup
_{I}^{\sigma }f\right\vert \frac{\left\vert I^{\prime }\right\vert _{\sigma }%
}{\left\vert I\right\vert ^{1-\frac{\lambda }{n}}}\right\vert ^{2}\mathbf{1}%
_{I^{\prime }}\left( x\right) \right) ^{\frac{p}{2}}d\omega \left( x\right)
\right\} ^{\frac{1}{p}} \\
&\lesssim &2^{-t\left[ 1-\varepsilon \left( n+1-\lambda \right) \right]
}A_{p}^{\lambda ,\ell ^{2},\limfunc{offset}}\left( \sigma ,\omega \right)
\left\{ \int_{\mathbb{R}^{n}}\left( \dsum\limits_{I\in \mathcal{D}%
}\sum_{I^{\prime }\in \mathfrak{C}_{\mathcal{D}}\left( I\left[ J\right]
\right) }\left( E_{I^{\prime }}^{\sigma }\left\vert \bigtriangleup
_{I}^{\sigma }f\right\vert \right) ^{2}\mathbf{1}_{I^{\prime }}\left(
x\right) \right) ^{\frac{p}{2}}d\sigma \left( x\right) \right\} ^{\frac{1}{p}%
} \\
&\lesssim &2^{-t\left[ 1-\varepsilon \left( n+1-\lambda \right) \right]
}A_{p}^{\lambda ,\ell ^{2},\limfunc{offset}}\left( \sigma ,\omega \right)
\left\Vert f\right\Vert _{L^{p}\left( \sigma \right) },
\end{eqnarray*}%
and provided $0<\varepsilon <\frac{1}{n+1-\lambda }$, we can sum in $t$ to
complete the proof of (\ref{routine'}).
\end{proof}

\subsection{Comparable form}

We decompose%
\begin{eqnarray*}
\mathsf{B}_{\diagup }\left( f,g\right) &=&\mathsf{B}_{\diagup }^{\func{below}%
}\left( f,g\right) +\mathsf{B}_{\diagup }^{\func{above}}\left( f,g\right) ;
\\
\text{where }\mathsf{B}_{\diagup }^{\func{below}}\left( f,g\right) &\equiv
&\dsum\limits_{I,J\in \mathcal{D}:\ 2^{-\rho }\leq \frac{\ell \left(
J\right) }{\ell \left( I\right) }\leq 1\text{ and }\overline{J}\cap 
\overline{I}=\emptyset }\left\langle T_{\sigma }^{\lambda }\left(
\bigtriangleup _{I}^{\sigma }f\right) ,\left( \bigtriangleup _{J}^{\omega
}g\right) \right\rangle _{\omega } \\
&=&\dsum\limits_{I,J\in \mathcal{D}:\ 2^{-\rho }\leq \frac{\ell \left(
J\right) }{\ell \left( I\right) }\leq 1\text{ and }\overline{J}\cap 
\overline{I}=\emptyset \text{ and }J\subset 3I}\left\langle T_{\sigma
}^{\lambda }\left( \bigtriangleup _{I}^{\sigma }f\right) ,\left(
\bigtriangleup _{J}^{\omega }g\right) \right\rangle _{\omega } \\
&&+\dsum\limits_{I,J\in \mathcal{D}:\ 2^{-\rho }\leq \frac{\ell \left(
J\right) }{\ell \left( I\right) }\leq 1\text{ and }\overline{J}\cap 
\overline{I}=\emptyset \text{ and }J\cap 3I=\emptyset }\left\langle
T_{\sigma }^{\lambda }\left( \bigtriangleup _{I}^{\sigma }f\right) ,\left(
\bigtriangleup _{J}^{\omega }g\right) \right\rangle _{\omega } \\
&\equiv &\mathsf{B}_{\diagup }^{\func{below}\func{near}}\left( f,g\right) +%
\mathsf{B}_{\diagup }^{\func{below}\func{far}}\left( f,g\right) \ .
\end{eqnarray*}%
The second form $\mathsf{B}_{\diagup }^{\func{below}\func{far}}\left(
f,g\right) $ is handled in the same way as the disjoint far form $\mathsf{B}%
_{\cap }^{\func{far}}\left( f,g\right) $ in the previous subsection, and for
the first form $\mathsf{B}_{\diagup }^{\func{below}\func{near}}\left(
f,g\right) $, we write 
\begin{eqnarray*}
&&\left\vert \mathsf{B}_{\diagup }^{\func{below}\func{near}}\left(
f,g\right) \right\vert =\left\vert \int_{\mathbb{R}^{n}}\dsum\limits_{I,J\in 
\mathcal{D}:\ 2^{-\rho }\leq \frac{\ell \left( J\right) }{\ell \left(
I\right) }\leq 1\text{ and }\overline{J}\cap \overline{I}=\emptyset \text{
and }J\subset 3I}T_{\sigma }^{\lambda }\left( \bigtriangleup _{I}^{\sigma
}f\right) \left( x\right) \ \left( \bigtriangleup _{J}^{\omega }g\right)
\left( x\right) \ d\omega \left( x\right) \right\vert \\
&\lesssim &\int_{\mathbb{R}^{n}}\dsum\limits_{I,J\in \mathcal{D}:\ 2^{-\rho
}\leq \frac{\ell \left( J\right) }{\ell \left( I\right) }\leq 1\text{ and }%
\overline{J}\cap \overline{I}=\emptyset \text{ and }J\subset 3I}\left(
\int_{I}\frac{\left\vert \bigtriangleup _{I}^{\sigma }f\left( y\right)
\right\vert }{\left\vert y-x\right\vert ^{n-\lambda }}d\sigma \left(
y\right) \right) \ \left\vert \bigtriangleup _{J}^{\omega }g\left( x\right)
\right\vert \ d\omega \left( x\right) \\
&\lesssim &\int_{\mathbb{R}^{n}}\dsum\limits_{I,J\in \mathcal{D}:\ 2^{-\rho
}\leq \frac{\ell \left( J\right) }{\ell \left( I\right) }\leq 1\text{ and }%
\overline{J}\cap \overline{I}=\emptyset \text{ and }J\subset 3I}\left( \frac{%
1}{\left\vert I\right\vert _{\sigma }}\int_{I}\left\vert \bigtriangleup
_{I}^{\sigma }f\right\vert d\sigma \right) \frac{\left\vert I\right\vert
_{\sigma }}{\left\vert I\right\vert ^{1-\frac{\lambda }{n}}}\mathbf{1}%
_{3I}\left( x\right) \ \left\vert \bigtriangleup _{J}^{\omega }g\left(
x\right) \right\vert \ d\omega \left( x\right) ,
\end{eqnarray*}%
and so by the Cauchy-Schwarz inequality, we have%
\begin{eqnarray*}
\left\vert \mathsf{B}_{\diagup }^{\func{below}\func{near}}\left( f,g\right)
\right\vert &\lesssim &\int_{\mathbb{R}^{n}}\left( \dsum\limits_{\substack{ %
I,J\in \mathcal{D}:\ 2^{-\rho }\leq \frac{\ell \left( J\right) }{\ell \left(
I\right) }\leq 1  \\ \overline{J}\cap \overline{I}=\emptyset ,\ J\subset 3I}}%
\left\vert \left( \frac{1}{\left\vert I\right\vert _{\sigma }}%
\int_{I}\left\vert \bigtriangleup _{I}^{\sigma }f\right\vert d\sigma \right) 
\frac{\left\vert I\right\vert _{\sigma }}{\left\vert I\right\vert ^{1-\frac{%
\lambda }{n}}}\right\vert ^{2}\mathbf{1}_{3I}\left( x\right) \right) ^{\frac{%
1}{2}} \\
&&\ \ \ \ \ \ \ \ \ \ \ \ \ \ \ \ \ \ \ \ \ \ \ \ \ \ \ \ \ \ \times \left(
\dsum\limits_{\substack{ I,J\in \mathcal{D}:\ 2^{-\rho }\leq \frac{\ell
\left( J\right) }{\ell \left( I\right) }\leq 1  \\ \overline{J}\cap 
\overline{I}=\emptyset ,\ J\subset 3I}}\left\vert \bigtriangleup
_{J}^{\omega }g\left( x\right) \right\vert ^{2}\right) ^{\frac{1}{2}}\ \
d\omega \left( x\right) \\
&\leq &\left\Vert \left( \dsum\limits_{\substack{ I,J\in \mathcal{D}:\
2^{-\rho }\leq \frac{\ell \left( J\right) }{\ell \left( I\right) }\leq 1  \\ 
\overline{J}\cap \overline{I}=\emptyset ,\ J\subset 3I}}\left\vert \left( 
\frac{1}{\left\vert I\right\vert _{\sigma }}\int_{I}\left\vert
\bigtriangleup _{I}^{\sigma }f\right\vert d\sigma \right) \frac{\left\vert
I\right\vert _{\sigma }}{\left\vert I\right\vert ^{1-\frac{\lambda }{n}}}%
\right\vert ^{2}\mathbf{1}_{3I}\left( x\right) \right) ^{\frac{1}{2}%
}\right\Vert _{L^{p}\left( \omega \right) }\left\Vert \mathcal{S}_{\func{Haar%
}}g\right\Vert _{L^{p^{\prime }}\left( \omega \right) }
\end{eqnarray*}%
and%
\begin{eqnarray*}
&&\left\Vert \left( \dsum\limits_{\substack{ I,J\in \mathcal{D}:\ 2^{-\rho
}\leq \frac{\ell \left( J\right) }{\ell \left( I\right) }\leq 1  \\ 
\overline{J}\cap \overline{I}=\emptyset ,\ J\subset 3I}}\left\vert \left( 
\frac{1}{\left\vert I\right\vert _{\sigma }}\int_{I}\left\vert
\bigtriangleup _{I}^{\sigma }f\right\vert d\sigma \right) \frac{\left\vert
I\right\vert _{\sigma }}{\left\vert I\right\vert ^{1-\frac{\lambda }{n}}}%
\right\vert ^{2}\mathbf{1}_{3I}\left( x\right) \right) ^{\frac{1}{2}%
}\right\Vert _{L^{p}\left( \omega \right) } \\
&\lesssim &A_{p}^{\lambda ,\ell ^{2},\limfunc{offset}}\left( \sigma ,\omega
\right) \left\Vert \left( \dsum\limits_{\substack{ I,J\in \mathcal{D}:\
2^{-\rho }\leq \frac{\ell \left( J\right) }{\ell \left( I\right) }\leq 1  \\ 
\overline{J}\cap \overline{I}=\emptyset ,\ J\subset 3I}}\left( \frac{1}{%
\left\vert I\right\vert _{\sigma }}\int_{I}\left\vert \bigtriangleup
_{I}^{\sigma }f\right\vert d\sigma \right) ^{2}\mathbf{1}_{3I}\left(
x\right) \right) ^{\frac{1}{2}}\right\Vert _{L^{p}\left( \sigma \right) },
\end{eqnarray*}%
and by the Fefferman-Stein maximal inequality in the space of homogeneous
type $\left( \mathbb{R}^{n},\sigma \right) $, where $\sigma $ is doubling (%
\cite{GrLiYa})%
\begin{eqnarray*}
&&\left\Vert \left( \dsum\limits_{\substack{ I,J\in \mathcal{D}:\ 2^{-\rho
}\leq \frac{\ell \left( J\right) }{\ell \left( I\right) }\leq 1  \\ 
\overline{J}\cap \overline{I}=\emptyset ,\ J\subset 3I}}\left( \frac{1}{%
\left\vert 3I\right\vert _{\sigma }}\int_{3I}\left\vert \bigtriangleup
_{I}^{\sigma }f\right\vert d\sigma \right) ^{2}\mathbf{1}_{3I}\left(
x\right) \right) ^{\frac{1}{2}}\right\Vert _{L^{p}\left( \sigma \right)
}\lesssim \left\Vert \left( \dsum\limits_{\substack{ I,J\in \mathcal{D}:\
2^{-\rho }\leq \frac{\ell \left( J\right) }{\ell \left( I\right) }\leq 1  \\ 
\overline{J}\cap \overline{I}=\emptyset ,\ J\subset 3I}}\left[ \mathcal{M}%
_{\sigma }\left\vert \bigtriangleup _{I}^{\sigma }f\right\vert \left(
x\right) \right] ^{2}\right) ^{\frac{1}{2}}\right\Vert _{L^{p}\left( \sigma
\right) } \\
&&\ \ \ \ \ \ \ \ \ \ \ \ \ \ \ \ \ \ \ \ \ \ \ \ \ \ \ \ \ \ \ \ \ \ \ \ \
\ \ \ \ \ \ \ \ \lesssim \left\Vert \left( \dsum\limits_{\substack{ I,J\in 
\mathcal{D}:\ 2^{-\rho }\leq \frac{\ell \left( J\right) }{\ell \left(
I\right) }\leq 1  \\ \overline{J}\cap \overline{I}=\emptyset ,\ J\subset 3I}}%
\left\vert \bigtriangleup _{I}^{\sigma }f\left( x\right) \right\vert
^{2}\right) ^{\frac{1}{2}}\right\Vert _{L^{p}\left( \sigma \right) }\lesssim
\left\Vert \mathcal{S}_{\func{Haar}}f\right\Vert _{L^{p}\left( \sigma
\right) }.
\end{eqnarray*}

Altogether, since both $\left\Vert \mathcal{S}_{\func{Haar}}f\right\Vert
_{L^{p}\left( \sigma \right) }\approx \left\Vert f\right\Vert _{L^{p}\left(
\sigma \right) }$ and $\left\Vert \mathcal{S}_{\func{Haar}}g\right\Vert
_{L^{p^{\prime }}\left( \omega \right) }\approx \left\Vert g\right\Vert
_{L^{p^{\prime }}\left( \omega \right) }$ by square function estimates, we
have controlled the norms of the below forms $\mathsf{B}_{\diagup }^{\func{%
below}\func{near}}\left( f,g\right) $ and $\mathsf{B}_{\diagup }^{\func{below%
}\func{far}}\left( f,g\right) $ by the quadratic offset Muckenhoupt constant 
$A_{p}^{\lambda ,\ell ^{2},\limfunc{offset}}\left( \sigma ,\omega \right) $,
hence%
\begin{equation}
\left\vert \mathsf{B}_{\diagup }^{\func{below}}\left( f,g\right) \right\vert
\lesssim A_{p}^{\lambda ,\ell ^{2},\limfunc{offset}}\left( \sigma ,\omega
\right) \left\Vert f\right\Vert _{L^{p}\left( \sigma \right) }\left\Vert
g\right\Vert _{L^{p^{\prime }}\left( \omega \right) }.  \label{routine''}
\end{equation}%
Finally, the form $\mathsf{B}_{\diagup }^{\func{above}}\left( f,g\right) $
is handled in dual fashion to $\mathsf{B}_{\diagup }^{\func{below}}\left(
f,g\right) $.

\begin{description}
\item[Porism] It is important to note that from the proof given, we may
replace the sum%
\begin{equation*}
\dsum\limits_{I,J\in \mathcal{D}:\ 2^{-\rho }\leq \frac{\ell \left( J\right) 
}{\ell \left( I\right) }\leq 1\text{ and }\overline{J}\cap \overline{I}%
=\emptyset \text{ and }J\subset 3I}
\end{equation*}%
in the left hand side of (\ref{routine''}) with a sum over any \emph{subset}
of the pairs $I,J\,$arising in $\mathsf{B}_{\diagup }^{\func{below}}\left(
f,g\right) $. A similar remark of course applies to $\mathsf{B}_{\diagup }^{%
\func{above}}\left( f,g\right) $.
\end{description}

\subsection{Stopping form}

We assume that $\sigma $ and $\omega $ are doubling measures. We will use a
variant of the Haar stopping form argument due to Nazarov, Treil and Volberg 
\cite{NTV4} to bound the stopping form by local quadratic testing $\mathfrak{%
T}_{\mathbf{R}^{\lambda },p}^{\ell ^{2},\func{loc}}\left( \sigma ,\omega
\right) $ and offset Muckenhoupt $A_{p}^{\lambda ,\ell ^{2},\limfunc{offset}%
}\left( \sigma ,\omega \right) $ constants defined in (\ref{quad HV})\ and (%
\ref{quad A2 tailless'}) respectively. We start the proof by pigeonholing
the ratio of side lengths of $I$ and $J$ in the local stopping forms:%
\begin{align*}
& \mathsf{B}_{\limfunc{stop}}^{F}\left( f,g\right) \equiv \sum_{I\in 
\mathcal{C}_{F}}\sum_{I^{\prime }\in \mathfrak{C}_{\mathcal{D}}\left(
I\right) }\sum_{\substack{ J\in \mathcal{C}_{F}^{\tau -\limfunc{shift}}  \\ %
J\subset I^{\prime }\text{ and }J\Subset _{\rho ,\varepsilon }I}}%
\left\langle \mathbf{1}_{I^{\prime }}\bigtriangleup _{I}^{\sigma }f\
T_{\sigma }^{\lambda }\mathbf{1}_{F\setminus I^{\prime }},\bigtriangleup
_{J}^{\omega }g\right\rangle _{\omega } \\
& =\sum_{I\in \mathcal{C}_{F}}\sum_{I^{\prime }\in \mathfrak{C}_{\mathcal{D}%
}\left( I\right) }\sum_{\substack{ J\in \mathcal{C}_{F}^{\tau -\limfunc{shift%
}}  \\ J\subset I^{\prime }\text{ and }J\Subset _{\rho ,\varepsilon }I}}%
\left\langle \bigtriangleup _{I}^{\sigma }f\ T_{\sigma }^{\lambda }\mathbf{1}%
_{F\setminus I^{\prime }},\bigtriangleup _{J}^{\omega }g\right\rangle
_{\omega } \\
& =\sum_{s=0}^{\infty }\sum_{I\in \mathcal{C}_{F}}\sum_{I^{\prime }\in 
\mathfrak{C}_{\mathcal{D}}\left( I\right) }\sum_{\substack{ J\in \mathcal{C}%
_{F}^{\tau -\limfunc{shift}}\text{and }\ell \left( J\right) =2^{-s}\ell
\left( I\right)  \\ J\subset I^{\prime }\text{ and }J\Subset _{\rho
,\varepsilon }I}}\left\langle \bigtriangleup _{J}^{\omega }\left[ \left(
\bigtriangleup _{I}^{\sigma }f\right) \ T_{\sigma }^{\lambda }\mathbf{1}%
_{F\setminus I^{\prime }}\right] ,\bigtriangleup _{J}^{\omega
}g\right\rangle _{\omega }\ .
\end{align*}%
Now we write $J\prec _{s}I^{\prime }$ when $\pi _{\mathcal{D}}I^{\prime }\in 
\mathcal{C}_{F}^{\tau -\limfunc{shift}}$ and%
\begin{equation*}
J\in \mathcal{C}_{F}^{\tau -\limfunc{shift}}\text{, }\ell \left( J\right)
=2^{-s}\ell \left( I\right) \text{, }J\subset I^{\prime }\text{ and }%
J\Subset _{\rho ,\varepsilon }I,
\end{equation*}%
so that we have%
\begin{eqnarray*}
&&\mathsf{B}_{\limfunc{stop}}\left( f,g\right) =\sum_{F\in \mathcal{F}}%
\mathsf{B}_{\limfunc{stop}}^{F}\left( f,g\right) \\
&=&\sum_{s=0}^{\infty }\sum_{F\in \mathcal{F}}\sum_{I\in \mathcal{C}%
_{F}}\sum_{I^{\prime }\in \mathfrak{C}_{\mathcal{D}}\left( I\right) }\sum 
_{\substack{ J\in \mathcal{C}_{F}^{\tau -\limfunc{shift}}\text{and }\ell
\left( J\right) =2^{-s}\ell \left( I\right)  \\ J\subset I^{\prime }\text{
and }J\Subset _{\rho ,\varepsilon }I}}\left\langle \bigtriangleup
_{J}^{\omega }\left[ \left( \bigtriangleup _{I}^{\sigma }f\right) \
T_{\sigma }^{\lambda }\mathbf{1}_{F\setminus I^{\prime }}\right]
,\bigtriangleup _{J}^{\omega }g\right\rangle _{\omega } \\
&=&\sum_{s=0}^{\infty }\sum_{J\in \mathcal{D}}\left\langle \sum_{F\in 
\mathcal{F}}\sum_{I\in \mathcal{C}_{F}}\sum_{I^{\prime }\in \mathfrak{C}_{%
\mathcal{D}}\left( I\right) :\ J\prec _{s}I^{\prime }}\bigtriangleup
_{J}^{\omega }\left[ \left( \bigtriangleup _{I}^{\sigma }f\right) \
T_{\sigma }^{\lambda }\mathbf{1}_{F\setminus I^{\prime }}\right]
,\bigtriangleup _{J}^{\omega }g\right\rangle _{\omega } \\
&=&\sum_{s=0}^{\infty }\int_{\mathbb{R}^{n}}\sum_{J\in \mathcal{D}}\left(
\sum_{F\in \mathcal{F}}\sum_{I\in \mathcal{C}_{F}}\sum_{I^{\prime }\in 
\mathfrak{C}_{\mathcal{D}}\left( I\right) :\ J\prec _{s}I^{\prime
}}\bigtriangleup _{J}^{\omega }\left[ \left( \bigtriangleup _{I}^{\sigma
}f\right) \ T_{\sigma }^{\lambda }\mathbf{1}_{F\setminus I^{\prime }}\right]
\left( x\right) \right) \bigtriangleup _{J}^{\omega }g\left( x\right) \
d\omega \left( x\right) \\
&\equiv &\sum_{s=0}^{\infty }\mathsf{B}_{\limfunc{stop};s}\left( f,g\right) .
\end{eqnarray*}%
But now we observe that if $J\subset I^{\prime }$ then $\bigtriangleup
_{I}^{\sigma }f$ is a constant on $J$ and so (\ref{analogue}) and (\ref%
{analogue'}), together with the observation that $\bigtriangleup
_{J}^{\omega }\left[ \left( \bigtriangleup _{I}^{\sigma }f\right) \
T_{\sigma }^{\lambda }\mathbf{1}_{F\setminus I^{\prime }}\right] $ has a
vanishing moment, yield the following inequality, 
\begin{equation}
\left\vert \bigtriangleup _{J}^{\omega }\left[ \left( \bigtriangleup
_{I}^{\sigma }f\right) \ T_{\sigma }^{\lambda }\mathbf{1}_{F\setminus
I^{\prime }}\right] \left( x\right) \right\vert \lesssim \left\Vert \mathbb{E%
}_{I^{\prime }}^{\sigma }\bigtriangleup _{I}^{\sigma }f\right\Vert _{\infty
}\ \mathrm{P}^{\lambda }\left( J,\mathbf{1}_{F\setminus I^{\prime }}\sigma
\right) \ \mathbf{1}_{J}\left( x\right) \lesssim E_{I^{\prime }}^{\sigma
}\left\vert \bigtriangleup _{I}^{\sigma }f\right\vert \ \mathrm{P}^{\lambda
}\left( J,\mathbf{1}_{F\setminus I^{\prime }}\sigma \right) \ \mathbf{1}%
_{J}\left( x\right) .  \label{gap}
\end{equation}

Now we can obtain geometric decay in $s$. Indeed, applying Cauchy-Schwarz we
obtain for each $s$,%
\begin{eqnarray*}
&&\mathsf{B}_{\limfunc{stop};s}\left( f,g\right) =\int_{\mathbb{R}%
^{n}}\sum_{J\in \mathcal{D}}\left( \sum_{F\in \mathcal{F}}\sum_{I\in 
\mathcal{C}_{F}}\sum_{I^{\prime }\in \mathfrak{C}_{\mathcal{D}}\left(
I\right) :\ J\prec _{s}I^{\prime }}\bigtriangleup _{J}^{\omega
}\bigtriangleup _{I}^{\sigma }f\left( x\right) T_{\sigma }^{\lambda }\mathbf{%
1}_{F\setminus I^{\prime }}\left( x\right) \right) \bigtriangleup
_{J}^{\omega }g\left( x\right) \ d\omega \left( x\right) \\
&\leq &\int_{\mathbb{R}^{n}}\left( \sum_{J\in \mathcal{D}}\left( \sum_{F\in 
\mathcal{F}}\sum_{I\in \mathcal{C}_{F}}\sum_{I^{\prime }\in \mathfrak{C}_{%
\mathcal{D}}\left( I\right) :\ J\prec _{s}I^{\prime }}E_{I^{\prime
}}^{\sigma }\left\vert \bigtriangleup _{I}^{\sigma }f\right\vert \ \mathrm{P}%
^{\lambda }\left( J,\mathbf{1}_{F\setminus I^{\prime }}\sigma \right) \ 
\mathbf{1}_{J}\left( x\right) \right) ^{2}\right) ^{\frac{1}{2}}\left(
\sum_{J\in \mathcal{D}}\left\vert \bigtriangleup _{J}^{\omega }g\left(
x\right) \right\vert ^{2}\right) ^{\frac{1}{2}}d\omega \left( x\right) \\
&\leq &\left\Vert S\left( x\right) \right\Vert _{L^{p}\left( \omega \right)
}\left\Vert \left( \sum_{J\in \mathcal{D}}\left\vert \bigtriangleup
_{J}^{\omega }g\left( x\right) \right\vert ^{2}\right) ^{\frac{1}{2}%
}\right\Vert _{L^{p^{\prime }}\left( \omega \right) }; \\
&&\text{where }S\left( x\right) ^{2}\equiv \sum_{J\in \mathcal{D}}\left(
\sum_{F\in \mathcal{F}}\sum_{I\in \mathcal{C}_{F}}\sum_{I^{\prime }\in 
\mathfrak{C}_{\mathcal{D}}\left( I\right) :\ J\prec _{s}I^{\prime
}}E_{I^{\prime }}^{\sigma }\left\vert \bigtriangleup _{I}^{\sigma
}f\right\vert \ \mathrm{P}^{\lambda }\left( J,\mathbf{1}_{F\setminus
I^{\prime }}\sigma \right) \ \mathbf{1}_{J}\left( x\right) \right) ^{2}.
\end{eqnarray*}%
For fixed $x\in J$, the pigeonholing above yields $I=\pi _{\mathcal{D}%
}^{\left( s\right) }J$ and $F=\pi _{\mathcal{F}}\pi _{\mathcal{D}}^{\left(
s\right) }J$, and thus we obtain%
\begin{eqnarray*}
S\left( x\right) ^{2} &\equiv &\sum_{J\in \mathcal{D}}\left( \sum_{F\in 
\mathcal{F}}\sum_{I\in \mathcal{C}_{F}}\sum_{I^{\prime }\in \mathfrak{C}_{%
\mathcal{D}}\left( I\right) :\ J\prec _{s}I^{\prime }}E_{I^{\prime
}}^{\sigma }\left\vert \bigtriangleup _{I}^{\sigma }f\right\vert \ \mathrm{P}%
^{\lambda }\left( J,\mathbf{1}_{F\setminus I^{\prime }}\sigma \right) \ 
\mathbf{1}_{J}\left( x\right) \right) ^{2} \\
&\lesssim &\sum_{J\in \mathcal{D}}\left( E_{\pi _{\mathcal{D}}^{\left(
s-1\right) }J}^{\sigma }\left\vert \bigtriangleup _{\pi _{\mathcal{D}%
}^{\left( s\right) }J}^{\sigma }f\right\vert \right) ^{2}\mathrm{P}^{\lambda
}\left( J,\mathbf{1}_{\pi _{\mathcal{F}}\pi _{\mathcal{D}}^{\left( s\right)
}J\setminus \pi _{\mathcal{D}}^{\left( s-1\right) }J}\sigma \right) ^{2}%
\mathbf{1}_{J}\left( x\right) ,
\end{eqnarray*}%
and now using the Poisson inequality with%
\begin{equation*}
\eta \equiv 1-\varepsilon \left( n+1-\lambda \right) >0,
\end{equation*}%
we obtain%
\begin{eqnarray*}
S\left( x\right) ^{2} &\lesssim &2^{-2\eta s}\sum_{J\in \mathcal{D}}\left(
E_{\pi _{\mathcal{D}}^{\left( s-1\right) }J}^{\sigma }\left\vert
\bigtriangleup _{\pi _{\mathcal{D}}^{\left( s\right) }J}^{\sigma
}f\right\vert \right) ^{2}\mathrm{P}^{\lambda }\left( \pi _{\mathcal{D}%
}^{\left( s-1\right) }J,\mathbf{1}_{\pi _{\mathcal{F}}\pi _{\mathcal{D}%
}^{\left( s\right) }J\setminus \pi _{\mathcal{D}}^{\left( s-1\right)
}J}\sigma \right) ^{2}\mathbf{1}_{J}\left( x\right) \\
&\lesssim &2^{-2\eta s}\sum_{J\in \mathcal{D}}\left\vert \widehat{f}\left(
\pi _{\mathcal{D}}^{\left( s\right) }J\right) \right\vert ^{2}\frac{1}{%
\left\vert \pi _{\mathcal{D}}^{\left( s-1\right) }J\right\vert _{\sigma }}%
\mathrm{P}^{\lambda }\left( \pi _{\mathcal{D}}^{\left( s-1\right) }J,\mathbf{%
1}_{\pi _{\mathcal{F}}\pi _{\mathcal{D}}^{\left( s\right) }J\setminus \pi _{%
\mathcal{D}}^{\left( s-1\right) }J}\sigma \right) ^{2}\mathbf{1}_{J}\left(
x\right) .
\end{eqnarray*}%
Since $E_{J}^{\sigma }\left\vert \bigtriangleup _{I}^{\sigma }\right\vert
\lesssim E_{I^{\prime }}^{\sigma }\left\vert \bigtriangleup _{I}^{\sigma
}f\right\vert $ by (\ref{analogue}) and (\ref{analogue'}), we have%
\begin{equation}
S\left( x\right) ^{2}\lesssim 2^{-2\eta s}\sum_{F\in \mathcal{F}}\sum_{I\in 
\mathcal{C}_{F}}\sum_{I^{\prime }\in \mathfrak{C}_{\mathcal{D}}\left(
I\right) }\left\vert \widehat{f}\left( I\right) \right\vert ^{2}\frac{1}{%
\left\vert I^{\prime }\right\vert _{\sigma }}\mathrm{P}^{\lambda }\left(
I^{\prime },\mathbf{1}_{F\setminus I^{\prime }}\sigma \right) ^{2}\mathbf{1}%
_{I^{\prime }}\left( x\right) ,  \label{S2}
\end{equation}%
Then from inequality (\ref{SLp replacement}) below we get, 
\begin{equation}
\left\Vert S\left( x\right) \right\Vert _{L^{p}\left( \omega \right)
}\lesssim 2^{-\eta s}\left( \mathfrak{T}_{\mathbf{R}^{\lambda },p}^{\ell
^{2};\func{loc}}\left( \sigma ,\omega \right) +A_{p}^{\lambda ,\ell ^{2},%
\limfunc{offset}}\left( \sigma ,\omega \right) \right) \left\Vert
f\right\Vert _{L^{p}\left( \sigma \right) }\ .  \label{SLp}
\end{equation}

Finally then, by H\"{o}lder's inequality we obtain 
\begin{equation*}
\left\vert \mathsf{B}_{\limfunc{stop};s}\left( f,g\right) \right\vert
\lesssim \left\Vert S\left( x\right) \right\Vert _{L^{p}\left( \omega
\right) }\left\Vert \left( \sum_{J\in \mathcal{D}}\left\vert \bigtriangleup
_{J}^{\omega }g\left( x\right) \right\vert ^{2}\right) ^{\frac{1}{2}%
}\right\Vert _{L^{p^{\prime }}\left( \omega \right) }\lesssim 2^{-\eta
s}\left( \mathfrak{T}_{\mathbf{R}^{\lambda },p}^{\ell ^{2};\func{loc}}\left(
\sigma ,\omega \right) +A_{p}^{\lambda ,\ell ^{2},\limfunc{offset}}\left(
\sigma ,\omega \right) \right) \left\Vert f\right\Vert _{L^{p}\left( \sigma
\right) }\left\Vert g\right\Vert _{L^{p^{\prime }}\left( \omega \right) }\ ,
\end{equation*}%
and provided $\varepsilon <\frac{1}{n+1-\lambda }$, i.e. $\eta >0$, summing
in $s$ gives%
\begin{equation*}
\left\vert \mathsf{B}_{\limfunc{stop}}\left( f,g\right) \right\vert \leq
\sum_{s=0}^{\infty }\left\vert \mathsf{B}_{\limfunc{stop};s}\left(
f,g\right) \right\vert \lesssim C_{n,\lambda }\left( \mathfrak{T}_{\mathbf{R}%
^{\lambda },p}^{\ell ^{2};\func{loc}}\left( \sigma ,\omega \right)
+A_{p}^{\lambda ,\ell ^{2},\limfunc{offset}}\left( \sigma ,\omega \right)
\right) \left\Vert f\right\Vert _{L^{p}\left( \sigma \right) }\left\Vert
g\right\Vert _{L^{p^{\prime }}\left( \omega \right) }\ .
\end{equation*}

It remains to justify (\ref{SLp}). Since $E_{J}^{\sigma }\left\vert
\bigtriangleup _{I}^{\sigma }\right\vert \leq \left\Vert \mathbf{1}%
_{J}\left\vert \bigtriangleup _{I}^{\sigma }\right\vert \right\Vert _{\infty
}\leq \left\Vert \mathbf{1}_{I^{\prime }}\left\vert \bigtriangleup
_{I}^{\sigma }\right\vert \right\Vert _{\infty }\lesssim E_{I^{\prime
}}^{\sigma }\left\vert \bigtriangleup _{I}^{\sigma }f\right\vert $ by (\ref%
{analogue}) and (\ref{analogue'}), we have%
\begin{equation*}
S\left( x\right) ^{2}\lesssim 2^{-2\eta s}\sum_{F\in \mathcal{F}}\sum_{I\in 
\mathcal{C}_{F}}\sum_{I^{\prime }\in \mathfrak{C}_{\mathcal{D}}\left(
I\right) }\left( E_{I^{\prime }}^{\sigma }\left\vert \bigtriangleup
_{I}^{\sigma }f\right\vert \right) ^{2}\mathrm{P}^{\lambda }\left( I^{\prime
},\mathbf{1}_{F\setminus I^{\prime }}\sigma \right) ^{2}\mathbf{1}%
_{I^{\prime }}\left( x\right) .
\end{equation*}%
We claim that%
\begin{equation}
\left\Vert \left( \sum_{F\in \mathcal{F}}\sum_{I\in \mathcal{C}%
_{F}}\sum_{I^{\prime }\in \mathfrak{C}_{\mathcal{D}}\left( I\right) }\left(
E_{I^{\prime }}^{\sigma }\left\vert \bigtriangleup _{I}^{\sigma
}f\right\vert \right) ^{2}\mathrm{P}^{\lambda }\left( I^{\prime },\mathbf{1}%
_{F\setminus I^{\prime }}\sigma \right) ^{2}\mathbf{1}_{I^{\prime }}\left(
x\right) \right) ^{\frac{1}{2}}\right\Vert _{L^{p}\left( \omega \right)
}\lesssim \left( \mathfrak{T}_{\mathbf{R}^{\lambda },p}^{\ell ^{2};\func{loc}%
}\left( \sigma ,\omega \right) +A_{p}^{\lambda ,\ell ^{2},\limfunc{offset}%
}\left( \sigma ,\omega \right) \right) \left\Vert f\right\Vert _{L^{p}\left(
\sigma \right) }.  \label{claim that}
\end{equation}%
With this established we will then obtain, 
\begin{eqnarray}
&&\left\Vert S\left( x\right) \right\Vert _{L^{p}\left( \omega \right)
}\lesssim 2^{-\eta s}\left\Vert \left( \sum_{F\in \mathcal{F}}\sum_{I\in 
\mathcal{C}_{F}}\sum_{I^{\prime }\in \mathfrak{C}_{\mathcal{D}}\left(
I\right) }\left( E_{I^{\prime }}^{\sigma }\left\vert \bigtriangleup
_{I}^{\sigma }f\right\vert \right) ^{2}\mathrm{P}^{\lambda }\left( I^{\prime
},\mathbf{1}_{F\setminus I^{\prime }}\sigma \right) ^{2}\mathbf{1}%
_{I^{\prime }}\left( x\right) \right) ^{\frac{1}{2}}\right\Vert
_{L^{p}\left( \omega \right) }  \label{SLp replacement} \\
&\lesssim &2^{-\eta s}\left( \mathfrak{T}_{\mathbf{R}^{\lambda },p}^{\ell
^{2};\func{loc}}\left( \sigma ,\omega \right) +A_{p}^{\lambda ,\ell ^{2},%
\limfunc{offset}}\left( \sigma ,\omega \right) \right) \left\Vert
f\right\Vert _{L^{p}\left( \sigma \right) }\ ,  \notag
\end{eqnarray}%
which gives (\ref{SLp}).

In order to prove (\ref{claim that}), we will need a stronger notion of
energy reversal, which we now describe. But first we recall the definition
of strong energy reversal from \cite{SaShUr9}. We say that a vector $\mathbf{%
T}^{\lambda }=\left\{ T_{\ell }^{\lambda }\right\} _{\ell =1}^{2}$ of $%
\lambda $-fractional transforms has \emph{strong} reversal of $\omega $%
-energy on a cube $J$ if there is a positive constant $C_{0}$ such that for
all $2\leq \gamma \leq 2^{\mathbf{r}\left( 1-\varepsilon \right) }$ and for
all positive measures $\mu $ supported outside $\gamma J$, we have the
inequality%
\begin{equation}
\mathbb{E}_{J}^{\omega }\left[ \left( \mathbf{x}-\mathbb{E}_{J}^{\omega }%
\mathbf{x}\right) ^{2}\right] \left( \frac{\mathrm{P}^{\lambda }\left( J,\mu
\right) }{\left\vert J\right\vert ^{\frac{1}{n}}}\right) ^{2}=\mathsf{E}%
\left( J,\omega \right) ^{2}\mathrm{P}^{\lambda }\left( J,\mu \right)
^{2}\leq C_{0}\ \mathbb{E}_{J}^{\omega }\left\vert \mathbf{T}^{\lambda }\mu -%
\mathbb{E}_{J}^{\omega }\mathbf{T}^{\lambda }\mu \right\vert ^{2}.
\label{will fail}
\end{equation}%
We now introduce a \emph{stronger} notion of energy reversal which we call
extreme energy reversal. We say that a vector $\mathbf{T}^{\lambda }=\left\{
T_{\ell }^{\lambda }\right\} _{\ell =1}^{2}$ of $\lambda $-fractional
transforms in has \emph{extreme} reversal of $\omega $-energy on a cube $J$
if there is a Haar function $h_{J}^{\omega }\left( x\right) $ and a positive
constant $C_{0}$, such that for all $2\leq \gamma \leq 2^{\mathbf{r}\left(
1-\varepsilon \right) }$ and for all positive measures $\mu $ supported
outside $\gamma J$, we have the inequality,%
\begin{eqnarray}
&&\mathbb{E}_{J}^{\omega }\left[ \left( \mathbf{x}-\mathbb{E}_{J}^{\omega }%
\mathbf{x}\right) ^{2}\right] \left( \frac{\mathrm{P}^{\lambda }\left( J,\mu
\right) }{\left\vert J\right\vert ^{\frac{1}{n}}}\right) ^{2}\left\vert
J\right\vert _{\omega }=\mathsf{E}\left( J,\omega \right) ^{2}\mathrm{P}%
^{\lambda }\left( J,\mu \right) ^{2}\left\vert J\right\vert _{\omega }
\label{will fail extreme} \\
&&\ \ \ \ \ \ \ \ \ \ \leq C\left\vert \int_{J}\int_{\mathbb{R}^{n}\setminus
\gamma J}\left[ \mathbf{K}^{\lambda }\left( x,y\right) -\mathbf{K}^{\lambda
}\left( c_{J},y\right) \right] h_{J}^{\omega }\left( x\right) d\mu \left(
y\right) d\omega \left( x\right) \right\vert ^{2},  \notag
\end{eqnarray}%
and $\mathbf{K}^{\lambda }$ is the kernel of $\mathbf{T}^{\lambda }$. Note
that (\ref{will fail extreme}) is weaker than (\ref{will fail}) in that the
there is no absolute value inside the integral, and the difference of
kernels $\mathbf{K}^{\lambda }\left( x,y\right) -\mathbf{K}^{\lambda }\left(
c_{J},y\right) $ is multiplied by a single Haar function $h_{J}^{\omega
}\left( x\right) $.

Clearly extreme reversal of energy implies strong reversal of energy. We
prove below that extreme reversal of energy holds for the vector $\lambda $%
-fractional Riesz transform $\mathbf{R}^{\lambda ,n}$ in $\mathbb{R}^{n}$.
But first we will use extreme reversal of energy to prove (\ref{claim that}).

\begin{lemma}
Suppose $\sigma $ and $\omega $ are doubling measures on $\mathbb{R}^{n}$.
Then (\ref{claim that}) holds for $1<p<\infty $, $0\leq \lambda <n$ and $f$
and $\mathcal{F}$ as above.
\end{lemma}

\begin{proof}
For each $I^{\prime }$ we first write $F\setminus I^{\prime }=\left( \gamma
I^{\prime }\setminus I^{\prime }\right) \cup \left( F\setminus \gamma
I^{\prime }\right) $ and $\mathrm{P}^{\lambda }\left( I^{\prime },\mathbf{1}%
_{F\setminus I^{\prime }}\sigma \right) =\mathrm{P}^{\lambda }\left(
I^{\prime },\mathbf{1}_{\gamma I^{\prime }\setminus I^{\prime }}\sigma
\right) +\mathrm{P}^{\lambda }\left( I^{\prime },\mathbf{1}_{F\setminus
\gamma I^{\prime }}\sigma \right) $. For convenience, we sometimes write $%
a_{I^{\prime }}=E_{I^{\prime }}^{\sigma }\left\vert \bigtriangleup
_{I}^{\sigma }f\right\vert $, and only use $E_{I^{\prime }}^{\sigma
}\left\vert \bigtriangleup _{I}^{\sigma }f\right\vert $ when it matters.
Because $\sigma $ is doubling we have $\mathrm{P}_{1}^{\lambda }\left(
I^{\prime },\mathbf{1}_{\gamma I^{\prime }\setminus I^{\prime }}\sigma
\right) \approx \frac{\left\vert I^{\prime }\right\vert _{\sigma }}{%
\left\vert I^{\prime }\right\vert ^{1-\frac{\lambda }{n}}}$, and 
\begin{eqnarray*}
&&\left\Vert \left( \sum_{F\in \mathcal{F}}\sum_{I\in \mathcal{C}%
_{F}}\sum_{I^{\prime }\in \mathfrak{C}_{\mathcal{D}}\left( I\right)
}a_{I^{\prime }}^{2}\mathrm{P}^{\lambda }\left( I^{\prime },\mathbf{1}%
_{\gamma I^{\prime }\setminus I^{\prime }}\sigma \right) ^{2}\mathbf{1}%
_{I^{\prime }}\left( x\right) \right) ^{\frac{1}{2}}\right\Vert
_{L^{p}\left( \omega \right) }\approx \left\Vert \left( \sum_{F\in \mathcal{F%
}}\sum_{I\in \mathcal{C}_{F}}\sum_{I^{\prime }\in \mathfrak{C}_{\mathcal{D}%
}\left( I\right) }a_{I^{\prime }}^{2}\left( \frac{\left\vert I^{\prime
}\right\vert _{\sigma }}{\left\vert I^{\prime }\right\vert ^{1-\frac{\lambda 
}{n}}}\right) ^{2}\mathbf{1}_{I^{\prime }}\left( x\right) \right) ^{\frac{1}{%
2}}\right\Vert _{L^{p}\left( \omega \right) } \\
&\leq &A_{p}^{\lambda ,\ell ^{2},\limfunc{offset}}\left( \sigma ,\omega
\right) \left\Vert \left( \sum_{F\in \mathcal{F}}\sum_{I\in \mathcal{C}%
_{F}}\sum_{I^{\prime }\in \mathfrak{C}_{\mathcal{D}}\left( I\right) }\left(
E_{I^{\prime }}^{\sigma }\left\vert \bigtriangleup _{I}^{\sigma
}f\right\vert \right) ^{2}\mathbf{1}_{I^{\prime }}\left( x\right) \right) ^{%
\frac{1}{2}}\right\Vert _{L^{p}\left( \sigma \right) }\leq A_{p}^{\lambda
,\ell ^{2},\limfunc{offset}}\left( \sigma ,\omega \right) \left\Vert
f\right\Vert _{L^{p}\left( \sigma \right) }\ .
\end{eqnarray*}%
To handle the remaining term involving $\mathrm{P}^{\lambda }\left(
I^{\prime },\mathbf{1}_{F\setminus \gamma I^{\prime }}\sigma \right) $ we
will use extreme reversal of energy inequality for the vector $\lambda $%
-fractional Riesz transform $\mathbf{R}^{\lambda ,n}$ in $\mathbb{R}^{n}$.
Since $\omega $ is doubling, we have $\mathsf{E}\left( I^{\prime },\omega
\right) \approx 1$, and so by the Fefferman-Stein vector valued maximal
inequality,%
\begin{eqnarray*}
&&\left\Vert \left( \sum_{F\in \mathcal{F}}\sum_{I\in \mathcal{C}%
_{F}}\sum_{I^{\prime }\in \mathfrak{C}_{\mathcal{D}}\left( I\right)
}a_{I^{\prime }}^{2}\mathrm{P}^{\lambda }\left( I^{\prime },\mathbf{1}%
_{F\setminus \gamma I^{\prime }}\sigma \right) ^{2}\mathbf{1}_{I^{\prime
}}\right) ^{\frac{1}{2}}\right\Vert _{L^{p}\left( \omega \right) } \\
&\approx &\left\Vert \left( \sum_{F\in \mathcal{F}}\sum_{I\in \mathcal{C}%
_{F}}\sum_{I^{\prime }\in \mathfrak{C}_{\mathcal{D}}\left( I\right)
}a_{I^{\prime }}^{2}\mathrm{P}^{\lambda }\left( I^{\prime },\mathbf{1}%
_{F\setminus \gamma I^{\prime }}\sigma \right) ^{2}\mathsf{E}\left(
I^{\prime },\omega \right) ^{2}\mathbf{1}_{I^{\prime }}\right) ^{\frac{1}{2}%
}\right\Vert _{L^{p}\left( \omega \right) } \\
&\lesssim &\left\Vert \left( \sum_{F\in \mathcal{F}}\sum_{I\in \mathcal{C}%
_{F}}\sum_{I^{\prime }\in \mathfrak{C}_{\mathcal{D}}\left( I\right)
}a_{I^{\prime }}^{2}\left\vert \int_{I^{\prime }}\int_{F\setminus \gamma
I^{\prime }}\left[ \mathbf{K}^{\lambda }\left( x,y\right) -\mathbf{K}%
^{\lambda }\left( c_{I_{i}},y\right) \right] \frac{h_{I^{\prime }}^{\omega
}\left( x\right) }{\sqrt{\left\vert I^{\prime }\right\vert _{\omega }}}%
d\sigma \left( y\right) d\omega \left( x\right) \right\vert ^{2}\mathbf{1}%
_{I^{\prime }}\right) ^{\frac{1}{2}}\right\Vert _{L^{p}\left( \omega \right)
} \\
&=&\left\Vert \left( \sum_{F\in \mathcal{F}}\sum_{I\in \mathcal{C}%
_{F}}\sum_{I^{\prime }\in \mathfrak{C}_{\mathcal{D}}\left( I\right)
}a_{I^{\prime }}^{2}\left\vert \int_{I^{\prime }}\int_{F\setminus \gamma
I^{\prime }}\mathbf{K}^{\lambda }\left( x,y\right) \frac{h_{I^{\prime
}}^{\omega }\left( x\right) }{\sqrt{\left\vert I^{\prime }\right\vert
_{\omega }}}d\sigma \left( y\right) d\omega \left( x\right) \right\vert ^{2}%
\mathbf{1}_{I^{\prime }}\right) ^{\frac{1}{2}}\right\Vert _{L^{p}\left(
\omega \right) } \\
&=&\left\Vert \left( \sum_{F\in \mathcal{F}}\sum_{I\in \mathcal{C}%
_{F}}\sum_{I^{\prime }\in \mathfrak{C}_{\mathcal{D}}\left( I\right)
}a_{I^{\prime }}^{2}\left\vert \left\langle \mathbf{R}_{\sigma }^{\lambda }%
\mathbf{1}_{F\setminus \gamma I^{\prime }},\frac{h_{I^{\prime }}^{\omega
}\left( x\right) }{\sqrt{\left\vert I^{\prime }\right\vert _{\omega }}}%
\right\rangle _{\omega }\right\vert ^{2}\mathbf{1}_{I^{\prime }}\right) ^{%
\frac{1}{2}}\right\Vert _{L^{p}\left( \omega \right) },
\end{eqnarray*}%
which is at most%
\begin{eqnarray}
&&\left\Vert \left( \sum_{F\in \mathcal{F}}\sum_{I\in \mathcal{C}%
_{F}}\sum_{I^{\prime }\in \mathfrak{C}_{\mathcal{D}}\left( I\right)
}a_{I^{\prime }}^{2}\left\vert \left\langle \mathbf{R}_{\sigma }^{\lambda }%
\mathbf{1}_{F},\frac{h_{I^{\prime }}^{\omega }\left( x\right) }{\sqrt{%
\left\vert I^{\prime }\right\vert _{\omega }}}\right\rangle _{\omega
}\right\vert ^{2}\mathbf{1}_{I^{\prime }}\right) ^{\frac{1}{2}}\right\Vert
_{L^{p}\left( \omega \right) }  \label{lines} \\
&&+\left\Vert \left( \sum_{F\in \mathcal{F}}\sum_{I\in \mathcal{C}%
_{F}}\sum_{I^{\prime }\in \mathfrak{C}_{\mathcal{D}}\left( I\right)
}a_{I^{\prime }}^{2}\left\vert \left\langle \mathbf{R}_{\sigma }^{\lambda }%
\mathbf{1}_{\gamma I^{\prime }},\frac{h_{I^{\prime }}^{\omega }\left(
x\right) }{\sqrt{\left\vert I^{\prime }\right\vert _{\omega }}}\right\rangle
_{\omega }\right\vert ^{2}\mathbf{1}_{I^{\prime }}\right) ^{\frac{1}{2}%
}\right\Vert _{L^{p}\left( \omega \right) }\equiv A+B.  \notag
\end{eqnarray}%
Then using the Fefferman-Stein vector valued maximal inequality in \cite%
{GrLiYa}, first applied to the dyadic operator $M_{\omega }^{\func{dy}}$,
followed by the quadratic testing condition, and finally another application
of the Fefferman-Stein vector valued maximal inequality applied to the
classical operator $M_{\sigma }$, we obtain%
\begin{eqnarray*}
A &=&\left\Vert \left( \sum_{F\in \mathcal{F}}\sum_{I\in \mathcal{C}%
_{F}}\sum_{I^{\prime }\in \mathfrak{C}_{\mathcal{D}}\left( I\right)
}a_{I^{\prime }}^{2}\left\vert \left\langle \mathbf{R}_{\sigma }^{\lambda }%
\mathbf{1}_{\gamma I^{\prime }},\frac{h_{I^{\prime }}^{\omega }\left(
x\right) }{\sqrt{\left\vert I^{\prime }\right\vert _{\omega }}}\right\rangle
_{\omega }\right\vert ^{2}\mathbf{1}_{I^{\prime }}\right) ^{\frac{1}{2}%
}\right\Vert _{L^{p}\left( \omega \right) } \\
&\lesssim &\left\Vert \left( \sum_{F\in \mathcal{F}}\sum_{I\in \mathcal{C}%
_{F}}\sum_{I^{\prime }\in \mathfrak{C}_{\mathcal{D}}\left( I\right)
}a_{I^{\prime }}^{2}\left\vert M_{\omega }^{\func{dy}}\mathbf{1}_{I^{\prime
}}\mathbf{R}_{\sigma }^{\lambda }\mathbf{1}_{\gamma I^{\prime }}\right\vert
^{2}\mathbf{1}_{I^{\prime }}\right) ^{\frac{1}{2}}\right\Vert _{L^{p}\left(
\omega \right) }\lesssim \left\Vert \left( \sum_{F\in \mathcal{F}}\sum_{I\in 
\mathcal{C}_{F}}\sum_{I^{\prime }\in \mathfrak{C}_{\mathcal{D}}\left(
I\right) }a_{I^{\prime }}^{2}\left\vert \mathbf{1}_{I^{\prime }}\mathbf{R}%
_{\sigma }^{\lambda }\mathbf{1}_{\gamma I^{\prime }}\right\vert ^{2}\mathbf{1%
}_{I^{\prime }}\right) ^{\frac{1}{2}}\right\Vert _{L^{p}\left( \omega
\right) } \\
&\lesssim &\mathfrak{T}_{T^{\lambda },p}^{\ell ^{2};\func{loc}}\left( \sigma
,\omega \right) \left\Vert \left( \sum_{F\in \mathcal{F}}\sum_{I\in \mathcal{%
C}_{F}}\sum_{I^{\prime }\in \mathfrak{C}_{\mathcal{D}}\left( I\right)
}a_{I^{\prime }}^{2}\mathbf{1}_{\gamma I^{\prime }}\right) ^{\frac{1}{2}%
}\right\Vert _{L^{p}\left( \sigma \right) }\lesssim \mathfrak{T}_{T^{\lambda
},p}^{\ell ^{2};\func{loc}}\left( \sigma ,\omega \right) \left\Vert \left(
\sum_{F\in \mathcal{F}}\sum_{I\in \mathcal{C}_{F}}\sum_{I^{\prime }\in 
\mathfrak{C}_{\mathcal{D}}\left( I\right) }a_{I^{\prime }}^{2}M_{\sigma }%
\mathbf{1}_{I^{\prime }}\right) ^{\frac{1}{2}}\right\Vert _{L^{p}\left(
\sigma \right) } \\
&\lesssim &\mathfrak{T}_{T^{\lambda },p}^{\ell ^{2};\func{loc}}\left( \sigma
,\omega \right) \left\Vert \left( \sum_{F\in \mathcal{F}}\sum_{I\in \mathcal{%
C}_{F}}\sum_{I^{\prime }\in \mathfrak{C}_{\mathcal{D}}\left( I\right)
}\left( E_{I^{\prime }}^{\sigma }\left\vert \bigtriangleup _{I}^{\sigma
}f\right\vert \right) ^{2}\mathbf{1}_{I^{\prime }}\right) ^{\frac{1}{2}%
}\right\Vert _{L^{p}\left( \sigma \right) }\lesssim \mathfrak{T}_{T^{\lambda
},p}^{\ell ^{2};\func{loc}}\left( \sigma ,\omega \right) \left\Vert
f\right\Vert _{L^{p}\left( \sigma \right) }.
\end{eqnarray*}

In order to estimate term $B$ in (\ref{lines}), we use $\left\vert
E_{I^{\prime }}^{\sigma }\left\vert \bigtriangleup _{I}^{\sigma
}f\right\vert \right\vert \lesssim \alpha _{\mathcal{F}}\left( F\right) $
for $I^{\prime }\in \mathfrak{C}_{\mathcal{D}}\left( I\right) $ and $I\in 
\mathcal{C}_{F}$, which holds since $\sigma $ is doubling, and the
inequality $\left\vert \mathsf{P}_{\mathcal{C}_{F}}^{\omega }\mathbf{R}%
_{\sigma }^{\lambda }\mathbf{1}_{F}\right\vert \lesssim M_{\omega }^{\func{dy%
}}\left( \mathbf{R}_{\sigma }^{\lambda }\mathbf{1}_{F}\right) $, and the
Fefferman-Stein vector valued maximal inequality in \cite{GrLiYa}, to obtain 
\begin{eqnarray*}
B &=&\left\Vert \left( \sum_{F\in \mathcal{F}}\alpha _{\mathcal{F}}\left(
F\right) ^{2}\sum_{I\in \mathcal{C}_{F}}\sum_{I^{\prime }\in \mathfrak{C}_{%
\mathcal{D}}\left( I\right) }\left\vert \left\langle \mathbf{R}_{\sigma
}^{\lambda }\mathbf{1}_{F},\frac{h_{I^{\prime }}^{\omega }\left( x\right) }{%
\sqrt{\left\vert I^{\prime }\right\vert _{\omega }}}\right\rangle _{\omega
}\right\vert ^{2}\mathbf{1}_{I^{\prime }}\right) ^{\frac{1}{2}}\right\Vert
_{L^{p}\left( \omega \right) } \\
&\lesssim &\left\Vert \left( \sum_{F\in \mathcal{F}}\alpha _{\mathcal{F}%
}\left( F\right) ^{2}\sum_{I\in \mathcal{C}_{F}}\sum_{I^{\prime }\in 
\mathfrak{C}_{\mathcal{D}}\left( I\right) }\left\vert \left\langle \mathbf{R}%
_{\sigma }^{\lambda }\mathbf{1}_{F},\frac{h_{I^{\prime }}^{\omega }\left(
x\right) }{\sqrt{\left\vert I^{\prime }\right\vert _{\omega }}}\right\rangle
_{\omega }\frac{h_{I^{\prime }}^{\omega }\left( x\right) }{\left\Vert
h_{I^{\prime }}^{\omega }\right\Vert _{L^{\infty }\left( \omega \right) }}%
\right\vert ^{2}\mathbf{1}_{I^{\prime }}\right) ^{\frac{1}{2}}\right\Vert
_{L^{p}\left( \omega \right) } \\
&=&\left\Vert \left( \sum_{F\in \mathcal{F}}\sum_{I\in \mathcal{C}%
_{F}}\sum_{I^{\prime }\in \mathfrak{C}_{\mathcal{D}}\left( I\right)
}\left\vert \frac{1}{\sqrt{\left\vert I^{\prime }\right\vert _{\omega }}%
\left\Vert h_{I^{\prime }}^{\omega }\right\Vert _{L^{\infty }\left( \omega
\right) }}\alpha _{\mathcal{F}}\left( F\right) \bigtriangleup _{I^{\prime
}}^{\omega }\mathbf{R}_{\sigma }^{\lambda }\mathbf{1}_{F}\right\vert
^{2}\right) ^{\frac{1}{2}}\right\Vert _{L^{p}\left( \omega \right) },
\end{eqnarray*}%
which is approximately%
\begin{eqnarray*}
&\approx &\left\Vert \left( \sum_{F\in \mathcal{F}}\sum_{I\in \mathcal{C}%
_{F}}\sum_{I^{\prime }\in \mathfrak{C}_{\mathcal{D}}\left( I\right)
}\left\vert \alpha _{\mathcal{F}}\left( F\right) \bigtriangleup _{I^{\prime
}}^{\omega }\mathbf{R}_{\sigma }^{\lambda }\mathbf{1}_{F}\right\vert
^{2}\right) ^{\frac{1}{2}}\right\Vert _{L^{p}\left( \omega \right) }\approx
\left\Vert \sum_{F\in \mathcal{F}}\alpha _{\mathcal{F}}\left( F\right)
\sum_{I\in \mathcal{C}_{F}}\sum_{I^{\prime }\in \mathfrak{C}_{\mathcal{D}%
}\left( I\right) }\bigtriangleup _{I^{\prime }}^{\omega }\mathbf{R}_{\sigma
}^{\lambda }\mathbf{1}_{F}\right\Vert _{L^{p}\left( \omega \right) } \\
&=&\left\Vert \sum_{F\in \mathcal{F}}\alpha _{\mathcal{F}}\left( F\right) 
\mathsf{P}_{\mathcal{C}_{F}}^{\omega }\mathbf{R}_{\sigma }^{\lambda }\mathbf{%
1}_{F}\right\Vert _{L^{p}\left( \omega \right) }\approx \left\Vert \left(
\sum_{F\in \mathcal{F}}\alpha _{\mathcal{F}}\left( F\right) ^{2}\left\vert 
\mathsf{P}_{\mathcal{C}_{F}}^{\omega }\mathbf{R}_{\sigma }^{\lambda }\mathbf{%
1}_{F}\right\vert ^{2}\right) ^{\frac{1}{2}}\right\Vert _{L^{p}\left( \omega
\right) } \\
&\lesssim &\left\Vert \left( \sum_{F\in \mathcal{F}}\alpha _{\mathcal{F}%
}\left( F\right) ^{2}\left\vert M_{\omega }^{\func{dy}}\mathbf{1}_{F}\left( 
\mathbf{R}_{\sigma }^{\lambda }\mathbf{1}_{F}\right) \right\vert ^{2}\right)
^{\frac{1}{2}}\right\Vert _{L^{p}\left( \omega \right) }\lesssim \left\Vert
\left( \sum_{F\in \mathcal{F}}\alpha _{\mathcal{F}}\left( F\right)
^{2}\left\vert \mathbf{1}_{F}\mathbf{R}_{\sigma }^{\lambda }\left( \mathbf{1}%
_{F}\right) \right\vert ^{2}\right) ^{\frac{1}{2}}\right\Vert _{L^{p}\left(
\omega \right) } \\
&\lesssim &\mathfrak{T}_{T^{\lambda },p}^{\ell ^{2};\func{loc}}\left( \sigma
,\omega \right) \left\Vert \left( \sum_{F\in \mathcal{F}}\alpha _{\mathcal{F}%
}\left( F\right) ^{2}\mathbf{1}_{F}\right) ^{\frac{1}{2}}\right\Vert
_{L^{p}\left( \sigma \right) }\lesssim \mathfrak{T}_{T^{\lambda },p}^{\ell
^{2};\func{loc}}\left( \sigma ,\omega \right) \left\Vert f\right\Vert
_{L^{p}\left( \sigma \right) },
\end{eqnarray*}%
where the final inequality follows from Theorem \ref{using Carleson}.
\end{proof}

In order to show that extreme reversal of energy holds for the vector Riesz
transform, we will model our argument on some of the material from \cite%
{SaShUr9}, beginning with a calculation of the Laplacian of powers of $%
\left\vert x\right\vert $. An earlier, and somewhat similar and simpler,
argument can be found in \cite{LaWi}, but we do not see how to immediately
adapt that argument to the setting of $L^{p}$.

\subsubsection{Fractional Riesz transforms}

Now we compute for $\beta $ real that%
\begin{eqnarray*}
\bigtriangleup \left\vert x\right\vert ^{\beta } &=&\nabla \cdot \nabla
\left\vert x\right\vert ^{2\frac{\beta }{2}}=\nabla \cdot \left\{ \frac{%
\beta }{2}\left\vert x\right\vert ^{2\left( \frac{\beta }{2}-1\right)
}2x\right\} =\beta \nabla \cdot \left\{ x\left\vert x\right\vert ^{2\frac{%
\beta -2}{2}}\right\} \\
&=&\beta \left\{ \left( \nabla \cdot x\right) \left\vert x\right\vert ^{2%
\frac{\beta -2}{2}}+x\cdot \nabla \left\vert x\right\vert ^{2\frac{\beta -2}{%
2}}\right\} =\beta \left\{ n\left\vert x\right\vert ^{2\frac{\beta -2}{2}%
}+x\cdot \frac{\beta -2}{2}\left\vert x\right\vert ^{2\left( \frac{\beta -2}{%
2}-1\right) }2x\right\} \\
&=&\beta \left\{ n\left\vert x\right\vert ^{\beta -2}+\left( \beta -2\right)
\left\vert x\right\vert ^{2}\left\vert x\right\vert ^{\beta -4}\right\}
=\beta \left( n+\beta -2\right) \left\vert x\right\vert ^{\beta -2}.
\end{eqnarray*}%
The case of interest for us is when $\beta =\alpha -n+1$, since then 
\begin{equation}
\bigtriangleup \left\vert x\right\vert ^{\beta }=\nabla \cdot \nabla
\left\vert x\right\vert ^{\alpha -n+1}=\nabla \cdot \nabla \left\vert
x\right\vert ^{\alpha -n+1}=c_{\alpha ,n}\nabla \cdot \mathbf{K}^{\alpha
,n}\left( x\right) ,  \label{interest}
\end{equation}%
where $\mathbf{K}^{\alpha ,n}$ is the vector convolution kernel of the $%
\alpha $-fractional Riesz transform $\mathbf{R}^{\alpha ,n}$. We conclude
that $\bigtriangleup \left\vert x\right\vert ^{\beta }$ is of one sign for
all $x$, provided $\beta \neq 0$ and $n+\beta -2\neq 0$, i.e. $\alpha \notin
\left\{ 1,n-1\right\} $. The case $\alpha =1$ is not included since $%
\left\vert x\right\vert ^{\alpha -n+1}=\left\vert x\right\vert ^{2-n}$ is
the fundamental solution of the Laplacian for $n>2$ and constant for $n=2$.
The case $\alpha =n-1$ is not included since $\left\vert x\right\vert
^{\alpha -n+1}=1$ is constant.

Thus $z\in J$, we have from (\ref{interest}) with $\mathbf{I}^{\alpha
+1,n}\mu \left( z\right) \equiv \int_{\mathbb{R}^{n}}\left\vert
z-y\right\vert ^{\alpha +1-n}d\mu \left( y\right) $ denoting the convolution
of $\left\vert x\right\vert ^{\alpha +1-n}$ with $\mu $, that 
\begin{equation}
\left\vert \nabla \mathbf{R}^{\alpha ,n}\mu \left( z\right) \right\vert
\gtrsim \left\vert \limfunc{trace}\nabla \mathbf{R}^{\alpha ,n}\mu \left(
z\right) \right\vert =\left\vert \bigtriangleup \mathbf{I}^{\alpha +1,n}\mu
\left( z\right) \right\vert \approx \int_{\mathbb{R}^{n}}\left\vert
y-z\right\vert ^{\alpha -n-1}d\mu \left( y\right) \approx \frac{\mathrm{P}%
^{\alpha }\left( J,\mu \right) }{\ell \left( J\right) },  \label{L control}
\end{equation}%
where we assume that the positive measure $\mu $ is supported outside the
expanded cube $\gamma J$.

Recall that the trace of a matrix is invariant under conjugation by
rotations, and hence is the sum of the eigenvalues of a symmetric matrix. We
now claim that for every $z\in J$, the full matrix gradient $\nabla \mathbf{R%
}^{\alpha ,n}\mu \left( z\right) $ has at least $1$ eigenvalue of size at
least $c\frac{\mathrm{P}^{\alpha }\left( J,\mu \right) }{\ell \left(
J\right) }$. Indeed, if all eigenvalues of the matrix $\nabla \mathbf{R}%
^{\alpha ,n}\mu \left( z\right) $ have size at most $c\frac{\mathrm{P}%
^{\alpha }\left( J,\mu \right) }{\ell \left( J\right) }$, then $\left\vert
\nabla \mathbf{R}^{\alpha ,n}\mu \left( z\right) \right\vert \leq c\frac{%
\mathrm{P}^{\alpha }\left( J,\mu \right) }{\ell \left( J\right) }$, which
contradicts (\ref{L control}) if $c$ is chosen small enough. This proves our
claim, and moreover, it satisfies the quantitative quadratic estimate%
\begin{equation}
\left\vert \xi \cdot \nabla \mathbf{R}^{\alpha ,n}\mu \left( z\right) \xi
\right\vert \geq c\frac{\mathrm{P}^{\alpha }\left( J,\mu \right) }{\ell
\left( J\right) }\left\vert \xi \right\vert ^{2},\ \ \ \ \ \xi \in \mathsf{S}%
_{z},\ \text{for }z\in J.  \label{form est}
\end{equation}%
where $\mathsf{S}_{z}\equiv \limfunc{Span}\mathbf{v}_{z}$, for some $\mathbf{%
v}_{z}\in \mathbb{S}^{n-1}$. Thus to each $z$ in $J$, there corresponds a
unit vector $\mathbf{v}_{z}$ for which%
\begin{equation*}
\left\vert \mathbf{v}_{z}\cdot \nabla \mathbf{R}^{\alpha ,n}\mu \left(
z\right) \mathbf{v}_{z}\right\vert \geq c\frac{\mathrm{P}^{\alpha }\left(
J,\mu \right) }{\ell \left( J\right) }.
\end{equation*}%
However, for $w\in J$ we have%
\begin{eqnarray*}
&&\left\vert \mathbf{v}_{z}\cdot \nabla \mathbf{R}^{\alpha ,n}\mu \left(
z\right) \mathbf{v}_{z}-\mathbf{v}_{z}\cdot \nabla \mathbf{R}^{\alpha ,n}\mu
\left( w\right) \mathbf{v}_{z}\right\vert \leq \left\Vert \nabla \mathbf{R}%
^{\alpha ,n}\mu \left( z\right) -\nabla \mathbf{R}^{\alpha ,n}\mu \left(
w\right) \right\Vert \\
&\leq &\left\Vert \nabla ^{2}\mathbf{R}^{\alpha ,n}\mu \left( \theta
_{z,w}\right) \right\Vert \left\vert z-w\right\vert \leq \int \left\vert
\nabla ^{2}\mathbf{K}^{\alpha ,n}\left( \theta _{z,w},y\right) \right\vert
d\mu \left( y\right) \ \left\vert z-w\right\vert \\
&\leq &\int_{\mathbb{R}^{n}\setminus \gamma J}\frac{1}{\left\vert \theta
_{z,w}-y\right\vert ^{n-\alpha +2}}d\mu \left( y\right) \ \left\vert
z-w\right\vert =\int_{\mathbb{R}^{n}\setminus \gamma J}\frac{1}{\left\vert
\theta _{z,w}-y\right\vert ^{n-\alpha +1}}\frac{\ell \left( J\right) }{%
\left\vert \theta _{z,w}-y\right\vert }d\mu \left( y\right) \ \frac{%
\left\vert z-w\right\vert }{\ell \left( J\right) } \\
&\leq &C_{\gamma }\int_{\mathbb{R}^{n}\setminus \gamma J}\frac{1}{\left\vert
\theta _{z,w}-y\right\vert ^{n-\alpha +1}}d\mu \left( y\right) \ \frac{%
\left\vert z-w\right\vert }{\ell \left( J\right) }\lesssim C_{\gamma }\frac{%
\mathrm{P}^{\alpha }\left( J,\mu \right) }{\ell \left( J\right) }\frac{%
\left\vert z-w\right\vert }{\ell \left( J\right) },
\end{eqnarray*}%
since $\frac{\ell \left( J\right) }{\left\vert \theta _{z,w}-y\right\vert }%
\leq C_{\gamma }$. Thus there is a fixed $m$ such that for each $m^{th}$
order grandchild $J^{\prime }\in \mathfrak{C}_{\mathcal{D}}^{\left( m\right)
}\left( J\right) $, we have upon replacing $z$ by $c_{J^{\prime }}$ above,%
\begin{equation}
\left\vert \mathbf{v}_{c_{J^{\prime }}}\cdot \nabla \mathbf{R}^{\alpha
,n}\mu \left( w\right) \mathbf{v}_{c_{J^{\prime }}}\right\vert \geq c\frac{%
\mathrm{P}^{\alpha }\left( J,\mu \right) }{\ell \left( J\right) },\ \ \ \ \
w\in J^{\prime }\text{,}  \label{direction}
\end{equation}%
i.e. we can use the same unit vector $\mathbf{v}_{c_{J^{\prime }}}$ in place
of $v_{z}$ for all $z\in J^{\prime }$.

\subsubsection{Extreme reversal of energy}

We now show that (\ref{will fail extreme}) holds for the vector Riesz
transform $\mathbf{R}^{\alpha ,n}$.

\begin{lemma}
\label{extreme rev ener}Let $0\leq \alpha <n$ and suppose $\omega $
doubling. Then the $\alpha $-fractional Riesz transform $\mathbf{R}^{\alpha
,n}=\left\{ R_{\ell }^{n,\alpha }\right\} _{\ell =1}^{n}$ has extreme
reversal of $\omega $-energy (\ref{will fail extreme}) on \emph{all} cubes $%
J $ provided $\gamma $ is chosen large enough depending only on $n$ and $%
\alpha $, i.e.,%
\begin{equation}
\mathbb{E}_{J}^{\omega }\left[ \left( \mathbf{x}-\mathbb{E}_{J}^{\omega }%
\mathbf{x}\right) ^{2}\right] \left( \frac{\mathrm{P}^{\alpha }\left( J,\mu
\right) }{\left\vert J\right\vert ^{\frac{1}{n}}}\right) ^{2}\left\vert
J\right\vert _{\omega }\leq C\left\vert \int_{J}\int_{\mathbb{R}%
^{n}\setminus \gamma J}\left[ \mathbf{K}^{\alpha }\left( x,y\right) -\mathbf{%
K}^{\alpha }\left( c_{J},y\right) \right] \frac{h_{J}^{\omega }\left(
x\right) }{\sqrt{\left\vert J\right\vert _{\omega }}}d\mu \left( y\right)
d\omega \left( x\right) \right\vert ^{2}.  \label{p extreme}
\end{equation}
\end{lemma}

\begin{proof}
It suffices to show that (\ref{p extreme}) holds with 
\begin{equation*}
h_{J}^{\omega }\left( x\right) =\sum_{K\in \mathfrak{C}\left( J\right) }a_{K}%
\mathbf{1}_{K}\left( x\right) \text{, where }\left\{ 
\begin{array}{ccc}
a_{K}>0 & \text{ if } & K\text{ lies to the right of center} \\ 
a_{K}<0 & \text{ if } & K\text{ lies to the left of center}%
\end{array}%
\right. ,
\end{equation*}%
and without loss of generality $\mathbf{v}_{c_{J}}=\mathbf{e}_{1}$. To see
this we compute,%
\begin{eqnarray*}
&&\int_{J}\left[ K_{1}^{\alpha ,n}\left( x,y\right) -K_{1}^{\alpha ,n}\left(
c_{J},y\right) \right] \frac{h_{J}^{\omega }\left( x\right) }{\sqrt{%
\left\vert J\right\vert _{\omega }}}d\omega \left( x\right) \\
&=&\int_{J}\left[ \frac{x_{1}-y_{1}}{\left\vert x-y\right\vert ^{n-\alpha +1}%
}-\frac{\left( c_{J}\right) _{1}-y_{1}}{\left\vert c_{J}-y\right\vert
^{n-\alpha +1}}\right] \frac{h_{J}^{\omega }\left( x\right) }{\sqrt{%
\left\vert J\right\vert _{\omega }}}d\omega \left( x\right) \\
&=&\int_{J}\left( x_{1}-y_{1}\right) \left\{ \frac{1}{\left\vert
x-y\right\vert ^{n-\alpha +1}}-\frac{1}{\left\vert c_{J}-y\right\vert
^{n-\alpha +1}}\right\} \frac{h_{J}^{\omega }\left( x\right) }{\sqrt{%
\left\vert J\right\vert _{\omega }}}d\omega \left( x\right) +\int_{J}\left\{ 
\frac{x_{1}-\left( c_{J}\right) _{1}}{\left\vert c_{J}-y\right\vert
^{n-\alpha +1}}\right\} \frac{h_{J}^{\omega }\left( x\right) }{\sqrt{%
\left\vert J\right\vert _{\omega }}}d\omega \left( x\right) \\
&\equiv &A+B.
\end{eqnarray*}%
Now in term $B$ we have $\left( x_{1}-\left( c_{J}\right) _{1}\right) \frac{%
h_{J}^{\omega }\left( x\right) }{\sqrt{\left\vert J\right\vert _{\omega }}}$
is of one sign and so%
\begin{equation*}
\left\vert B\right\vert =\left\vert \int_{J}\left\{ \frac{x_{1}-\left(
c_{J}\right) _{1}}{\left\vert c_{J}-y\right\vert ^{n-\alpha +1}}\right\} 
\frac{h_{J}^{\omega }\left( x\right) }{\sqrt{\left\vert J\right\vert
_{\omega }}}d\omega \left( x\right) \right\vert =\int_{J}\frac{\left\vert
x_{1}-\left( c_{J}\right) _{1}\right\vert }{\left\vert c_{J}-y\right\vert
^{n-\alpha +1}}\left\vert \frac{h_{J}^{\omega }\left( x\right) }{\sqrt{%
\left\vert J\right\vert _{\omega }}}\right\vert d\omega \left( x\right) \geq
c\frac{\ell \left( J\right) }{\left\vert c_{J}-y\right\vert ^{n-\alpha +1}}%
\sqrt{\left\vert J\right\vert },
\end{equation*}%
because $\omega $\ is doubling. On the other hand,%
\begin{eqnarray*}
\left\vert A\right\vert &\leq &\int_{J}\left\vert x_{1}-y_{1}\right\vert
\left\vert \frac{1}{\left\vert x-y\right\vert ^{n-\alpha +1}}-\frac{1}{%
\left\vert c_{J}-y\right\vert ^{n-\alpha +1}}\right\vert \left\vert \frac{%
h_{J}^{\omega }\left( x\right) }{\sqrt{\left\vert J\right\vert _{\omega }}}%
\right\vert d\omega \left( x\right) \\
&\lesssim &\frac{\ell \left( J\right) ^{2}}{\left\vert c_{J}-y\right\vert
^{n-\alpha +2}}\sqrt{\left\vert J\right\vert _{\omega }}=\frac{\ell \left(
J\right) }{\left\vert c_{J}-y\right\vert }\frac{\ell \left( J\right) }{%
\left\vert c_{J}-y\right\vert ^{n-\alpha +1}}\sqrt{\left\vert J\right\vert
_{\omega }}\leq C\frac{1}{\gamma }\frac{\ell \left( J\right) }{\left\vert
c_{J}-y\right\vert ^{n-\alpha +1}}\sqrt{\left\vert J\right\vert _{\omega }}
\end{eqnarray*}%
and so for $\gamma >1$ chosen sufficiently large, we obtain%
\begin{eqnarray*}
\left\vert \int_{J}\left[ K_{1}^{\alpha ,n}\left( x,y\right) -K_{1}^{\alpha
,n}\left( c_{J},y\right) \right] \frac{h_{J}^{\omega }\left( x\right) }{\sqrt%
[p]{\left\vert J\right\vert _{\omega }}}d\omega \left( x\right) \right\vert
&\gtrsim &\left\vert B\right\vert -\left\vert A\right\vert \geq \left( c-C%
\frac{1}{\gamma }\right) \frac{\ell \left( J\right) }{\left\vert
c_{J}-y\right\vert ^{n-\alpha +1}}\sqrt{\left\vert J\right\vert _{\omega }}
\\
&\geq &\frac{c}{2}\frac{\ell \left( J\right) }{\left\vert c_{J}-y\right\vert
^{n-\alpha +1}}\sqrt{\left\vert J\right\vert _{\omega }}.
\end{eqnarray*}

Since $\int_{J}\left[ K_{1}^{\alpha ,n}\left( x,y\right) -K_{1}^{\alpha
,n}\left( c_{J},y\right) \right] \frac{h_{J}^{\omega }\left( x\right) }{%
\sqrt{\left\vert J\right\vert _{\omega }}}d\omega \left( x\right) $ is also
of one sign, it follows that%
\begin{eqnarray*}
&&\left\vert \int_{J}\int_{\mathbb{R}^{n}\setminus \gamma J}\mathbf{v}%
_{c_{J}}\cdot \left[ \mathbf{K}^{\alpha }\left( x,y\right) -\mathbf{K}%
^{\alpha }\left( c_{J},y\right) \right] \frac{h_{J}^{\omega }\left( x\right) 
}{\sqrt{\left\vert J\right\vert _{\omega }}}d\mu \left( y\right) d\omega
\left( x\right) \right\vert \\
&=&\int_{\mathbb{R}^{n}\setminus \gamma J}\left\vert \int_{J}\left[
K_{1}^{\alpha ,n}\left( x,y\right) -K_{1}^{\alpha ,n}\left( c_{J},y\right) %
\right] \frac{h_{J}^{\omega }\left( x\right) }{\sqrt{\left\vert J\right\vert
_{\omega }}}d\omega \left( x\right) \right\vert d\mu \left( y\right) \\
&\geq &\int_{\mathbb{R}^{n}\setminus \gamma J}\frac{c}{2}\frac{\ell \left(
J\right) }{\left\vert c_{J}-y\right\vert ^{n-\alpha +1}}\sqrt{\left\vert
J\right\vert _{\omega }}d\mu \left( y\right) =\frac{c}{2}\sqrt{\left\vert
J\right\vert _{\omega }}\mathrm{P}^{\alpha }\left( J,\mathbf{1}_{\mathbb{R}%
^{n}\setminus \gamma J}\right) ,
\end{eqnarray*}%
which proves the extreme reversal of energy.
\end{proof}

\subsection{Far below form}

Recall that we decomposed the far below form $\mathsf{T}_{\limfunc{far}%
\limfunc{below}}\left( f,g\right) $ as $\mathsf{T}_{\limfunc{far}\limfunc{%
below}}^{1}\left( f,g\right) +\mathsf{T}_{\limfunc{far}\limfunc{below}%
}^{2}\left( f,g\right) $, where we claimed that the second form $\mathsf{T}_{%
\limfunc{far}\limfunc{below}}^{2}\left( f,g\right) $ was controlled by the
disjoint, comparable and adjacent forms and $\mathsf{B}_{\cap }\left(
f,g\right) $, $\mathsf{B}_{\diagup }\left( f,g\right) $ and $\mathsf{B}_{%
\limfunc{adj},\rho }\left( f,g\right) $, upon noting the porisms following (%
\ref{routine'}) and (\ref{routine''}). Indeed, if $\bigtriangleup
_{J}^{\omega }g$ is not identically zero, then $J$ must be good, and in that
case the condition "$J\subset I$ but $J\not\Subset _{\rho ,\varepsilon }I$"
implies that the pair of cubes $I,J$ is included in \textbf{either} the sum
defining the disjoint down form $\mathsf{B}_{\cap }^{\limfunc{down}}\left(
f,g\right) $ \textbf{or} in the sum defining the comparable below form $%
\mathsf{B}_{\diagup }^{\func{below}}\left( f,g\right) $ \textbf{or} in the
sum defining the adjacent below form $\mathsf{B}_{\limfunc{adj},\rho }^{%
\func{below}}\left( f,g\right) $. The first far below form $\mathsf{T}_{%
\limfunc{far}\limfunc{below}}^{1}\left( f,g\right) $ is handled by the
following Intertwining Proposition.

\begin{proposition}[The Intertwining Proposition]
\label{Int Prop}Suppose $\sigma ,\omega $ are positive locally finite Borel
measures on $\mathbb{R}^{n}$, that $\sigma $ is doubling, and that $\mathcal{%
F}$ satisfies a $\sigma $-Carleson condition. Then for a smooth $\lambda $%
-fractional singular integral $T^{\lambda }$, and for $\limfunc{good}$
functions $f\in L^{2}\left( \sigma \right) \cap L^{p}\left( \sigma \right) $
and $g\in L^{2}\left( \omega \right) \cap L^{p^{\prime }}\left( \omega
\right) $, and with $\kappa \geq 1$ sufficiently large, we have the
following bound for $\mathsf{T}_{\limfunc{far}\limfunc{below}}\left(
f,g\right) =\sum_{F\in \mathcal{F}}\ \sum_{I:\ I\supsetneqq F}\ \left\langle
T_{\sigma }^{\alpha }\bigtriangleup _{I}^{\sigma }f,\mathsf{P}_{\mathcal{C}%
_{F}^{\tau -\limfunc{shift}}}^{\omega }g\right\rangle _{\omega }$: 
\begin{equation}
\left\vert \mathsf{T}_{\limfunc{far}\limfunc{below}}^{1}\left( f,g\right)
\right\vert \lesssim A_{p}^{\lambda ,\ell ^{2},\limfunc{offset}}\ \left\Vert
f\right\Vert _{L^{p}\left( \sigma \right) }\left\Vert g\right\Vert
_{L^{p^{\prime }}\left( \omega \right) }.  \label{far below est}
\end{equation}
\end{proposition}

\begin{proof}
For any dyadic cube $I$, let $\theta \left( I\right) $ denote any of the
dyadic siblings of $I$, namely the children of the dyadic parent $\pi I$
other than $I$ itself. We write%
\begin{eqnarray*}
f_{F} &\equiv &\sum_{I:\ I\supsetneqq F}\bigtriangleup _{I}^{\sigma
}f=\sum_{m=1}^{\infty }\sum_{I:\ \pi _{\mathcal{F}}^{m}F\subsetneqq I\subset
\pi _{\mathcal{F}}^{m+1}F}\bigtriangleup _{I}^{\sigma }f \\
&=&\sum_{m=1}^{\infty }\sum_{I:\ \pi _{\mathcal{F}}^{m}F\subsetneqq I\subset
\pi _{\mathcal{F}}^{m+1}F}\mathbf{1}_{\theta \left( I\right) }\left( \mathbb{%
E}_{I}^{\sigma }f-\mathbb{E}_{\pi _{\mathcal{F}}^{m+1}F}^{\sigma }f\right) \\
&=&\sum_{m=1}^{\infty }\sum_{I:\ \pi _{\mathcal{F}}^{m}F\subsetneqq I\subset
\pi _{\mathcal{F}}^{m+1}F}\mathbf{1}_{\theta \left( I\right) }\left( \mathbb{%
E}_{I}^{\sigma }f\right) -\sum_{m=1}^{\infty }\mathbf{1}_{\pi _{\mathcal{F}%
}^{m+1}F\setminus \pi _{\mathcal{F}}^{m}F}\left( \mathbb{E}_{\pi _{\mathcal{F%
}}^{m+1}F}^{\sigma }f\right) \\
&\equiv &\beta _{F}-\gamma _{F}\ ,
\end{eqnarray*}%
and then%
\begin{equation*}
\sum_{F\in \mathcal{F}}\ \left\langle T_{\sigma }^{\lambda
}f_{F},g_{F}\right\rangle _{\omega }=\sum_{F\in \mathcal{F}}\ \left\langle
T_{\sigma }^{\lambda }\beta _{F},g_{F}\right\rangle _{\omega }-\sum_{F\in 
\mathcal{F}}\ \left\langle T_{\sigma }^{\lambda }\gamma
_{F},g_{F}\right\rangle _{\omega }\ .
\end{equation*}%
Now we use the Poisson inequality (\ref{e.Jsimeq}), namely 
\begin{equation*}
\mathrm{P}^{\lambda }\left( J,\sigma \mathbf{1}_{K\setminus I}\right)
\lesssim \left( \frac{\ell \left( J\right) }{\ell \left( I\right) }\right)
^{1-\varepsilon \left( n+1-\lambda \right) }\mathrm{P}^{\lambda }\left(
I,\sigma \mathbf{1}_{K\setminus I}\right) ,
\end{equation*}%
to obtain that%
\begin{eqnarray*}
&&\left\vert \sum_{F\in \mathcal{F}}\left\langle T_{\sigma }^{\lambda
}\gamma _{F},g_{F}\right\rangle _{\omega }\right\vert =\left\vert \sum_{F\in 
\mathcal{F}}\int_{\mathbb{R}^{n}}T_{\sigma }^{\lambda }\left(
\sum_{m=1}^{\infty }\mathbf{1}_{\pi _{\mathcal{F}}^{m+1}F\setminus \pi _{%
\mathcal{F}}^{m}F}\left( \mathbb{E}_{\pi _{\mathcal{F}}^{m+1}F}^{\sigma
}f\right) \right) \left( x\right) \ \left( \sum_{J\in \mathcal{C}%
_{F}^{\omega ,\tau \text{-}\func{shift}}}\bigtriangleup _{J}^{\omega
}g\left( x\right) \right) \ d\omega \left( x\right) \right\vert \\
&=&\left\vert \int_{\mathbb{R}^{n}}\sum_{J\in \mathcal{D}}\left\{ \sum_{F\in 
\mathcal{F}}\bigtriangleup _{J}^{\omega }T_{\sigma }^{\lambda }\left(
\sum_{m=1}^{\infty }\mathbf{1}_{\pi _{\mathcal{F}}^{m+1}F\setminus \pi _{%
\mathcal{F}}^{m}F}\left( \mathbb{E}_{\pi _{\mathcal{F}}^{m+1}F}^{\sigma
}f\right) \right) \left( x\right) \ \bigtriangleup _{J}^{\omega }g\left(
x\right) \right\} d\omega \left( x\right) \right\vert \\
&\leq &\int_{\mathbb{R}^{n}}\left( \sum_{J\in \mathcal{D}}\left\vert
\sum_{F\in \mathcal{F}}\bigtriangleup _{J}^{\omega }T_{\sigma }^{\lambda
}\left( \sum_{m=1}^{\infty }\mathbf{1}_{\pi _{\mathcal{F}}^{m+1}F\setminus
\pi _{\mathcal{F}}^{m}F}\left( \mathbb{E}_{\pi _{\mathcal{F}%
}^{m+1}F}^{\sigma }f\right) \right) \left( x\right) \right\vert ^{2}\right)
^{\frac{1}{2}}\left( \sum_{J\in \mathcal{D}}\left\vert \bigtriangleup
_{J}^{\omega }g\left( x\right) \right\vert ^{2}\right) ^{\frac{1}{2}}d\omega
\left( x\right) \\
&\leq &\left\Vert \left( \sum_{J\in \mathcal{D}}\left\vert \sum_{F\in 
\mathcal{F}}\bigtriangleup _{J}^{\omega }T_{\sigma }^{\lambda }\left(
\sum_{m=1}^{\infty }\mathbf{1}_{\pi _{\mathcal{F}}^{m+1}F\setminus \pi _{%
\mathcal{F}}^{m}F}\left( \mathbb{E}_{\pi _{\mathcal{F}}^{m+1}F}^{\sigma
}f\right) \right) \left( x\right) \right\vert ^{2}\right) ^{\frac{1}{2}%
}\right\Vert _{L^{p}\left( \omega \right) }\left\Vert \left( \sum_{J\in 
\mathcal{D}}\left\vert \bigtriangleup _{J}^{\omega }g\left( x\right)
\right\vert ^{2}\right) ^{\frac{1}{2}}\right\Vert _{L^{p^{\prime }}\left(
\omega \right) },
\end{eqnarray*}%
where the second factor is equivalent to $\left\Vert g\right\Vert
_{L^{p^{\prime }}\left( \omega \right) }$, and then using the Pivotal Lemma %
\ref{ener}, the first factor $S$ is dominated by%
\begin{eqnarray*}
S &\lesssim &\left\Vert \left( \sum_{J\in \mathcal{D}}\left\vert \sum_{F\in 
\mathcal{F}:\ J\in \mathcal{C}_{F}^{\omega ,\tau \text{-}\func{shift}%
}}\sum_{m=1}^{\infty }\mathrm{P}^{\lambda }\left( J,\mathbf{1}_{\pi _{%
\mathcal{F}}^{m+1}F\setminus \pi _{\mathcal{F}}^{m}F}\left\vert \mathbb{E}%
_{\pi _{\mathcal{F}}^{m+1}F}^{\sigma }f\right\vert \sigma \right)
\right\vert ^{2}\mathbf{1}_{J}\right) ^{\frac{1}{2}}\right\Vert
_{L^{p}\left( \omega \right) } \\
&=&\left\Vert \left( \sum_{J\in \mathcal{D}}\left\vert \sum_{m=1}^{\infty
}\sum_{F\in \mathcal{F}:\ J\in \mathcal{C}_{F}^{\omega ,\tau \text{-}\func{%
shift}}}\left\Vert \mathbb{E}_{\pi _{\mathcal{F}}^{m+1}F}^{\sigma
}f\right\Vert _{\infty }\left( \frac{\ell \left( J\right) }{\ell \left( \pi
_{\mathcal{F}}^{m}F\right) }\right) ^{1-\varepsilon \left( n+1-\lambda
\right) }\mathrm{P}_{\kappa }^{\lambda }\left( \pi _{\mathcal{F}}^{m}F,%
\mathbf{1}_{\pi _{\mathcal{F}}^{m+1}F\setminus \pi _{\mathcal{F}%
}^{m}F}\sigma \right) \right\vert ^{2}\mathbf{1}_{J}\right) ^{\frac{1}{2}%
}\right\Vert _{L^{p}\left( \omega \right) } \\
&\leq &\sum_{m=1}^{\infty }\left\Vert \left( \sum_{J\in \mathcal{D}%
}\left\vert \sum_{F\in \mathcal{F}:\ J\in \mathcal{C}_{F}^{\omega ,\tau 
\text{-}\func{shift}}}\left\Vert \mathbb{E}_{\pi _{\mathcal{F}%
}^{m+1}F}^{\sigma }f\right\Vert _{\infty }\left( \frac{\ell \left( J\right) 
}{\ell \left( \pi _{\mathcal{F}}^{m}F\right) }\right) ^{1-\varepsilon \left(
n+1-\lambda \right) }\frac{\left\vert \pi _{\mathcal{F}}^{m}F\right\vert
_{\sigma }}{\left\vert \pi _{\mathcal{F}}^{m}F\right\vert ^{1-\frac{\lambda 
}{n}}}\right\vert ^{2}\mathbf{1}_{J}\left( x\right) \right) ^{\frac{1}{2}%
}\right\Vert _{L^{p}\left( \omega \right) },
\end{eqnarray*}%
where in the last line we have used (\ref{kappa large}). Now we note that
for each $J\in \mathcal{D}$ the number of cubes $F\in \mathcal{F}$ such that 
$J\in \mathcal{C}_{F}^{\tau -\func{shift}}$ is at most $\tau $. So without
loss of generality, we may simply suppose that there is just one such cube
denoted $F\left[ J\right] $. Thus for each $m\in \mathbb{N}$, the above norm
is at most%
\begin{equation*}
\left( 2^{-m}\right) ^{1-\varepsilon \left( n+1-\lambda \right) }\left\Vert
\left( \sum_{J\in \mathcal{D}}\left\vert \left\Vert \mathbb{E}_{\pi _{%
\mathcal{F}}^{m+1}F\left[ J\right] }^{\sigma }f\right\Vert _{\infty }\left( 
\frac{\ell \left( J\right) }{\ell \left( F\left[ J\right] \right) }\right)
^{1-\varepsilon \left( n+1-\lambda \right) }\frac{\left\vert \pi _{\mathcal{F%
}}^{m}F\left[ J\right] \right\vert _{\sigma }}{\left\vert \pi _{\mathcal{F}%
}^{m}F\left[ J\right] \right\vert ^{1-\frac{\lambda }{n}}}\right\vert ^{2}%
\mathbf{1}_{J}\left( x\right) \right) ^{\frac{1}{2}}\right\Vert
_{L^{p}\left( \omega \right) },
\end{equation*}%
and the sum inside the parentheses equals%
\begin{eqnarray*}
&&\sum_{F\in \mathcal{F}}\sum_{J\in \mathcal{C}_{F}^{\omega ,\tau \text{-}%
\func{shift}}:\ x\in J\subset F}\left( \frac{\ell \left( J\right) }{\ell
\left( F\left[ J\right] \right) }\right) ^{1-\varepsilon \left( n+1-\lambda
\right) }\left\vert \left\Vert \mathbb{E}_{\pi _{\mathcal{F}}^{m+1}F\left[ J%
\right] }^{\sigma }f\right\Vert _{\infty }\frac{\left\vert \pi _{\mathcal{F}%
}^{m}F\left[ J\right] \right\vert _{\sigma }}{\left\vert \pi _{\mathcal{F}%
}^{m}F\left[ J\right] \right\vert ^{1-\frac{\lambda }{n}}}\right\vert ^{2}%
\mathbf{1}_{J}\left( x\right) \\
&\lesssim &\sum_{F\in \mathcal{F}}\sum_{J\in \mathcal{C}_{F}^{\omega ,\tau 
\text{-}\func{shift}}:\ x\in J\subset F}\left( \frac{\ell \left( J\right) }{%
\ell \left( F\right) }\right) ^{1-\varepsilon \left( n+1-\lambda \right)
}\left\vert \left\Vert \mathbb{E}_{\pi _{\mathcal{F}}^{m+1}F}^{\sigma
}f\right\Vert _{\infty }\frac{\left\vert \pi _{\mathcal{F}}^{m}F\right\vert
_{\sigma }}{\left\vert \pi _{\mathcal{F}}^{m}F\right\vert ^{1-\frac{\lambda 
}{n}}}\right\vert ^{2}\mathbf{1}_{J}\left( x\right) \\
&\lesssim &\sum_{F\in \mathcal{F}}\left\vert \left\Vert \mathbb{E}_{\pi _{%
\mathcal{F}}^{m+1}F}^{\sigma }f\right\Vert _{\infty }\frac{\left\vert \pi _{%
\mathcal{F}}^{m}F\right\vert _{\sigma }}{\left\vert \pi _{\mathcal{F}%
}^{m}F\right\vert ^{1-\frac{\lambda }{n}}}\right\vert ^{2}\mathbf{1}%
_{J}\left( x\right) .
\end{eqnarray*}%
Altogether then, using the quadratic offset $A_{p}^{\lambda ,\ell ^{2},%
\limfunc{offset}}$ condition and doubling, we have%
\begin{eqnarray*}
S &\lesssim &\sum_{m=1}^{\infty }\left( 2^{-m}\right) ^{1-\varepsilon \left(
n+1-\lambda \right) }\left\Vert \left( \sum_{F\in \mathcal{F}}\left\vert
\left\Vert \mathbb{E}_{\pi _{\mathcal{F}}^{m+1}F}^{\sigma }f\right\Vert
_{\infty }\frac{\left\vert \pi _{\mathcal{F}}^{m}F\right\vert _{\sigma }}{%
\left\vert \pi _{\mathcal{F}}^{m}F\right\vert ^{1-\frac{\lambda }{n}}}%
\right\vert ^{2}\mathbf{1}_{F}\left( x\right) \right) ^{\frac{1}{2}%
}\right\Vert _{L^{p}\left( \omega \right) } \\
&\lesssim &A_{p}^{\lambda ,\ell ^{2},\limfunc{offset}}\left( \sigma ,\omega
\right) \sum_{m=1}^{\infty }\left( 2^{-m}\right) ^{1-\varepsilon \left(
n+1-\lambda \right) }\left\Vert \left( \sum_{F\in \mathcal{F}}\left\vert
\bigtriangleup _{\pi _{\mathcal{F}}^{m+1}F}^{\sigma }f\left( x\right)
\right\vert ^{2}\mathbf{1}_{F}\left( x\right) \right) ^{\frac{1}{2}%
}\right\Vert _{L^{p}\left( \sigma \right) },
\end{eqnarray*}%
and we can continue with%
\begin{eqnarray*}
&=&A_{p}^{\lambda ,\ell ^{2},\limfunc{offset}}\left( \sigma ,\omega \right)
\sum_{m=1}^{\infty }\left( 2^{-m}\right) ^{1-\varepsilon \left( n+1-\lambda
\right) }\left\Vert \left( \sum_{G\in \mathcal{F}}\sum_{F\in \mathcal{F}:\
\pi _{\mathcal{F}}^{m+1}F=G}\left\vert \bigtriangleup _{G}^{\sigma }f\left(
x\right) \right\vert ^{2}\mathbf{1}_{F}\left( x\right) \right) ^{\frac{1}{2}%
}\right\Vert _{L^{p}\left( \sigma \right) } \\
&\leq &A_{p}^{\lambda ,\ell ^{2},\limfunc{offset}}\left( \sigma ,\omega
\right) \sum_{m=1}^{\infty }\left( 2^{-m}\right) ^{1-\varepsilon \left(
n+1-\lambda \right) }\left\Vert \left( \sum_{G\in \mathcal{F}}\left\vert
\bigtriangleup _{G}^{\sigma }f\left( x\right) \right\vert ^{2}\mathbf{1}%
_{G}\left( x\right) \right) ^{\frac{1}{2}}\right\Vert _{L^{p}\left( \sigma
\right) } \\
&\leq &A_{p}^{\lambda ,\ell ^{2},\limfunc{offset}}\left( \sigma ,\omega
\right) \sum_{m=1}^{\infty }\left( 2^{-m}\right) ^{1-\varepsilon \left(
n+1-\lambda \right) }\left\Vert f\right\Vert _{L^{p}\left( \sigma \right)
}=C_{\varepsilon ,\lambda }A_{p}^{\lambda ,\ell ^{2},\limfunc{offset}}\left(
\sigma ,\omega \right) \left\Vert f\right\Vert _{L^{p}\left( \sigma \right)
}.
\end{eqnarray*}%
Thus provided $1-\varepsilon >\varepsilon \left( n-\lambda \right) $, we
have proved the estimate%
\begin{equation*}
\left\vert \sum_{F\in \mathcal{F}}\left\langle T_{\sigma }^{\lambda }\gamma
_{F},g_{F}\right\rangle _{\omega }\right\vert \lesssim A_{p}^{\lambda ,\ell
^{2},\limfunc{offset}}\left( \sigma ,\omega \right) \left\Vert f\right\Vert
_{L^{p}\left( \sigma \right) }\left\Vert g\right\Vert _{L^{p^{\prime
}}\left( \omega \right) }.
\end{equation*}

It remains to bound $\sum_{F\in \mathcal{F}}\left\langle T_{\sigma
}^{\lambda }\beta _{F},g_{F}\right\rangle _{\omega }$ where 
\begin{equation*}
\beta _{F}=\sum_{m=1}^{\infty }\sum_{I:\ \pi _{\mathcal{F}}^{m}F\subsetneqq
I\subset \pi _{\mathcal{F}}^{m+1}F}\mathbf{1}_{\theta \left( I\right)
}\left( \mathbb{E}_{I;\kappa }^{\sigma }f\right) \text{ and }g_{F}\left(
x\right) =\sum_{J\in \mathcal{C}_{F}^{\omega ,\tau \text{-}\func{shift}%
}}\bigtriangleup _{J;\kappa }^{\omega }g\left( x\right) .
\end{equation*}%
The difference between the previous estimate and this one is that the
averages $\mathbf{1}_{\pi _{\mathcal{F}}^{m+1}F\setminus \pi _{\mathcal{F}%
}^{m}F}\left\vert \mathbb{E}_{\pi _{\mathcal{F}}^{m+1}F}^{\sigma
}f\right\vert $ inside the Poisson kernel have been replaced with the sum of
averages $\sum_{I:\ \pi _{\mathcal{F}}^{m}F\subsetneqq I\subset \pi _{%
\mathcal{F}}^{m+1}F}\mathbf{1}_{\theta \left( I\right) }\left\vert \mathbb{E}%
_{I}^{\sigma }f\right\vert $, but where the sum is taken over pairwise
disjoint sets $\left\{ \theta \left( I\right) \right\} _{\pi _{\mathcal{F}%
}^{m}F\subsetneqq I\subset \pi _{\mathcal{F}}^{m+1}F}$. Just as in the
previous estimate we start with%
\begin{eqnarray*}
&&\left\vert \sum_{F\in \mathcal{F}}\left\langle T_{\sigma }^{\lambda }\beta
_{F},g_{F}\right\rangle _{\omega }\right\vert =\left\vert \sum_{F\in 
\mathcal{F}}\int_{\mathbb{R}^{n}}T_{\sigma }^{\lambda }\left(
\sum_{m=1}^{\infty }\sum_{I:\ \pi _{\mathcal{F}}^{m}F\subsetneqq I\subset
\pi _{\mathcal{F}}^{m+1}F}\mathbf{1}_{\theta \left( I\right) }\left( \mathbb{%
E}_{I}^{\sigma }f\right) \right) \left( x\right) \ \left( \sum_{J\in 
\mathcal{C}_{F}^{\omega ,\tau \text{-}\func{shift}}}\bigtriangleup
_{J}^{\omega }g\left( x\right) \right) \ d\omega \left( x\right) \right\vert
\\
&=&\left\vert \int_{\mathbb{R}^{n}}\sum_{J\in \mathcal{D}}\left\{ \sum_{F\in 
\mathcal{F}}\bigtriangleup _{J}^{\omega }T_{\sigma }^{\lambda }\left(
\sum_{m=1}^{\infty }\sum_{I:\ \pi _{\mathcal{F}}^{m}F\subsetneqq I\subset
\pi _{\mathcal{F}}^{m+1}F}\mathbf{1}_{\theta \left( I\right) }\left( \mathbb{%
E}_{I}^{\sigma }f\right) \right) \left( x\right) \ \bigtriangleup
_{J}^{\omega }g\left( x\right) \right\} d\omega \left( x\right) \right\vert
\\
&\leq &\int_{\mathbb{R}^{n}}\left( \sum_{J\in \mathcal{D}}\left\vert
\sum_{F\in \mathcal{F}}\bigtriangleup _{J}^{\omega }T_{\sigma }^{\lambda
}\left( \sum_{m=1}^{\infty }\sum_{I:\ \pi _{\mathcal{F}}^{m}F\subsetneqq
I\subset \pi _{\mathcal{F}}^{m+1}F}\mathbf{1}_{\theta \left( I\right)
}\left( \mathbb{E}_{I}^{\sigma }f\right) \right) \left( x\right) \right\vert
^{2}\right) ^{\frac{1}{2}}\left( \sum_{J\in \mathcal{D}}\left\vert
\bigtriangleup _{J}^{\omega }g\left( x\right) \right\vert ^{2}\right) ^{%
\frac{1}{2}}d\omega \left( x\right) \\
&\leq &\left\Vert \left( \sum_{J\in \mathcal{D}}\left\vert \sum_{F\in 
\mathcal{F}}\bigtriangleup _{J}^{\omega }T_{\sigma }^{\lambda }\left(
\sum_{m=1}^{\infty }\sum_{I:\ \pi _{\mathcal{F}}^{m}F\subsetneqq I\subset
\pi _{\mathcal{F}}^{m+1}F}\mathbf{1}_{\theta \left( I\right) }\left( \mathbb{%
E}_{I}^{\sigma }f\right) \right) \left( x\right) \right\vert ^{2}\right) ^{%
\frac{1}{2}}\right\Vert _{L^{p}\left( \omega \right) }\left\Vert \left(
\sum_{J\in \mathcal{D}}\left\vert \bigtriangleup _{J}^{\omega }g\left(
x\right) \right\vert ^{2}\right) ^{\frac{1}{2}}\right\Vert _{L^{p^{\prime
}}\left( \omega \right) }.
\end{eqnarray*}

The second factor is equivalent to $\left\Vert g\right\Vert _{L^{p^{\prime
}}\left( \omega \right) }$, and the first factor $S$ is dominated by%
\begin{eqnarray*}
S &\lesssim &\left\Vert \left( \sum_{J\in \mathcal{D}}\left\vert \sum_{F\in 
\mathcal{F}:\ J\in \mathcal{C}_{F}^{\omega ,\tau \text{-}\func{shift}%
}}\sum_{m=1}^{\infty }\sum_{I:\ \pi _{\mathcal{F}}^{m}F\subsetneqq I\subset
\pi _{\mathcal{F}}^{m+1}F}\mathrm{P}^{\lambda }\left( J,\mathbf{1}_{\theta
\left( I\right) }\left( \mathbb{E}_{I}^{\sigma }f\right) \sigma \right)
\right\vert ^{2}\mathbf{1}_{J}\right) ^{\frac{1}{2}}\right\Vert
_{L^{p}\left( \omega \right) } \\
&\lesssim &\sum_{m=1}^{\infty }\left\Vert \left( \sum_{J\in \mathcal{D}%
}\left\vert \sum_{F\in \mathcal{F}:\ J\in \mathcal{C}_{F}^{\omega ,\tau 
\text{-}\func{shift}}}\sum_{I:\ \pi _{\mathcal{F}}^{m}F\subsetneqq I\subset
\pi _{\mathcal{F}}^{m+1}F}\left\Vert \mathbb{E}_{I}^{\sigma }f\right\Vert
_{\infty }\mathrm{P}^{\lambda }\left( J,\mathbf{1}_{\theta \left( I\right)
}\sigma \right) \right\vert ^{2}\mathbf{1}_{J}\right) ^{\frac{1}{2}%
}\right\Vert _{L^{p}\left( \omega \right) }.
\end{eqnarray*}%
Then we use%
\begin{eqnarray*}
\sum_{I:\ \pi _{\mathcal{F}}^{m}F\subsetneqq I\subset \pi _{\mathcal{F}%
}^{m+1}F}\left\Vert \mathbb{E}_{I}^{\sigma }f\right\Vert _{\infty }\mathrm{P}%
^{\lambda }\left( J,\mathbf{1}_{\theta \left( I\right) }\sigma \right) &\leq
&\left( \sup_{I:\ \pi _{\mathcal{F}}^{m}F\subsetneqq I\subset \pi _{\mathcal{%
F}}^{m+1}F}\left\Vert \mathbb{E}_{I}^{\sigma }f\right\Vert _{\infty }\right) 
\mathrm{P}^{\lambda }\left( J,\sum_{I:\ \pi _{\mathcal{F}}^{m}F\subsetneqq
I\subset \pi _{\mathcal{F}}^{m+1}F}\mathbf{1}_{\theta \left( I\right)
}\sigma \right) \\
&=&\left( \sup_{I:\ \pi _{\mathcal{F}}^{m}F\subsetneqq I\subset \pi _{%
\mathcal{F}}^{m+1}F}\left\Vert \mathbb{E}_{I}^{\sigma }f\right\Vert _{\infty
}\right) \mathrm{P}^{\lambda }\left( J,\mathbf{1}_{\pi _{\mathcal{F}%
}^{m+1}F\setminus \pi _{\mathcal{F}}^{m}F}\sigma \right) ,
\end{eqnarray*}%
and obtain that%
\begin{equation*}
S\lesssim \sum_{m=1}^{\infty }\left\Vert \left( \sum_{J\in \mathcal{D}%
}\left\vert \sum_{F\in \mathcal{F}:\ J\in \mathcal{C}_{F}^{\omega ,\tau 
\text{-}\func{shift}}}\left( \sup_{I:\ \pi _{\mathcal{F}}^{m}F\subsetneqq
I\subset \pi _{\mathcal{F}}^{m+1}F}\left\Vert \mathbb{E}_{I}^{\sigma
}f\right\Vert _{\infty }\right) \mathrm{P}^{\lambda }\left( J,\mathbf{1}%
_{\pi _{\mathcal{F}}^{m+1}F\setminus \pi _{\mathcal{F}}^{m}F}\sigma \right)
\right\vert ^{2}\mathbf{1}_{J}\right) ^{\frac{1}{2}}\right\Vert
_{L^{p}\left( \omega \right) }.
\end{equation*}%
Now we define $G_{m}\left[ F\right] \in \left( \pi _{\mathcal{F}}^{m}F,\pi _{%
\mathcal{F}}^{m+1}F\right] $ so that $\sup_{I:\ \pi _{\mathcal{F}%
}^{m}F\subsetneqq I\subset \pi _{\mathcal{F}}^{m+1}F}\left\Vert \mathbb{E}%
_{I}^{\sigma }f\right\Vert _{\infty }=\left\Vert \mathbb{E}_{G_{m}\left[ F%
\right] }^{\sigma }f\right\Vert _{\infty }$, and dominate $S$ by%
\begin{eqnarray*}
&&\sum_{m=1}^{\infty }\left\Vert \left( \sum_{J\in \mathcal{D}}\left\vert
\sum_{F\in \mathcal{F}:\ J\in \mathcal{C}_{F}^{\omega ,\tau \text{-}\func{%
shift}}}\left\Vert \mathbb{E}_{G_{m}\left[ F\right] }^{\sigma }f\right\Vert
_{\infty }\mathrm{P}^{\lambda }\left( J,\mathbf{1}_{\pi _{\mathcal{F}%
}^{m+1}F\setminus \pi _{\mathcal{F}}^{m}F}\sigma \right) \right\vert ^{2}%
\mathbf{1}_{J}\right) ^{\frac{1}{2}}\right\Vert _{L^{p}\left( \omega \right)
} \\
&\lesssim &\sum_{m=1}^{\infty }\left\Vert \left( \sum_{J\in \mathcal{D}%
}\left\vert \sum_{F\in \mathcal{F}:\ J\in \mathcal{C}_{F}^{\omega ,\tau 
\text{-}\func{shift}}}\left\Vert \mathbb{E}_{G_{m}\left[ F\right] }^{\sigma
}f\right\Vert _{\infty }\left( \frac{\ell \left( J\right) }{\ell \left( G_{m}%
\left[ F\right] \right) }\right) ^{\eta }\mathrm{P}^{\lambda }\left( G_{m}%
\left[ F\right] ,\mathbf{1}_{\pi _{\mathcal{F}}^{m+1}F\setminus \pi _{%
\mathcal{F}}^{m}F}\sigma \right) \right\vert ^{2}\mathbf{1}_{J}\right) ^{%
\frac{1}{2}}\right\Vert _{L^{p}\left( \omega \right) } \\
&\lesssim &\sum_{m=1}^{\infty }2^{-m\eta }\left\Vert \left( \sum_{J\in 
\mathcal{D}}\left\vert \sum_{F\in \mathcal{F}:\ J\in \mathcal{C}_{F}^{\omega
,\tau \text{-}\func{shift}}}\left\Vert \mathbb{E}_{G_{m}\left[ F\right]
}^{\sigma }f\right\Vert _{\infty }\left( \frac{\ell \left( J\right) }{\ell
\left( F\right) }\right) ^{\eta }\mathrm{P}^{\lambda }\left( G_{m}\left[ F%
\right] ,\mathbf{1}_{\pi _{\mathcal{F}}^{m+1}F\setminus \pi _{\mathcal{F}%
}^{m}F}\sigma \right) \right\vert ^{2}\mathbf{1}_{J}\right) ^{\frac{1}{2}%
}\right\Vert _{L^{p}\left( \omega \right) },
\end{eqnarray*}%
where $\eta =1-\varepsilon \left( n+1-\lambda \right) $ is the constant
appearing in (\ref{e.Jsimeq}).

Just as above we note that for each $J\in \mathcal{D}$ the number of cubes $%
F\in \mathcal{F}$ such that $J\in \mathcal{C}_{F}^{\omega ,\tau \text{-}%
\func{shift}}$ is at most $\tau $. So without loss of generality, we may
simply suppose that there is just one such cube denoted $F\left[ J\right] $.
Thus for each $m\in \mathbb{N}$, the above norm is at most%
\begin{equation*}
\left\Vert \left( \sum_{J\in \mathcal{D}}\left\vert \left\Vert \mathbb{E}%
_{G_{m}\left[ F\left[ J\right] \right] }^{\sigma }f\right\Vert _{\infty
}\left( \frac{\ell \left( J\right) }{\ell \left( F\left[ J\right] \right) }%
\right) ^{\eta }\frac{\left\vert G_{m}\left[ F\left[ J\right] \right]
\right\vert _{\sigma }}{\left\vert G_{m}\left[ F\left[ J\right] \right]
\right\vert ^{1-\frac{\lambda }{n}}}\right\vert ^{2}\mathbf{1}_{J}\right) ^{%
\frac{1}{2}}\right\Vert _{L^{p}\left( \omega \right) },
\end{equation*}%
and the sum inside the parentheses equals%
\begin{eqnarray*}
&&\sum_{J\in \mathcal{D}}\left\vert \left\Vert \mathbb{E}_{G_{m}\left[ F%
\left[ J\right] \right] }^{\sigma }f\right\Vert _{\infty }\frac{\left\vert
G_{m}\left[ F\left[ J\right] \right] \right\vert _{\sigma }}{\left\vert G_{m}%
\left[ F\left[ J\right] \right] \right\vert ^{1-\frac{\lambda }{n}}}%
\right\vert ^{2}\left( \frac{\ell \left( J\right) }{\ell \left( F\left[ J%
\right] \right) }\right) ^{2\eta }\mathbf{1}_{J}\left( x\right)  \\
&\lesssim &\left\vert \left\Vert \mathbb{E}_{G_{m}\left[ F\right] }^{\sigma
}f\right\Vert _{\infty }\frac{\left\vert G_{m}\left[ F\right] \right\vert
_{\sigma }}{\left\vert G_{m}\left[ F\right] \right\vert ^{1-\frac{\lambda }{n%
}}}\right\vert ^{2}\mathbf{1}_{G_{m}\left[ F\right] }\left( x\right) .
\end{eqnarray*}%
Altogether then, using the quadratic offset $A_{p}^{\lambda ,\ell ^{2},%
\limfunc{offset}}$ condition and doubling, we have%
\begin{eqnarray*}
&&S\lesssim A_{p}^{\lambda ,\ell ^{2},\limfunc{offset}}\left( \sigma ,\omega
\right) \sum_{m=1}^{\infty }2^{-m\eta }\left\Vert \left( \sum_{F\in \mathcal{%
F}}\left\vert \left\Vert \mathbb{E}_{G_{m}\left[ F\right] }^{\sigma
}f\right\Vert _{\infty }\frac{\left\vert G_{m}\left[ F\left[ J\right] \right]
\right\vert _{\sigma }}{\left\vert G_{m}\left[ F\left[ J\right] \right]
\right\vert ^{1-\frac{\lambda }{n}}}\right\vert ^{2}\mathbf{1}_{G_{m}\left[ F%
\right] }\left( x\right) \right) ^{\frac{1}{2}}\right\Vert _{L^{p}\left(
\sigma \right) } \\
&\lesssim &A_{p}^{\lambda ,\ell ^{2},\limfunc{offset}}\left( \sigma ,\omega
\right) \sum_{m=1}^{\infty }2^{-m\eta }\left\Vert \left( \sum_{F\in \mathcal{%
F}}\left\vert \left\Vert \mathbb{E}_{G_{m}\left[ F\right] }^{\sigma
}f\right\Vert _{\infty }\frac{\left\vert G_{m}\left[ F\right] \right\vert
_{\sigma }}{\left\vert G_{m}\left[ F\right] \right\vert ^{1-\frac{\lambda }{n%
}}}\right\vert ^{2}\mathbf{1}_{G_{m}\left[ F\right] }\left( x\right) \right)
^{\frac{1}{2}}\right\Vert _{L^{p}\left( \sigma \right) } \\
&\lesssim &A_{p}^{\lambda ,\ell ^{2},\limfunc{offset}}\left( \sigma ,\omega
\right) \sum_{m=1}^{\infty }2^{-m\eta }\left\Vert \left( \sum_{F\in \mathcal{%
F}}\left\vert \bigtriangleup _{G_{m}\left[ F\right] }^{\sigma }f\left(
x\right) \right\vert ^{2}\mathbf{1}_{F}\left( x\right) \right) ^{\frac{1}{2}%
}\right\Vert _{L^{p}\left( \sigma \right) }
\end{eqnarray*}%
and we can continue with%
\begin{eqnarray*}
S &\leq &A_{p}^{\lambda ,\ell ^{2},\limfunc{offset}}\left( \sigma ,\omega
\right) \sum_{m=1}^{\infty }2^{-m\eta }\left\Vert \left( \sum_{G\in \mathcal{%
G}}\sum_{F\in \mathcal{F}:\ G_{m}\left[ F\right] =G}\left\vert
\bigtriangleup _{G}^{\sigma }f\left( x\right) \right\vert ^{2}\mathbf{1}%
_{G}\left( x\right) \right) ^{\frac{1}{2}}\right\Vert _{L^{p}\left( \sigma
\right) } \\
&\leq &A_{p}^{\lambda ,\ell ^{2},\limfunc{offset}}\left( \sigma ,\omega
\right) \sum_{m=1}^{\infty }2^{-m\eta }\left\Vert \left( \sum_{G\in \mathcal{%
G}}\left\vert \bigtriangleup _{G}^{\sigma }f\left( x\right) \right\vert ^{2}%
\mathbf{1}_{G}\left( x\right) \right) ^{\frac{1}{2}}\right\Vert
_{L^{p}\left( \sigma \right) } \\
&\leq &A_{p}^{\lambda ,\ell ^{2},\limfunc{offset}}\left( \sigma ,\omega
\right) \sum_{m=1}^{\infty }2^{-m\eta }\left\Vert f\right\Vert _{L^{p}\left(
\sigma \right) }=C_{\varepsilon ,n,\kappa ,\lambda }A_{p}^{\lambda ,\ell
^{2},\limfunc{offset}}\left\Vert f\right\Vert _{L^{p}\left( \sigma \right) },
\end{eqnarray*}%
provided $\eta =1-\varepsilon \left( n+1-\lambda \right) >0$ holds. Thus we
have proved the estimate%
\begin{equation*}
\left\vert \sum_{F\in \mathcal{F}}\left\langle T_{\sigma }^{\lambda }\beta
_{F},g_{F}\right\rangle _{\omega }\right\vert \lesssim A_{p}^{\lambda ,\ell
^{2},\limfunc{offset}}\left( \sigma ,\omega \right) \left\Vert f\right\Vert
_{L^{p}\left( \sigma \right) }\left\Vert g\right\Vert _{L^{p^{\prime
}}\left( \omega \right) },
\end{equation*}%
which together with the corresponding estimate for $\sum_{F\in \mathcal{F}%
}\left\langle T_{\sigma }^{\lambda }\gamma _{F},g_{F}\right\rangle _{\omega }
$ proved above, completes the proof of the Intertwining Proposition.
\end{proof}

Thus we have controlled both the first and second far below forms $\mathsf{T}%
_{\limfunc{far}\limfunc{below}}^{1}\left( f,g\right) $ and $\mathsf{T}_{%
\limfunc{far}\limfunc{below}}^{2}\left( f,g\right) $ by the quadratic offset
Muckenhoupt constant $A_{p}^{\lambda ,\ell ^{2},\limfunc{offset}}$.

\subsection{Neighbour form}

We begin with $M_{I^{\prime }}=\mathbf{1}_{I^{\prime }}\bigtriangleup
_{I}^{\sigma }f$ to obtain 
\begin{align*}
& \mathsf{B}_{\limfunc{neighbour}}\left( f,g\right) =\sum_{F\in \mathcal{F}}%
\mathsf{B}_{\limfunc{neighbour}}^{F}\left( f,g\right) \\
& =\sum_{F\in \mathcal{F}}\sum_{\substack{ I\in \mathcal{C}_{F}\text{ and }%
J\in \mathcal{C}_{F}^{\tau -\limfunc{shift}}  \\ J\Subset _{\rho
,\varepsilon }I}}\sum_{\theta \left( I_{J}\right) \in \mathfrak{C}_{\mathcal{%
D}}\left( I\right) \setminus \left\{ I_{J}\right\} }\int_{\mathbb{R}%
^{n}}T_{\sigma }^{\lambda }\left( \mathbf{1}_{\theta \left( I_{J}\right)
}\bigtriangleup _{I}^{\sigma }f\right) \left( x\right) \ \bigtriangleup
_{J}^{\omega }g\left( x\right) \ d\omega \left( x\right) \\
& =\int_{\mathbb{R}^{n}}\sum_{J\in \mathcal{D}}\left\{ \sum_{F\in \mathcal{F}%
}\sum_{\substack{ I\in \mathcal{C}_{F}\text{ and }J\in \mathcal{C}_{F}^{\tau
-\limfunc{shift}}  \\ J\Subset _{\rho ,\varepsilon }I}}\sum_{\theta \left(
I_{J}\right) \in \mathfrak{C}_{\mathcal{D}}\left( I\right) \setminus \left\{
I_{J}\right\} }\right\} T_{\sigma }^{\lambda }\left( \mathbf{1}_{\theta
\left( I_{J}\right) }\bigtriangleup _{I}^{\sigma }f\right) \left( x\right) \
\bigtriangleup _{J}^{\omega }g\left( x\right) \ d\omega \left( x\right) \\
& =\int_{\mathbb{R}^{n}}\sum_{J\in \mathcal{D}}\sum_{I\succ J}\bigtriangleup
_{J}^{\omega }T_{\sigma }^{\lambda }\left( \mathbf{1}_{\theta \left(
I_{J}\right) }\bigtriangleup _{I}^{\sigma }f\right) \left( x\right) \
\bigtriangleup _{J}^{\omega }g\left( x\right) \ d\omega \left( x\right) ,
\end{align*}%
where for $J\in \mathcal{D}$ we write $I\succ J$ if $I$ satisfies 
\begin{equation*}
\text{there is }F\in \mathcal{F}\text{ such that }I\in \mathcal{C}_{F}\text{%
, }J\in \mathcal{C}_{F}^{\tau -\limfunc{shift}}\text{ and }J\Subset _{\rho
,\varepsilon }I.
\end{equation*}%
Applying the Cauchy-Schwarz and H\"{o}lder inequalities gives%
\begin{eqnarray*}
&&\left\vert \mathsf{B}_{\limfunc{neighbour}}\left( f,g\right) \right\vert
\leq \int_{\mathbb{R}^{n}}\left( \sum_{J\in \mathcal{D}}\sum_{I\succ
J}\left\vert \bigtriangleup _{J}^{\omega }T_{\sigma }^{\lambda }\left( 
\mathbf{1}_{\theta \left( I_{J}\right) }\bigtriangleup _{I}^{\sigma
}f\right) \left( x\right) \right\vert ^{2}\right) ^{\frac{1}{2}}\left(
\sum_{J\in \mathcal{D}}\sum_{I\succ J}\left\vert \bigtriangleup _{J}^{\omega
}g\left( x\right) \right\vert ^{2}\right) ^{\frac{1}{2}}d\omega \left(
x\right) \\
&\leq &\left\Vert \left( \sum_{J\in \mathcal{D}}\sum_{I\succ J}\left\vert
\bigtriangleup _{J}^{\omega }T_{\sigma }^{\lambda }\left( \mathbf{1}_{\theta
\left( I_{J}\right) }\bigtriangleup _{I}^{\sigma }f\right) \left( x\right)
\right\vert ^{2}\right) ^{\frac{1}{2}}\right\Vert _{L^{p}\left( \omega
\right) }\left\Vert \left( \sum_{J\in \mathcal{D}}\sum_{I\succ J}\left\vert
\bigtriangleup _{J}^{\omega }g\left( x\right) \right\vert ^{2}\right) ^{%
\frac{1}{2}}\right\Vert _{L^{p^{\prime }}\left( \omega \right) },
\end{eqnarray*}%
where the final factor is dominated by $\left\Vert g\right\Vert
_{L^{p^{\prime }}\left( \omega \right) }$. Using the Pivotal Lemma (\ref%
{ener}), and the estimate $\left\Vert M_{I^{\prime }}\right\Vert _{L^{\infty
}\left( \sigma \right) }\approx \frac{1}{\sqrt{\left\vert I^{\prime
}\right\vert _{\sigma }}}\left\vert \widehat{f}\left( I\right) \right\vert $
from (\ref{analogue'}), we have%
\begin{align*}
\left\vert \bigtriangleup _{J}^{\omega }T_{\sigma }^{\lambda }\left(
M_{I^{\prime }}\mathbf{1}_{I^{\prime }}\right) \left( x\right) \right\vert &
\lesssim \mathrm{P}^{\lambda }\left( J,\left\Vert M_{I^{\prime }}\right\Vert
_{L^{\infty }\left( \sigma \right) }\mathbf{1}_{I^{\prime }}\sigma \right) 
\mathbf{1}_{J}\left( x\right) \\
& \lesssim \frac{1}{\sqrt{\left\vert I^{\prime }\right\vert _{\sigma }}}%
\left\vert \widehat{f}\left( I\right) \right\vert \mathrm{P}^{\lambda
}\left( J,\mathbf{1}_{I^{\prime }}\sigma \right) \mathbf{1}_{J}\left(
x\right) .
\end{align*}%
Now we pigeonhole the side lengths of $I$ and $J$ by $\ell \left( J\right)
=2^{-s}\ell \left( I\right) $ and use goodness, followed by (\ref{kappa
large}), to obtain%
\begin{eqnarray*}
&&\left\Vert \left( \sum_{J\in \mathcal{D}}\sum_{I\succ J}\left\vert
\bigtriangleup _{J}^{\omega }T_{\sigma }^{\lambda }\left( \mathbf{1}_{\theta
\left( I_{J}\right) }\bigtriangleup _{I}^{\sigma }f\right) \left( x\right)
\right\vert ^{2}\right) ^{\frac{1}{2}}\right\Vert _{L^{p}\left( \omega
\right) } \\
&\lesssim &\left\Vert \left( \sum_{J\in \mathcal{D}}\sum_{I\succ J:\ \ell
\left( J\right) =2^{-s}\ell \left( I\right) }\left\vert \frac{1}{\sqrt{%
\left\vert I^{\prime }\right\vert _{\sigma }}}\left\vert \widehat{f}\left(
I\right) \right\vert \mathrm{P}^{\lambda }\left( J,\mathbf{1}_{I^{\prime
}}\sigma \right) \mathbf{1}_{J}\left( x\right) \right\vert ^{2}\right) ^{%
\frac{1}{2}}\right\Vert _{L^{p}\left( \omega \right) } \\
&\lesssim &2^{-\eta s}\left\Vert \left( \sum_{J\in \mathcal{D}}\sum_{I\succ
J\ \ell \left( J\right) =2^{-s}\ell \left( I\right) }\left\vert \frac{1}{%
\sqrt{\left\vert I^{\prime }\right\vert _{\sigma }}}\left\vert \widehat{f}%
\left( I\right) \right\vert \mathrm{P}^{\lambda }\left( I_{J},\mathbf{1}%
_{I^{\prime }}\sigma \right) \mathbf{1}_{J}\left( x\right) \right\vert
^{2}\right) ^{\frac{1}{2}}\right\Vert _{L^{p}\left( \omega \right) } \\
&\lesssim &2^{-\eta s}\left\Vert \left( \sum_{I\in \mathcal{D}}\left\vert
\left\Vert \mathbb{E}_{I^{\prime }}^{\sigma }\bigtriangleup _{I}^{\sigma
}f\left( I\right) \right\Vert _{\infty }\frac{\left\vert I^{\prime
}\right\vert _{\sigma }}{\left\vert I^{\prime }\right\vert ^{1-\frac{\lambda 
}{n}}}\mathbf{1}_{I^{\prime }}\left( x\right) \right\vert ^{2}\right) ^{%
\frac{1}{2}}\right\Vert _{L^{p}\left( \omega \right) },
\end{eqnarray*}%
where again $\eta $ is the exponent from (\ref{e.Jsimeq}), which by the
quadratic offset Muckenhoupt condition, is dominated by 
\begin{equation*}
2^{-\eta s}A_{p}^{\lambda ,\ell ^{2},\limfunc{offset}}\left( \sigma ,\omega
\right) \left\Vert \left( \sum_{I\in \mathcal{D}}\left\vert \left\Vert 
\mathbb{E}_{I^{\prime }}^{\sigma }\bigtriangleup _{I}^{\sigma }f\left(
I\right) \right\Vert _{\infty }\mathbf{1}_{I^{\prime }}\left( x\right)
\right\vert ^{2}\right) ^{\frac{1}{2}}\right\Vert _{L^{p}\left( \sigma
\right) }\lesssim 2^{-\eta s}A_{p}^{\lambda ,\ell ^{2},\limfunc{offset}%
}\left( \sigma ,\omega \right) \left\Vert f\right\Vert _{L^{p}\left( \sigma
\right) }.
\end{equation*}%
Summing in $s\geq 0$ proves the required bound for the neighbour form,%
\begin{equation}
\left\vert \mathsf{B}_{\limfunc{neighbour}}\left( f,g\right) \right\vert
\lesssim A_{p}^{\lambda ,\ell ^{2},\limfunc{offset}}\left( \sigma ,\omega
\right) \left\Vert f\right\Vert _{L^{p}\left( \sigma \right) }\left\Vert
g\right\Vert _{L^{p^{\prime }}\left( \omega \right) }.  \label{neigh est}
\end{equation}

\subsection{Conclusion of the proof}

An examination of the schematic diagram at the beginning of the section on
organization of the proof, together with all the estimates proved so far,
completes the proof that 
\begin{equation*}
\left\vert \left\langle T_{\sigma }^{\lambda }f,g\right\rangle _{\omega
}\right\vert \lesssim \left[ \Gamma _{T^{\lambda },p}^{\ell ^{2}}+\mathcal{%
HWBP}_{T^{\lambda },p}^{\ell ^{2},\rho }\left( \sigma ,\omega \right) \right]
\left\Vert f\right\Vert _{L^{p}\left( \sigma \right) }\left\Vert
g\right\Vert _{L^{p^{\prime }}\left( \omega \right) },
\end{equation*}%
where the constant $\Gamma _{T^{\lambda },p}^{\ell ^{2}}$ is the sum of the
scalar testing and quadratic Muckenhoupt offset conditions%
\begin{equation*}
\Gamma _{T^{\lambda },p}^{\ell ^{2}}\equiv \mathfrak{T}_{T^{\lambda
},p}\left( \sigma ,\omega \right) +\mathfrak{T}_{T^{\lambda ,\ast
},p^{\prime }}\left( \omega ,\sigma \right) +A_{p}^{\lambda ,\ell ^{2},%
\limfunc{offset}}\left( \sigma ,\omega \right) +A_{p^{\prime }}^{\lambda
,\ell ^{2},\limfunc{offset}}\left( \omega ,\sigma \right) .
\end{equation*}%
Now we invoke Lemma \ref{stronger} to obtain that for all $0<\varepsilon <1$%
, there is a constant $C_{\varepsilon }$ such that 
\begin{equation*}
\left\vert \left\langle T_{\sigma }^{\lambda }f,g\right\rangle _{\omega
}\right\vert \lesssim \left\{ C_{\varepsilon }\left[ \Gamma _{T^{\lambda
},p}^{\ell ^{2}}+\mathcal{WBP}_{T^{\lambda },p}^{\ell ^{2}}\left( \sigma
,\omega \right) \right] +\varepsilon \mathfrak{N}_{T^{\lambda },p}\left(
\sigma ,\omega \right) \right\} \left\Vert f\right\Vert _{L^{p}\left( \sigma
\right) }\left\Vert g\right\Vert _{L^{p^{\prime }}\left( \omega \right) },
\end{equation*}%
from which we conclude that%
\begin{equation*}
\mathfrak{N}_{T^{\lambda },p}\left( \sigma ,\omega \right) \lesssim \left\{
C_{\varepsilon }\left[ \Gamma _{T^{\lambda },p}^{\ell ^{2}}+\mathcal{WBP}%
_{T^{\lambda },p}^{\ell ^{2}}\left( \sigma ,\omega \right) \right]
+\varepsilon \mathfrak{N}_{T^{\lambda },p}\left( \sigma ,\omega \right)
\right\} \left\Vert f\right\Vert _{L^{p}\left( \sigma \right) }\left\Vert
g\right\Vert _{L^{p^{\prime }}\left( \omega \right) }.
\end{equation*}

At this point, a standard argument using the definition of the two weight
norm inequality (\ref{two weight'}), for which see e.g. \cite[Section 6]%
{AlSaUr}, shows that for any smooth truncation of $T^{\lambda }$, we can
absorb the term $\varepsilon \mathfrak{N}_{T^{\lambda },p}\left( \sigma
,\omega \right) \left\Vert f\right\Vert _{L^{p}\left( \sigma \right)
}\left\Vert g\right\Vert _{L^{p^{\prime }}\left( \omega \right) }$ into the
left hand side and obtain (\ref{main inequ}),%
\begin{equation*}
\mathfrak{N}_{T^{\lambda },p}\left( \sigma ,\omega \right) \lesssim \left[
\Gamma _{T^{\lambda },p}^{\ell ^{2}}+\mathcal{WBP}_{T^{\lambda },p}^{\ell
^{2}}\left( \sigma ,\omega \right) \right] \left\Vert f\right\Vert
_{L^{p}\left( \sigma \right) }\left\Vert g\right\Vert _{L^{p^{\prime
}}\left( \omega \right) }.
\end{equation*}%
This completes the proof of Theorem \ref{main}.

\section{Appendix}

\subsection{A counterexample}

Regarding the quadratic Muckenhoupt condition in the case $p=2$, we clearly
we have 
\begin{equation*}
A_{2}^{\lambda ,\ell ^{2}}\left( \sigma ,\omega \right) +A_{2}^{\lambda
,\ell ^{2}}\left( \omega ,\sigma \right) \leq A_{2}^{\lambda }\left( \sigma
,\omega \right) ,
\end{equation*}%
for any pair of locally finite positive Borel measures. However, this fails
when $1<p<\infty $, $\lambda =0$ and $p\neq 2$ as we now show.

Let $1<p<\infty $, $0<\alpha \leq 1$ and define%
\begin{equation*}
f\left( x\right) \equiv \frac{1}{x\left( \ln \frac{1}{x}\right) ^{1+\alpha }}%
\mathbf{1}_{\left( 0,\frac{1}{2}\right) }\left( x\right) ,
\end{equation*}%
and note that%
\begin{equation*}
Mf\left( x\right) \mathbf{1}_{\left( 0,\frac{1}{2}\right) }\left( x\right)
\approx \frac{1}{x\left( \ln \frac{1}{x}\right) ^{\alpha }}\mathbf{1}%
_{\left( 0,\frac{1}{2}\right) }\left( x\right) .
\end{equation*}%
Then define%
\begin{eqnarray*}
v\left( x\right) &\equiv &f\left( x\right) ^{1-p}dx=\left[ x\left( \ln \frac{%
1}{x}\right) ^{1+\alpha }\right] ^{p-1}\mathbf{1}_{\left( 0,\frac{1}{2}%
\right) }\left( x\right) dx, \\
w\left( x\right) &\equiv &Mf\left( x\right) ^{1-p}\approx \left[ x\left( \ln 
\frac{1}{x}\right) ^{\alpha }\right] ^{p-1}\mathbf{1}_{\left( 0,\frac{1}{2}%
\right) }\left( x\right) dx,
\end{eqnarray*}%
so that%
\begin{eqnarray*}
\int_{0}^{\frac{1}{2}}\left\vert f\left( x\right) \right\vert ^{p}v\left(
x\right) dx &=&\int_{0}^{\frac{1}{2}}f\left( x\right) dx=\int_{0}^{\frac{1}{2%
}}\frac{1}{x\left( \ln \frac{1}{x}\right) ^{1+\alpha }}<\infty , \\
\int_{0}^{\frac{1}{2}}\left\vert Mf\left( x\right) \right\vert ^{p}w\left(
x\right) dx &=&\int_{0}^{\frac{1}{2}}Mf\left( x\right) dx\approx \int_{0}^{%
\frac{1}{2}}\frac{1}{x\left( \ln \frac{1}{x}\right) ^{\alpha }}=\infty .
\end{eqnarray*}%
On the other hand, using $\left( p-1\right) \left( 1-p^{\prime }\right) =-1$
we have for $0<r<\frac{1}{2}$, 
\begin{eqnarray*}
&&\left( \frac{1}{r}\int_{0}^{r}w\left( x\right) dx\right) \left( \frac{1}{r}%
\int_{0}^{r}v\left( x\right) ^{1-p^{\prime }}dx\right) ^{p-1} \\
&=&\left( \frac{1}{r}\int_{0}^{r}\left[ x\left( \ln \frac{1}{x}\right)
^{\alpha }\right] ^{p-1}dx\right) \left( \frac{1}{r}\int_{0}^{r}\frac{1}{%
x\left( \ln \frac{1}{x}\right) ^{1+\alpha }}dx\right) ^{p-1} \\
&\approx &\left( \frac{1}{r}r^{p}\left( \ln \frac{1}{r}\right) ^{\alpha
\left( p-1\right) }\right) \left( \frac{1}{r}\left( \ln \frac{1}{r}\right)
^{-\alpha }\right) ^{p-1}=1,
\end{eqnarray*}%
and it follows easily that 
\begin{equation*}
\sup_{I\subset \left( 0,\frac{1}{2}\right) }\left( \frac{1}{\left\vert
I\right\vert }\int_{I}w\left( x\right) dx\right) \left( \frac{1}{\left\vert
I\right\vert }\int_{I}v\left( x\right) ^{1-p^{\prime }}dx\right)
^{p-1}<\infty .
\end{equation*}

Thus if we set 
\begin{eqnarray*}
d\omega _{p,\alpha }\left( x\right) &\equiv &\left[ x\left( \ln \frac{1}{x}%
\right) ^{\alpha }\right] ^{p-1}\mathbf{1}_{\left( 0,\frac{1}{2}\right)
}\left( x\right) dx, \\
d\sigma _{p,\alpha }\left( x\right) &\equiv &\frac{1}{x\left( \ln \frac{1}{x}%
\right) ^{1+\alpha }}\mathbf{1}_{\left( 0,\frac{1}{2}\right) }\left(
x\right) dx,
\end{eqnarray*}%
then we have both \textbf{finiteness} of the Muckenhoupt constant $A_{p}^{%
\func{local}}\left( \sigma ,\omega \right) $ localized to $\left( 0,\frac{1}{%
2}\right) $, and \textbf{failure} of the norm inequality%
\begin{equation*}
\int_{\mathbb{R}}\left\vert M\left( f\sigma \right) \left( x\right)
\right\vert ^{p}d\omega \left( x\right) \lesssim \int_{\mathbb{R}}\left\vert
f\left( x\right) \right\vert ^{p}d\sigma \left( x\right) .
\end{equation*}%
Now we investigate the local \emph{quadratic} Muckenhoupt constant 
\begin{equation*}
A_{p}^{\ell ^{2},\func{local}}\left( \sigma ,\omega \right) +A_{p^{\prime
}}^{\ell ^{2},\func{local}}\left( \omega ,\sigma \right)
\end{equation*}%
when $\lambda =0$, i.e. where 
\begin{eqnarray*}
\left\Vert \left( \sum_{k=1}^{\infty }\left\vert a_{k}\frac{\left\vert
I_{k}\right\vert _{\sigma }}{\left\vert I_{k}\right\vert }\right\vert ^{2}%
\mathbf{1}_{I_{k}}\right) ^{\frac{1}{2}}\right\Vert _{L^{p}\left( \omega
\right) } &\leq &A_{p}^{\ell ^{2},\func{local}}\left( \sigma ,\omega \right)
\left\Vert \left( \sum_{k=1}^{\infty }\left\vert a_{k}\right\vert ^{2}%
\mathbf{1}_{I_{k}}\right) ^{\frac{1}{2}}\right\Vert _{L^{p}\left( \sigma
\right) }, \\
\left\Vert \left( \sum_{k=1}^{\infty }\left\vert a_{k}\frac{\left\vert
I_{k}\right\vert _{\omega }}{\left\vert I_{k}\right\vert }\right\vert ^{2}%
\mathbf{1}_{I_{k}}\right) ^{\frac{1}{2}}\right\Vert _{L^{p^{\prime }}\left(
\sigma \right) } &\leq &A_{p^{\prime }}^{\ell ^{2},\func{local}}\left(
\omega ,\sigma \right) \left\Vert \left( \sum_{k=1}^{\infty }\left\vert
a_{k}\right\vert ^{2}\mathbf{1}_{I_{k}}\right) ^{\frac{1}{2}}\right\Vert
_{L^{p^{\prime }}\left( \omega \right) },
\end{eqnarray*}%
for all sequences $\left\{ I_{i}\right\} _{i=1}^{\infty }$ of intervals in $%
I_{i}$ and all sequences $\left\{ a_{i}\right\} _{i=1}^{\infty }$ of
numbers. We have%
\begin{eqnarray*}
\left\vert \left[ 0,r\right] \right\vert _{\sigma } &=&\int_{0}^{r}\frac{1}{%
x\left( \ln \frac{1}{x}\right) ^{1+\alpha }}dx\approx \frac{1}{\left( \ln 
\frac{1}{r}\right) ^{\alpha }}, \\
\left\vert \left[ 0,r\right] \right\vert _{\omega } &=&\int_{0}^{r}\left[
x\left( \ln \frac{1}{x}\right) ^{\alpha }\right] ^{p-1}dx\approx r^{p}\left(
\ln \frac{1}{r}\right) ^{\alpha \left( p-1\right) }.
\end{eqnarray*}%
Thus if we take $I_{k}=\left( 0,2^{-k}\right) $, the inequality becomes%
\begin{equation*}
\left\Vert \left( \sum_{k=1}^{\infty }\left\vert a_{k}2^{k}\frac{1}{%
k^{\alpha }}\right\vert ^{2}\mathbf{1}_{\left( 0,2^{-k}\right) }\right) ^{%
\frac{1}{2}}\right\Vert _{L^{p}\left( \omega \right) }\leq A_{p}^{\ell ^{2},%
\func{local}}\left( \sigma ,\omega \right) \left\Vert \left(
\sum_{k=1}^{\infty }\left\vert a_{k}\right\vert ^{2}\mathbf{1}_{\left(
0,2^{-k}\right) }\right) ^{\frac{1}{2}}\right\Vert _{L^{p}\left( \sigma
\right) }.
\end{equation*}%
Now the $p^{th}$ power of the right hand side is%
\begin{eqnarray*}
&&\int_{0}^{\frac{1}{2}}\left( \sum_{k=1}^{\infty }\left\vert
a_{k}\right\vert ^{2}\mathbf{1}_{\left( 0,2^{-k}\right) }\left( x\right)
\right) ^{\frac{p}{2}}\frac{1}{x\left( \ln \frac{1}{x}\right) ^{1+\alpha }}%
dx=\sum_{k=1}^{\infty }\int_{2^{-k-1}}^{2^{-k}}\left( \sum_{\ell
=1}^{k}\left\vert a_{\ell }\right\vert ^{2}\right) ^{\frac{p}{2}}\frac{1}{%
x\left( \ln \frac{1}{x}\right) ^{1+\alpha }}dx \\
&\approx &\sum_{k=1}^{\infty }\left( \sum_{\ell =1}^{k}\left\vert a_{\ell
}\right\vert ^{2}\right) ^{\frac{p}{2}}\left( \frac{1}{k^{\alpha }}-\frac{1}{%
\left( k+1\right) ^{\alpha }}\right) \approx \sum_{k=1}^{\infty }\left(
\sum_{\ell =1}^{k}\left\vert a_{\ell }\right\vert ^{2}\right) ^{\frac{p}{2}}%
\frac{1}{k^{1+\alpha }},
\end{eqnarray*}%
and the $p^{th}$ power of the left hand side is%
\begin{eqnarray*}
&&\int_{0}^{\frac{1}{2}}\left( \sum_{k=1}^{\infty }\left\vert a_{k}2^{k}%
\frac{1}{k^{\alpha }}\right\vert ^{2}\mathbf{1}_{\left( 0,2^{-k}\right)
}\left( x\right) \right) ^{\frac{p}{2}}\left[ x\left( \ln \frac{1}{x}\right)
^{\alpha }\right] ^{p-1}dx \\
&=&\sum_{k=1}^{\infty }\int_{2^{-k-1}}^{2^{-k}}\left( \sum_{\ell
=1}^{k}\left\vert a_{\ell }2^{\ell }\frac{1}{\ell ^{\alpha }}\right\vert
^{2}\right) ^{\frac{p}{2}}\left[ x\left( \ln \frac{1}{x}\right) ^{\alpha }%
\right] ^{p-1}dx\approx \sum_{k=1}^{\infty }\left( \sum_{\ell
=1}^{k}\left\vert a_{\ell }2^{\ell }\frac{1}{\ell ^{\alpha }}\right\vert
^{2}\right) ^{\frac{p}{2}}2^{-kp}k^{\alpha \left( p-1\right) }.
\end{eqnarray*}%
Thus the right hand side will be finite if 
\begin{equation*}
a_{\ell }=\ell ^{\eta },\ \ \ \ \ \text{where }2\eta +1=\left( \alpha
-\varepsilon \right) \frac{2}{p}>0,
\end{equation*}%
since then 
\begin{equation*}
\sum_{\ell =1}^{k}\left\vert a_{\ell }\right\vert ^{2}=\sum_{\ell
=1}^{k}\ell ^{2\eta }\approx k^{2\eta +1}=k^{\left( \alpha -\varepsilon
\right) \frac{2}{p}}\text{ and hence }\left( \sum_{\ell =1}^{k}\left\vert
a_{\ell }\right\vert ^{2}\right) ^{\frac{p}{2}}=\frac{k^{1+\alpha }}{%
k^{1+\varepsilon }},
\end{equation*}%
and so 
\begin{equation*}
\sum_{k=1}^{\infty }\left( \sum_{\ell =1}^{k}\left\vert a_{\ell }\right\vert
^{2}\right) ^{\frac{p}{2}}\frac{1}{k^{1+\alpha }}=\sum_{k=1}^{\infty }\frac{1%
}{k^{1+\varepsilon }}<\infty .
\end{equation*}%
On the other hand, with this choice of $a_{\ell }$, the $p^{th}$ power of
the left hand side is%
\begin{eqnarray*}
&&\sum_{k=1}^{\infty }\left( \sum_{\ell =1}^{k}\left\vert a_{\ell }2^{\ell }%
\frac{1}{\ell ^{\alpha }}\right\vert ^{2}\right) ^{\frac{p}{2}%
}2^{-kp}k^{\alpha \left( p-1\right) }=\sum_{k=1}^{\infty }\left( \sum_{\ell
=1}^{k}\left\vert 2^{\ell }\ell ^{\eta -\alpha }\right\vert ^{2}\right) ^{%
\frac{p}{2}}2^{-kp}k^{\alpha \left( p-1\right) } \\
&\approx &\sum_{k=1}^{\infty }\left( \left\vert 2^{k}k^{\eta -\alpha
}\right\vert ^{2}\right) ^{\frac{p}{2}}2^{-kp}k^{\alpha \left( p-1\right)
}=\sum_{k=1}^{\infty }2^{kp}k^{\left( \eta -\alpha \right)
p}2^{-kp}k^{\alpha \left( p-1\right) }=\sum_{k=1}^{\infty }k^{\eta p-\alpha
p}k^{\alpha p-\alpha }=\sum_{k=1}^{\infty }k^{\eta p-\alpha },
\end{eqnarray*}%
which will be infinite if $\eta p-\alpha >-1$, and since $2\eta +1=\left(
\alpha -\varepsilon \right) \frac{2}{p}$, this will be the case provided%
\begin{eqnarray*}
-1 &<&\eta p-\alpha =\left[ \frac{\left( \alpha -\varepsilon \right) \frac{2%
}{p}-1}{2}\right] p-\alpha =\left( \alpha -\varepsilon \right) -\frac{p}{2}%
-\alpha =-\varepsilon -\frac{p}{2}, \\
\text{i.e. }0 &<&\varepsilon <\frac{2-p}{2}.
\end{eqnarray*}%
Thus we have a counterexample to the implication $A_{p}^{\func{local}}\left(
\sigma ,\omega \right) \Longrightarrow A_{p}^{\ell ^{2},\func{local}}\left(
\sigma ,\omega \right) +A_{p^{\prime }}^{\ell ^{2},\func{local}}\left(
\omega ,\sigma \right) $ when $1<p<2$, provided we choose $\left( \sigma
,\omega \right) =\left( \sigma _{p,\alpha },\omega _{p,\alpha }\right) $
with $0<\alpha \leq 1$.

\begin{proposition}
Let $p\in \left( 1,\infty \right) \setminus \left\{ 2\right\} $. There is a
weight pair $\left( \sigma ,\omega \right) $ such that%
\begin{eqnarray*}
A_{p}^{\func{local}}\left( \sigma ,\omega \right) &<&\infty , \\
A_{p}^{\ell ^{2},\func{local}}\left( \sigma ,\omega \right) +A_{p^{\prime
}}^{\ell ^{2},\func{local}}\left( \omega ,\sigma \right) &=&\infty .
\end{eqnarray*}
\end{proposition}

\begin{proof}
Let $\left( \sigma _{p,\alpha },\omega _{p,\alpha }\right) $ be the weight
pair constructed above. If $1<p<2$, we can take $\left( \sigma ,\omega
\right) =\left( \sigma _{p,1},\omega _{p,1}\right) $. If $2<p<\infty $, then 
$1<p^{\prime }<2$ and we can take $\left( \sigma ,\omega \right) =\left(
\omega _{p^{\prime },1},\sigma _{p^{\prime },1}\right) $.
\end{proof}

\begin{remark}
If we take $0<\alpha \leq 1$, then the two weight norm inequality for the
maximal function fails with weights $\sigma _{p^{\prime },\alpha }$ and $%
\omega _{p^{\prime },\alpha }$.
\end{remark}

\end{document}